\RequirePackage{fix-cm}
\documentclass[smallextended]{svjour3}       
\smartqed  
\usepackage{graphicx}
\usepackage[labelfont = bf]{caption}
\usepackage{xcolor}
\usepackage{tikz}
\usepackage[caption=false]{subfig}
\usepackage{amsmath,amssymb,mathtools}
\usepackage[bbgreekl]{mathbbol} 
\usepackage{stmaryrd}
\usepackage{hyperref}
\usepackage{enumitem}
\usepackage{tabularx,booktabs,multirow}  
\usepackage{float}
\usepackage{microtype}

\usetikzlibrary{patterns}
\usetikzlibrary{fadings}
\usetikzlibrary{shadows.blur}
\usetikzlibrary{shapes}
\definecolor{denim}{rgb}{0.08, 0.38, 0.74}
\definecolor{byzantium}{rgb}{0.44, 0.16, 0.39} 
\definecolor{shamrockgreen}{rgb}{0.0, 0.62, 0.38} 
\newcommand{\vect}[1]{\boldsymbol{#1}}
\newcommand\R{\mathbb{R}}
\newcommand\Rd{\mathbb{R}^d}
\newcommand\Rdd{\mathbb{R}^{d\times d}}
\newcommand\Rext{\mathbb{R}\cup\{+\infty\}}
\newcommand\Rextd{\mathbb{R}\cup\{-\infty\}}

\newcommand\tria{\mathcal{T}_h}
\newcommand\facesi{\mathcal{F}^i_h}
\newcommand\facesD{\mathcal{F}^D_h}

\newcommand\bb{\vect{b}}
\newcommand\br{\vect{r}}

\newcommand\bq{\vect{q}}
\newcommand\bv{\vect{v}}
\newcommand\bw{\vect{w}}
\newcommand\bu{\vect{u}}
\newcommand\bn{\vect{n}}
\newcommand\bg{\vect{g}}
\newcommand\BA{\tens{A}}
\newcommand\BB{\tens{B}}
\newcommand\BF{\tens{F}}
\newcommand\BT{\tens{T}}

\newcommand\BI{\tens{I}}
\newcommand\BL{\tens{L}}
\newcommand\BH{\tens{H}}

\newcommand\brho{\vect{\varrho}}

\newcommand\bsigma{\vect{\sigma}}
\newcommand\btau{\vect{\tau}}
\newcommand\bff{\vect{f}}
\newcommand\bzero{\vect{0}}
\newcommand\load{\bff^*}
\newcommand\uDirichlet{\smash{\widehat{\bu}_D}}
\newcommand\uhDirichlet{\smash{\widehat{\bu}^h_D}}

\newcommand\bbC{\mathbb{C}}

\newcommand\CR{V^h}
\newcommand\CRDirichlet{V^h_D}
\newcommand\RT{\Sigma^h}
\newcommand\RTNeumann{\Sigma^h_N}
\newcommand{\Pbroken}[1]{\mathbb{P}^{#1}(\mathcal{T}_h)}

\newcommand{\PbrokenM}[1]{(\mathbb{P}^{#1}(\mathcal{T}_h))^{d\times d}}
\newcommand\Hdiv{\mathbf{H}(\textup{div};\Omega)}
\newcommand\HdivN{\mathbf{H}_N(\textup{div};\Omega)}
\newcommand\HNeumann{\mathbf{H}^{-\frac{1}{2}}_{00}(\Gamma_N)}
\newcommand\HNeumannDual{{\mathbf{H}^{\frac{1}{2}}_{00}(\Gamma_N)}}
\newcommand\HoneDirichlet{\mathbf{H}^1_{D}(\Omega)}
\newcommand\HminusDirichlet{\mathbf{H}^{-1}_{D}(\Omega)}
\newcommand\LtwoM{\mathbb{L}^2(\Omega)}

\newcommand\Ltwosym{\mathbb{L}_{\mathrm{sym}}^2(\Omega)}

\newcommand\diver{\textup{div}\,}
\newcommand\dev{\textup{dev}\,}
\newcommand\sym{\textup{sym}\,}
\newcommand\tr{\textup{tr}\,}
\newcommand\osclocal{\mathrm{osc}^2_T\,}
\newcommand\skewM{\textup{skew}\,} 
\newcommand\fp{\,{:}\,}
\newcommand\characteristic[1]{I_{\{#1\}}}
\newcommand\symgrad[1]{\bbespilon(#1)}
\newcommand\symgraddisc[1]{\bbespilon_h(#1)} 

\newcommand\gap{\eta^2_{\mathrm{gap}}}
\newcommand\gapext{\overline{\eta}^2_{\mathrm{gap}}}
\newcommand\gaplocal{\eta^2_{\mathrm{gap},T}}
\newcommand\gapextlocal{\overline{\eta}^2_{\mathrm{gap},T}}
\newcommand\rhoprimal{\rho^2_{I}}
\newcommand\rhoprimalh{\rho^2_{I_h}}

\newcommand\rhodualext{\rho^2_{-\overline{D}}}
\newcommand\rhodualh{\rho^2_{-D_h}}

\newcommand\rhototext{\overline{\rho}^2_{\mathrm{tot}}}
\newcommand\rhototh{\rho^2_{\mathrm{tot},h}}

\providecommand{\jumptmp}[2]{#1\llbracket{#2}#1\rrbracket}
\providecommand{\jump}[1]{\jumptmp{}{#1}}
\providecommand{\normtmp}[2]{{#1\lVert #2 #1\rVert}}
\providecommand{\norm}[1]{\normtmp{}{#1}}

\spnewtheorem{algorithm}{Algorithm}{\bf}{\it}
\begin{document}

\title{\emph{A Priori} and \emph{A Posteriori} Error Identities\\ for Vectorial Problems via Convex Duality
  \thanks{P.A. G.-O. gratefully acknowledges support from the Charles University Research program no.\ PRIMUS/25/SCI/025 and UNCE/24/SCI/005.}
}

\titlerunning{Estimates via duality for vectorial problems}        

\author{P. A.\ Gazca-Orozco         \and
        A.\ Kaltenbach
}


\institute{P. A. Gazca-Orozco \at
  Faculty of Mathematics and Physics\\  Charles University\\ Sokolovská 83 \\ 186 75, Prague \\
              \email{gazca@karlin.mff.cuni.cz}           
           \and
           A.\ Kaltenbach \at
           Institute of Mathematics \\ Technical University of Berlin \\ Stra\ss e des 17.\ Juni 136\\ 10623 Berlin \\
           \email{kaltenbach@math.tu-berlin.de}
}

\date{Received: date / Accepted: date}

\maketitle

\begin{abstract}
  Convex duality has been leveraged in recent years to derive \emph{a posteriori} error estimates and identities for a wide range of non-linear and non-smooth scalar problems.
By employing remarkable compatibility properties of the Crouzeix--Raviart and Raviart--Thomas elements, optimal convergence of non-conforming discretisations and flux reconstruction formulas have also been established.
This paper aims to extend these results to the vectorial setting, focusing on the archetypical problems of incompressible Stokes and Navier--Lam\'e.
Moreover, unlike most previous results, we consider inhomogeneous mixed boundary conditions and loads in the topological dual of the energy space.
At the discrete level, we derive error identities and estimates that enable to prove quasi-optimal error estimates for a Crouzeix--Raviart discretisation with minimal regularity assumptions and no data {oscillation} terms.
	\keywords{Incompressible Stokes problem; Navier--Lam\'e problem; minimal regularity estimates; Prager–Synge-type identities; convex duality; discrete convex duality; mixed boundary conditions.}
\subclass{49M29\and 65N15\and  65N30\and  65N50\and 74F10\and 74B05}
\end{abstract}

\section{Introduction}\label{sec:intro}
  \hspace{5mm}The Prager--Synge identity (\textit{cf}.\ \cite{PraSyn47}) forms the cornerstone of many developments in \emph{a posteriori} error analysis.
  For the classical Poisson (minimisation) problem posed with full homogeneous Dirichlet boundary conditions on a bounded Lipschitz domain $\Omega \subseteq \mathbb{R}^d$, $d\in \mathbb{N}$, and $f\in L^2(\Omega)$, \textit{i.e.},
  \begin{equation}\label{eq:laplace_primal_intro}
    \text{Find }u\in H^1_0(\Omega) \quad\text{ s.t. }\quad
    I(u) = \min_{v\in H^1_0(\Omega)}{\big\{I(v)
    \coloneqq \tfrac{1}{2}\|\nabla v\|^2_\Omega - (f,v)_\Omega\big\}}\,,
  \end{equation}
  for every $v\in H^1_0(\Omega)$,
   $\br\in H(\textup{div};\Omega)$ with  
    $\diver \br = -f$ a.e.\ in $\Omega$,
  the \emph{Prager--Synge identity} reads
  \begin{equation}\label{eq:prager_synge_intro}
    \tfrac{1}{2}\|\nabla u - \nabla v\|^2_\Omega 
    + \tfrac{1}{2} \|\bq - \br\|^2_\Omega 
    = \tfrac{1}{2}\|\br-\nabla v\|^2_\Omega \,.
  \end{equation}
  Here, $(\cdot,\cdot)_\Omega$ and $\|\cdot\|_{\Omega}$ denote the $L^2(\Omega)$-inner product and -norm, respectively (the precise notation is introduced in the following section).
  The vector field $\bq\coloneqq \nabla u \in H(\textup{div};\Omega)$ is the flux  associated to the \emph{primal solution} $u\in H^1_0(\Omega)$, 
  and is actually the solution of the \emph{dual problem} to \eqref{eq:laplace_primal_intro}:
  \begin{equation}\label{eq:dual_laplace_intro}
   \text{Find } \left\{\begin{aligned}
        \bq\in H(\textup{div};\Omega)\text{ with } \\\diver \bq=-f\text{ a.e.\ in }\Omega 
    \end{aligned}\right.\quad\text{ s.t. }\quad
    D(\bq) = \max_{\substack{\br \in H(\textup{div};\Omega)\\\diver \br = -f}}{\big \{D(\br)
    \coloneqq 
    -\tfrac{1}{2}\|\br\|^2_\Omega\big\}}\,.
  \end{equation}
In practice, the function  $v\in H_0^1(\Omega)$ and the vector field $\br\in H(\textup{div};\Omega)$ with $\diver \br = -f$ a.e.\ in $\Omega$ arise from numerical approximations, and the right-hand-side of the Prager--Synge identity \eqref{eq:prager_synge_intro} provides a computable error quantity, which is localisable and leads to error control with explicit constants (for further details and historical remarks, we refer to  \cite{Rep08,BB.2019}).

The traditional derivation of the Prager--Synge identity \eqref{eq:prager_synge_intro} is based on a Pythagorean identity,
which is not straightforward to generalise to non-linear settings.
However, a simple computation reveals that for every $v\in H^1_0(\Omega)$ and $\br\in H(\textup{div};\Omega)$ with $\diver\br=-f$ a.e.\ in $\Omega$, one has that 
\begin{align}
  I(v) - I(u) &= \tfrac{1}{2}\|\nabla u - \nabla v\|^2_\Omega\,,\\
  D(\bq) - D(\br) &= \tfrac{1}{2}\|\bq - \br\|^2_\Omega\,.
\end{align}
In other words, the natural error measure for convex optimisation problems (the energy difference) is nothing other than the usual error.
Moreover, since the \emph{strong duality relation} $I(u)=D(\bq)$ holds for the Poisson problem \eqref{eq:prager_synge_intro} (\textit{cf}.\ \cite[Prop.\ 2.1, p.\ 80]{ET99}),
it is straightforward to see that for every $v\in H^1_0(\Omega)$ and $\br\in H(\textup{div};\Omega)$ with $\diver\br=-f$ a.e.\ in $\Omega$, there holds
\begin{equation}\label{eq:prager_synge_intro2}
I(v) - I(u) + D(\bq) - D(\br) 
= I(v) - D(\br) 
= \tfrac{1}{2}\|\br - \nabla v\|^2_\Omega\,.
\end{equation}
This formulation of the Prager--Synge identity \eqref{eq:prager_synge_intro}, written in terms of the so-called \emph{primal-dual gap estimator} $\eta^2_{\mathrm{gap}}(\br,v)\coloneqq I(v) - D(\br)$, has the advantage that it can be readily generalized to a more general class of energy functionals, as long as a strong duality relation holds.
Such an identity has been derived for a class of energy functionals of the form $I\colon W^{1,p}_0(\Omega)\to \mathbb{R}\cup \{+\infty\}$, $p\in [1,+\infty]$, for every $v\in W^{1,p}_0(\Omega)$ defined by
\begin{equation}
I(v) \coloneqq \int_\Omega{\phi(\cdot,\nabla v)\,\mathrm{d}x} + \int_\Omega{\psi(\cdot,v)\,\mathrm{d}x}\,,
\end{equation}
where $\phi\colon \Omega\times\mathbb{R}^d\to \mathbb{R}\cup \{+\infty\}$ and $\psi\colon \Omega\times \mathbb{R}\to \mathbb{R}\cup \{+\infty\}$ are convex normal integrands (\textit{cf}.~\cite{Rockafellar1968}) satisfying appropriate growth and coercivity conditions, but which could be non-linear and  non-smooth (\textit{cf}.\ \cite{BM20,Bar21,BK.2023,BK.2024} and \cite{RepinXanthis1997,Rep99,Han05,Rep08,Rep20C}, for earlier results focusing mostly on upper bounds of 
the primal error).

While the results from \cite{BM20,Bar21,BK.2023,BK.2024} are quite general,
only problems in scalar function spaces, where the load $f$  belongs to $L^{p'}(\Omega)$ (instead of the (topological) dual space $W^{-1,p'}(\Omega)$, $\frac{1}{p}+\frac{1}{p'}=1$) have been considered.
One goal of this paper is to extend the ideas from  \cite{BM20,Bar21,BK.2023,BK.2024} to \emph{problems in vectorial function spaces}, namely we consider the incompressible Stokes problem (\textit{cf}.~\cite{Stokes_2009}) and the Navier--Lam\'e problem (\textit{cf}.~\cite{Navier1821,Lame1833});
Moreover, we carry out the analysis for \emph{general loads} belonging to the dual of the energy space, 
and also for \emph{inhomogeneous mixed boundary conditions}.
Even though the corresponding \emph{a posteriori} error identities were previously derived by other means (under more strict assumptions) 
our analysis constitutes the first step towards the generalisation to non-linear and/or non-smooth vectorial problems.

At the discrete level, we consider a Crouzeix--Raviart discretisation for the primal solution and a Raviart--Thomas approximation for the flux~(or~stress) solution.
For instance, for the Poisson problem \eqref{eq:laplace_primal_intro}, the discrete problem consists in finding a function $u_h$ belonging to the Crouzeix--Raviart space $V^h_D$ with zero boundary conditions (the precise definition is given in {Section} \ref{sec:CR})
that minimises the discrete energy functional $I_h\colon V^h_D \to \R$, for every~${v_h\in V_D^h}$~defined~by
\begin{equation}\label{intro:poisson_primal_discrete}
I_h(v_h) \coloneqq 
\tfrac{1}{2}\|\nabla_h v_h\|^2_\Omega - (f,\Pi_h v_h)_\Omega\,,
\end{equation}
where $\nabla_h$ denotes the element-wise  gradient operator and $\Pi_h$ represents the orthogonal projection of $V^h_D$ onto element-wise constant functions.
The presence of the projector $\Pi_h$ is important, as it ensures that a strong duality~relation~holds with the dual energy functional $D_h\colon \Sigma^h \to \R\cup\{-\infty\}$, for every $\br_h\in \Sigma_h$ defined by
\begin{equation}\label{intro:poisson_dual_discrete}
D_h(\br_h )\coloneqq 
- \tfrac{1}{2}\|\Pi_h \br_h\|^2_\Omega - I^\Omega_{\{-\Pi_h f\}}(\diver\br_h)\,,
\end{equation}
where $\Sigma^h$ the lowest-order Raviart--Thomas space (the precise definition is given in {Section} \ref{sec:RT}).

More precisely, this means that the discrete primal solution $u_h\hspace{-0.1em}\in \hspace{-0.1em} V^h_D$, \textit{i.e.}, the {minimiser} of \eqref{intro:poisson_primal_discrete},
and the discrete dual solution $\bq_h\in \Sigma_h$,  \textit{i.e.}, the maximiser of \eqref{intro:poisson_dual_discrete},
 satisfy $I_h(u_h)=D_h(\bq_h)$.
As recently noted in \cite{BarGudKal24}, for the error analysis of a discretisation of the scalar Signorini problem, 
this allows one to derive a discrete analogue of the Prager--Synge identity \eqref{eq:prager_synge_intro}, which, for every $v_h\in V^h_D$ and $\br_h\in \Sigma^h$ with $\diver\br_h = - \Pi_h f$ a.e.\ in $\Omega$, reads
\begin{equation}\label{intro:poisson_prager_synge_discrete}
    \tfrac{1}{2}\|\nabla_h u_h - \nabla_h v_h\|^2_\Omega 
    + \tfrac{1}{2} \|\Pi_h\bq_h - \Pi_h\br_h\|^2_\Omega 
    = \tfrac{1}{2}\|\Pi_h\br_h-\nabla_h v_h\|^2_\Omega \,.
\end{equation}
The \emph{a priori} identity \eqref{intro:poisson_prager_synge_discrete} characterises the distance between arbitrary discrete {admissible} functions and discrete solutions,
and could be employed in stopping criteria for inexact linear solvers.
In this paper, we exploit the \emph{a priori} identity \eqref{intro:poisson_prager_synge_discrete} to prove that the aforementioned discretisation is quasi-optimal,
meaning that the discretisation and best-approximation errors are equivalent, \textit{i.e.},
\begin{equation}\label{intro:poisson_quasioptimal}
  \begin{aligned}
  \|\nabla_h u_h - \nabla u \|^2_\Omega 
  + \|\Pi_h\bq_h - \bq\|^2_\Omega 
  &\sim 
  \inf_{v_h \in V^h_D}{\big\{\|\nabla_h v_h - \nabla u\|^2_\Omega\big\}} \\
  &\quad+ \inf_{\substack{\br_h \in \Sigma^h\\ \diver \br_h = -\Pi_h f}}
{\big\{\|\Pi_h\br_h - \bq\|^2_\Omega\big\}}\,,
\end{aligned}
\end{equation}
where the implicit constant involved in the equivalence `$\sim$' does not depend on any parameters.
We stress that the quasi-optimality result \eqref{intro:poisson_quasioptimal} has minimal regularity assumptions ($u\in H^1_0(\Omega)$ only) and later is shown to hold also for more general loads and mixed boundary conditions. Further, 
note the absence of data oscillation terms in \eqref{intro:poisson_quasioptimal}, unlike in the  classical result from \cite{Gud10A}. This observation seems to be new even for the Poisson problem \eqref{eq:laplace_primal_intro};
to the best of the author's {knowledge}, the only other genuine minimal-regularity quasi-optimality result for a Crouzeix--Raviart discreti\-sation of the Poisson problem \eqref{eq:laplace_primal_intro} is the result from \cite{VZ.2019.II}, which requires the implementation of a so-called smoothing operator and involves a constant that depends on shape-regularity and is not explicit (for more details, we refer to Remark \ref{rmk:quasi-optimal_laplace}).

\subsection{The incompressible Stokes problem}

The incompressible Stokes (minimisation) problem (\textit{cf}.~\cite{Stokes_2009}), for now considered with homogeneous Dirichlet boundary conditions,
seeks a \emph{velocity vector field} $\bu\in \BH^1_0(\Omega) \coloneqq  (H^1_0(\Omega))^d$ that minimises the  energy functional $I\colon \BH^1_0(\Omega) \to \Rext $, for every $\bv\in \BH^1_0(\Omega)$ defined by
\begin{equation}\label{intro:stokes_primal} 
I(\bv)\coloneqq 
\tfrac{\nu}{2}\|\nabla \bv\|^2_\Omega + 
I^\Omega_{\{0\}}(\diver \bv)
- \langle \bff^*, \bv\rangle_\Omega\,, 
\end{equation}
where $\nu>0$ is the \emph{kinematic viscosity}, $\bff^*\in \BH^{-1}(\Omega)$ a given \emph{body force},  $\langle \cdot,\cdot\rangle_\Omega$ the duality {pairing} bet\-ween $\BH^{-1}(\Omega)$ and $\BH^1_0(\Omega)$, and the indicator functional $\smash{I_{\{0\}}^{\Omega}}\colon L^2(\Omega)\to \mathbb{R}\cup \{+\infty\}$, for every $\widehat{v}\in L^2(\Omega) $, is defined by
	\begin{align*}
		I_{\{0\}}^{\Omega}(\widehat{v})
		\coloneqq 
		\begin{cases}
			0&\text{ if }\widehat{v}=0\text{ a.e.\ in }\Omega\,,\\[-0.5mm]
			+\infty&\text{ else}\,.
		\end{cases}
	\end{align*} 
  This zero divergence constraint   represents the main difficulty in the incompressible Stokes problem \eqref{intro:stokes_primal}  that is not present in the Poisson problem \eqref{eq:laplace_primal_intro}.
  Making use of a standard representation formula for elements in $\BH^{-1}(\Omega)$, \textit{i.e.}, fixing 
  $\bff\in \BL^2(\Omega)\coloneqq (L^2(\Omega))^d$ and 
  $\BF\in \LtwoM\coloneqq (L^2(\Omega))^{d\times d}$ such that
  \begin{equation*}
\langle \bff^*,\bv\rangle_\Omega 
= (\bff,\bv)_\Omega 
+ (\BF,\nabla \bv)_\Omega
\quad \text{for all  }\bv\in \BH^1_0(\Omega)\,,
  \end{equation*}
  a dual 
  energy functional $D\colon \BF + \BH(\textup{div};\Omega) \to \Rextd$, for every $\btau \in \BF + \BH(\textup{div};\Omega) $, is given via 
  \begin{equation} \label{intro:stokes_dual} 
D(\btau ) \coloneqq 
- \tfrac{1}{2\nu}\|\dev\btau\|^2_\Omega 
- I^\Omega_{\{-\bff\}}(\diver(\btau-\BF))\,, 
  \end{equation}
  where $\dev\btau = \btau - \frac{1}{d}\mathrm{tr}(\btau)\BI\in \LtwoM$ is the deviatoric part of $\btau$
  and $\BH(\textup{div};\Omega)\coloneqq  (H(\textup{div};\Omega))^d$.
  Similar to the Prager--Synge identity \eqref{eq:laplace_primal_intro}, 
one can exploit the strong duality relation $I(\bu)=D(\BT)$ between the unique minimiser $\bu\in \BH^1_0(\Omega)$ of \eqref{intro:stokes_primal}  and the unique maximiser $\BT \in \BF + \BH(\textup{div};\Omega)$ of \eqref{intro:stokes_dual} (which represents the \emph{Cauchy stress}) to derive an \emph{a posteriori} error identity for arbitrary $\bv\in \BH^1_0(\Omega)$ with $\diver\bv=0$ a.e. in $\Omega$ and $\btau\in \BF+\BH(\textup{div};\Omega)$ with $ \diver(\btau - \BF) = -\bff$~a.e.~in~$\Omega$:
\begin{equation}\label{intro:stokes_prager_synge_discrete}
    \tfrac{\nu}{2}\|\nabla \bv - \nabla \bu\|^2_\Omega 
    + \tfrac{1}{2\nu}\|\dev\btau - \dev \BT\|^2_\Omega 
    = 
    \tfrac{\nu}{2}\|\nabla \bv - \tfrac{1}{\nu}\dev \btau\|^2_\Omega\,.
  \end{equation}
  
  In order to be able to use the identity \eqref{intro:stokes_prager_synge_discrete} to estimate the error \emph{a posteriori}, an admissible tensor field $\btau\in \BF+\BH(\textup{div};\Omega)$ with $ \diver(\btau - \BF) = -\bff$ a.e.\ in $\Omega$ needs to be constructed. To this end,
 classical stress equilibration procedures such as \cite{DA.2005,HSV.2012,AABR.2012,CHTV.2018} can be used. 
 Taking inspiration from the work~of~Marini~\cite{Mar85}, a reconstruction formula that provides the discrete dual solution for a given Crouzeix--Raviart approximation has been derived for several problems (\textit{cf}.~\cite{CL15,LLC18,LC20,CGS.2013,Bar21,BK.2023}).
  For the incompressible Stokes problem with $\BF=\mathbf{0}$ a.e.\ in $\Omega$, it takes the form (\textit{cf}.\ \cite{CGS.2013})
			\begin{align}\label{intro:stokes_marini}
				\smash{\BT_h= \nu\nabla_h\bu_h - p_h \BI
					- \tfrac{1}{d}\Pi_h\bff \otimes (\mathrm{id}_{\Rd} -\Pi_h \mathrm{id}_{\Rd})\quad\text{ a.e.\ in }\Omega}\,,
			\end{align}
  where $\bu_h\in V_D^h$ denotes the Crouzeix--Raviart velocity vector field  and $p_h\in \mathbb{P}^0(\mathcal{T}_h)$ is the {discrete} (element-wise constant) kinematic pressure field.
  In this paper,  we establish that formula \eqref{intro:stokes_marini} is still valid in a minimal regularity setting with more general loads and boundary {conditions} as in \cite{CGS.2013} (\textit{cf}. Proposition~\ref{prop:stokes_marini}).
  Regarding the vector field $\bv\in \BH^1_0(\Omega)$, the identity \eqref{intro:stokes_prager_synge_discrete} requires it to be exactly {divergence-free},
which is not usually the case for approximations, so that, 
in general, 
post-processing is required.
For a Crouzeix--Raviart velocity vector field $\bu_h\in V_D^h$, a suitable post-processing is described in \cite{VZ.2019}.
However, in the case that such a post-processing is not desired, 
we also provide a result that allows one to employ an arbitrary $\bv\in \BH^1_0(\Omega)$ (\textit{cf}. {Theorem} \ref{thm:prager_synge_identity_stokes2});
the price to pay is that the equality in \eqref{intro:poisson_prager_synge_discrete} becomes an equivalence, albeit with an explicit constant that depends only on the constant from the so-called dev-div inequality (see \eqref{eq:dev_div} below);
we note that estimates with a similar structure were obtained in \cite{Rep.2002,Rep.2005,Kim.2014} under more restrictive assumptions.

Similar to the Poisson problem \eqref{eq:laplace_primal_intro}, we establish quasi-optimality of a Crouzeix--Raviart discretisation of the incompressible  Stokes problem \eqref{intro:stokes_primal} with minimal regularity assumptions (\textit{cf}. Theorem \ref{thm:apriori_stokes}):
\begin{align*}
 & \tfrac{\nu}{2}\|\nabla_h \bu_h - \nabla \bu\|^2_\Omega + 
  \tfrac{1}{2\nu}\|\Pi_h \dev\BT_h - \dev \BT\|^2_\Omega 
  \\
                                                         &\sim 
    \inf_{\bv_h \in V^h_D}{\big\{\tfrac{\nu}{2}\|\nabla_h \bv_h - \nabla \bu\|^2_\Omega \big\}}
    + \hspace{-3ex}\inf_{\substack{\btau_h \in \Sigma^h\\ \diver (\br_h-\Pi_h \BF) = -\Pi_h f}}{\big\{
  \tfrac{1}{2\nu}\|\dev\Pi_h\btau_h - \dev\BT\|^2_\Omega\big\}}\,,
\end{align*}
where the implicit constant involved in the equivalence `$\sim$' does not depend on any parameters.
The only other minimal regularity quasi-optimality result for a Crouzeix--Raviart approximation of the incompressible~Stokes~problem~\eqref{intro:stokes_primal}, to the best of the author's knowledge, is {\cite[Thm. 4.2]{VZ.2019}},
which likewise requires the implementation of a smoothing operator and involves a constant that might depend on shape-regularity (\textit{cf}.\  Remark \ref{rmk:quasi-optimal_stokes}).

\subsection{The Navier--Lam\'e problem}
The Navier--Lamé (minimisation) problem (\textit{cf}.\ \cite{Navier1821,Lame1833}), for now considered with homogeneous Dirichlet boundary conditions,  seeks 
a \emph{displacement~vector~field} $\bu \in \BH^1_0(\Omega)$ that  minimises the  energy functional  $I\colon \BH^1_0(\Omega) \to \R $,~for~every~$\bv\in \BH^1_0(\Omega)$ defined by  
\begin{equation} \label{intro:navier_lame_primal} 
I(\bv)\coloneqq 
\tfrac{1}{2}(\bbC\symgrad{\bv}, \symgrad{\bv})_\Omega 
- \langle \bff^*, \bv\rangle_\Omega\,. 
\end{equation}
where \hspace{-0.1mm}$\bbC\colon \hspace{-0.175em}\Rdd\hspace{-0.175em}\to\hspace{-0.175em} \Rdd$ \hspace{-0.1mm}denotes \hspace{-0.1mm}the \hspace{-0.1mm}\emph{linear \hspace{-0.1mm}elasticity \hspace{-0.1mm}tensor} (see \hspace{-0.1mm}Section~\hspace{-0.1mm}\ref{subsec:elasticity_primal},~\hspace{-0.1mm}for~\hspace{-0.1mm}the precise definition),
$\symgrad{\bv}\hspace{-0.15em}\coloneqq\hspace{-0.15em} \frac{1}{2}(\nabla \bv + \nabla \bv^\top)$  the 
\emph{symmetric~gradient~\mbox{operator}},~and $\bff^*\in \BH^{-1}(\Omega)$ a given \emph{load}. {Similar} to the incompressible Stokes~problem~\eqref{intro:stokes_primal},  
a dual  energy functional ${D\colon \hspace{-0.175em}\BF \hspace{-0.175em}+ \hspace{-0.175em}\BH(\textup{div};\Omega)\hspace{-0.175em} \to \hspace{-0.175em}\R\hspace{-0.175em}\cup\hspace{-0.175em}\{-\infty\}}$, for every $\btau\in \BF + \BH(\textup{div};\Omega)$ defined by
  \begin{equation} 
D(\btau ) \coloneqq 
- \tfrac{1}{2}(\bbC^{-1}\btau,\btau)_\Omega
- I^\Omega_{\{-\bff\}}(\diver(\btau-\BF))
- I^{\Omega}_{\{ \bzero\}}(\skewM{\btau})\,, 
  \end{equation}
where the second indicator functional accounts for the fact that the stresses should be symmetric. 
This is one of the main difficulties of the Navier--Lam\'e problem,
because while an \emph{a posteriori} error identity similar to \eqref{intro:poisson_prager_synge_discrete} is available 
\cite{PraSyn47,Rep.2001,Rep08,LS.2023},
in practice, it is challenging to construct $\BH(\textup{div};\Omega)$-conforming symmetric approximations.
Naturally, if discretisations or equilibration procedures that enforce exact symmetry are employed, such as those described in \cite{NWW.2008,AR.2010,AR.2011,RdPE.2017,LS.2023},
 the error identity may be used.
In this paper, we instead employ in the analysis an extended dual energy functional $\overline{D}\colon \BF + \BH(\textup{div};\Omega) \to \Rextd$ obtained by simply dropping the symmetry requirement and derive an \emph{a posteriori} equivalence, which for every $\bv\in \BH^1_0(\Omega)$ and $\btau \in \BF +$ $ \BH(\textup{div};\Omega)$ with $\diver(\btau-\BF)=-\bff$ a.e.\ in $\Omega$ reads:
\begin{equation}
\|\smash{\smash{\bbC^{\smash{\frac{1}{2}}}}}\symgrad{\bv-\bu}\|_{\Omega}^2
+
\|\smash{\smash{\bbC^{-\smash{\frac{1}{2}}}}}(\btau-\bsigma)\|_{\Omega}^2
\sim
\|\smash{\smash{\bbC^{\smash{\frac{1}{2}}}}}(\symgrad{\bv} - \smash{\bbC^{-1}}\btau)\|_{\Omega}^2\,,
\end{equation}
where $\|\smash{\smash{\bbC^{\smash{\frac{1}{2}}}}}(\cdot)\|_{\Omega}^2\coloneqq (\bbC(\cdot),\cdot)_\Omega$, $\|\smash{\smash{\bbC^{-\smash{\frac{1}{2}}}}}(\cdot)\|_{\Omega}^2\coloneqq (\bbC^{-1}(\cdot),\cdot)_\Omega$, $\bsigma\in \BF + \BH(\textup{div};\Omega)$ with $\diver(\bsigma-\BF)=-\bff$ a.e.\ in $\Omega$ and $\bsigma^\top= \bsigma$ a.e.\ in  $\Omega$ denotes the dual solution (\textit{i.e.}, the elastic stress tensor field), and the implicit constant in the equivalence `$\sim$' depends only on local Korn constants.
While estimates of this type have been derived before in the context of the Navier--Lam\'e problem  (\textit{cf}.\ \cite{Kim.2011,Kim.2012,DM.2013,BKMS.2021}) (typically under more restrictive assumptions on the data and boundary conditions),
the argument based on convex duality that we present here is much more readily applicable to non-linear and/or non-smooth generalisations of the Navier--Lam\'e problem (and the incompressible Stokes problem).

Similarly to the Poisson problem \eqref{eq:laplace_primal_intro} and the incompressible Stokes problem \eqref{intro:stokes_primal}, applying these ideas at the discrete level, we are able to derive a minimal regularity \emph{a {priori}} {error} {estimate}.
To remedy the well-known fact that Crouzeix--Raviart vector fields fail to satisfy Korn's inequality (\textit{cf}.~\cite{Falk74}),
we consider a discrete formulation that includes a jump stabilisation term $s_h(\cdot,\cdot)$.
The presence of a stabilisation term $s_h(\cdot,\cdot)$ has the consequence that a corresponding mixed formulation with a saddle-point structure for a discrete stress belonging to a Raviart--Thomas space cannot be defined in the usual way,
so, at the discrete level, we do not derive an error \emph{identity}. However,
by working with the discrete energy norm 
\begin{equation}
  \norm{\bv_h}^2_h \coloneqq \|
  \bbC^{\smash{\frac{1}{2}}}\symgraddisc{\bv_h}\|_\Omega ^2
  + s_h(\bv_h,\bv_h),
  \quad \bv_h\in V^h\,,
\end{equation}
where $\symgraddisc{\bv_h}\coloneqq \frac{1}{2}( \nabla_h\bv_h+ \nabla_h\bv_h^\top)$ denotes the element-wise  symmetric gradient operator,
we are still able to prove the following minimal-regularity error bound (see Theorem \ref{thm:apriori_elasticity}):
 \begin{align*}
\|\bu_h - \bu\|^2_h
\lesssim
    \inf_{\bv_h \in V^h_D}{\big\{\|\bv_h -  \bu\|^2_h\big\}}
  + \inf_{\substack{\btau_h \in \Sigma^h\\ \diver (\br_h-\Pi_h \BF) = -\Pi_h f}}{\big\{
  \|\bbC^{\smash{-\frac{1}{2}}}(\bsigma - \btau_h)\|^2_\Omega\big\}}\,,
\end{align*}
where $\Sigma^h$ denotes the lowest-order (vectorial) Raviart--Thomas space without symmetry {condition} and 
the implicit constant involved in the estimate  `$\lesssim$' depends solely on the discrete Korn constant satisfied by the norm $\|\cdot\|_h$. 
Although an estimate that only involves the displacement best-approximation error would be desirable,
we emphasize that, unlike more classical estimates for Crouzeix--Raviart approximations (such as \cite[Thm. 3.1]{HL.2003b}),
our result involves minimal regularity requirements, leading to optimal fractional convergence rates in case the required additional regularity is at hand.
Additionally, we provide as well a discrete stress reconstruction formula:
\begin{align}
  \bsigma_h^*\coloneqq \bbC\symgraddisc{\bu_h} +
\nabla_h \br_h 
- \tfrac{2}{d+1}\sym (\bff_h \otimes (\mathrm{id}_{\R^d} - \Pi_h \mathrm{id}_{\R^d}))\quad \text{ a.e.\ in }\Omega\,,
\end{align}
where $\br_h\in V_D^h$ denotes a Crouzeix--Raviart lifting of the jump stabilisation form ${s_h(\bu_h,\cdot)\in ( V_D^h)^*}$.
This provides an element-wise affine $\BH(\textup{div};\Omega)$-conforming stress approximation with a computational cost comparable to the discrete primal problem,
whose lack of symmetry can be quantified in terms of the jumps of $\bu_h\in V_D^h$ (for more details,
we refer to Remark \ref{rem:marini_elasticity} and Remark~\ref{rem:marini_elasticity2}).

\textit{This paper is organised as follows:} Section \ref{sec:preliminaries} introduces the functional-analytic setting, notation, and discretisation ingredients. Section \ref{sec:stokes} develops the primal-dual formulation of the incompressible Stokes problem, derives continuous and discrete error identities via convex duality, and proves minimal-regularity quasi-optimality of the Crouzeix--Raviart discretisation. Section \ref{sec:navier-lame} presents the analogous theory for the Navier--Lamé problem, including the dual formulation, \emph{a posteriori} equivalences, and a discrete stress reconstruction. Section \ref{sec:experiments} reports numerical experiments illustrating and supporting the theoretical findings.

	\section{Preliminaries}\label{sec:preliminaries}
Throughout the paper,  we denote by ${\Omega \subseteq \mathbb{R}^d}$, ${d \ge 2}$, a bounded polyhedral Lipschitz domain, whose (topological) boundary $\partial\Omega$ is  divided into a relatively open Dirichlet part $\Gamma_D$ and a relatively open Neumann part $\Gamma_N$, \textit{i.e.}, we have that $ \overline{\Gamma}_D\cup\overline{\Gamma}_N=\partial\Omega $ and $ \Gamma_D\cap\Gamma_N=\emptyset$.
	For simplicity, we assume that $\Gamma_D\not=\emptyset$. 
	Then, we make use of the following standard notation for function spaces:
	\begin{align*} 
		\mathbf{L}^2(\Omega)&\coloneqq (L^2(\Omega))^d\,,\quad\!\mathbb{L}^2(\Omega)\coloneqq (L^2(\Omega))^{d\times d}\,,\quad\mathbb{L}^2_{\textrm{sym}}(\Omega)\coloneqq (L^2(\Omega))^{d\times d}_{\textrm{sym}}\,,\\
		\mathbf{H}^1(\Omega)&\coloneqq (H^1(\Omega))^d\,,\quad\!\mathbf{H}_D^1(\Omega) \coloneqq \big\{\bv \in \mathbf{H}^1(\Omega)\mid \bv = \mathbf{0} \text{ q.e.\ in  }\Gamma_D \big\}\,,\\
		\mathbf{H}^{\smash{\frac{1}{2}}}(\gamma)&\coloneqq (H^{\frac{1}{2}}(\gamma))^d\text{ for }\gamma\in \{\partial\Omega,\Gamma_D,\Gamma_N\}\,,\\
		\smash{\HNeumannDual} &\coloneqq \big\{ \bv \in \mathbf{H}^{\smash{\frac{1}{2}}}(\Gamma_N) \mid \text{the zero extension }\widetilde{\bv}\text{ of }\bv\text{ belongs to }\mathbf{H}^{\smash{\frac{1}{2}}}(\partial\Omega)\big\}\,,\\
		\Hdiv &\coloneqq \big\{\btau=(\btau_{ij})_{i,j\in \{1,\ldots,d\}} \in \LtwoM \mid \diver \btau
      \in \mathbf{L}^2(\Omega) \big\}\,.
	\end{align*}
	For any (Lebesgue) measurable set $\omega\subseteq \Omega$, 
	the restriction of the canonical inner product or norm to $\omega$ of 
	$L^2(\Omega)$, $\mathbf{L}^2(\Omega)$, and $\mathbb{L}^2(\Omega)$ is denoted by $(\cdot,\cdot)_\omega$ and $\|\cdot\|_{\omega}$, respectively. For the duality pairings, we employ the~abbreviated~notations
	\begin{align*}
		\begin{aligned} 
			\langle \bff^*, \bv\rangle_\Omega&\coloneqq \langle \bff^*, \bv\rangle_{\HoneDirichlet} &&\; \forall\,\bff^* \in \mathbf{H}^{-1}_D(\Omega)\coloneqq (\HoneDirichlet)^*\,,\; \bv\in \HoneDirichlet\,, \\[-0.75mm]
			\langle \bw, \bv\rangle_{\partial\Omega}&\coloneqq\langle \bw, \bv\rangle_{\smash{\mathbf{H}^{\smash{\frac{1}{2}}}(\partial\Omega)}} &&\; \forall\,\bw \in \mathbf{H}^{-\frac{1}{2}}(\partial\Omega)\coloneqq (\mathbf{H}^{\smash{\frac{1}{2}}}(\partial\Omega))^*\,,\; \bv\in \mathbf{H}^{\smash{\frac{1}{2}}}(\partial\Omega)\,,\\[-0.75mm]
			\langle \bg, \bv\rangle_{\Gamma_N}&\coloneqq \langle \bg, \bv\rangle_{\smash{\mathbf{H}^{\frac{1}{2}}_{00}(\Gamma_N)}} &&\; \forall\,\bg \in \mathbf{H}^{-\frac{1}{2}}_{00}(\Gamma_N)\coloneqq (\mathbf{H}^{\frac{1}{2}}_{00}(\Gamma_N))^*\,,\; \bv\in \HNeumannDual\,.
		\end{aligned}
	\end{align*}
	Eventually, with this notation at hand, next, we define  the subspace of tensor fields in $\Hdiv$ with vanishing normal trace on the Neumann boundary part $\Gamma_N$ (in a generalised sense), \textit{i.e.},
	\begin{equation*}
		\smash{\HdivN \coloneqq  \Big\{ \BT\in \Hdiv \mid \BT\bn = \mathbf{0} \text{ in }\smash{\HNeumann}\Big\}\,.}
	\end{equation*}

	One of the goals of this paper is to perform the numerical analyses for general loads $\smash{\bff^*\hspace{-0.15em}\in \hspace{-0.15em}\mathbf{H}^{-1}_D(\Omega)}$. To this end, we resort to a well-known representation result for the (topological) dual space $\mathbf{H}^{-1}_D(\Omega)$.
	
	\begin{lemma}\label{lem:dual_representation}
		For every load $\bff^*\in \smash{\mathbf{H}^{-1}_D(\Omega)}$, there exist a vector field $\bff \in \mathbf{L}^2(\Omega)$ and a tensor field $\BF\in \LtwoM $ such that for every $\bv\in \HoneDirichlet$, there holds
		\begin{equation}\label{eq:dual_representation}
			\langle \bff^*,\bv\rangle_\Omega =
			(\bff,\bv)_\Omega
			+(\BF, \nabla\bv)_{\Omega}\,.
		\end{equation} 
		Moreover, one has that
		\begin{equation}
			\smash{\|\smash{\bff^*}\|_{\smash{\mathbf{H}^{-1}_D(\Omega)}}
				= \min\left\{\|\bff\|_{\Omega} + \|\BF\|_{\Omega} \mid \bff\in \mathbf{L}^2(\Omega)\,,\,\BF\in \mathbb{L}^2(\Omega) \text{ with }\eqref{eq:dual_representation}\right\}\,.}
		\end{equation} 
		In other words,  $\smash{\mathbf{H}^{-1}_D(\Omega)}$ is isometrically isomorphic to the sum space $\mathbf{L}^2(\Omega)+\textup{div}(\mathbb{L}^2(\Omega))$, equipped with the customary sum norm, where the image space $\textup{div}(\mathbb{L}^2(\Omega))$ is equipped with the customary image norm with respect to the \emph{distributional row-wise divergence} operator $\diver\colon \mathbb{L}^2(\Omega)\to \mathbf{H}^{-1}_D(\Omega)$, for every $\btau\in \LtwoM $ and $\bv\in \HoneDirichlet$  defined by
		\begin{align*}
			\langle\diver\btau, \bv\rangle_{\Omega}\coloneqq-(\btau,\nabla\bv)_{\Omega}\,. 
		\end{align*}
	\end{lemma}
	
	\begin{proof}
		See \cite[Thm. 3.8]{AF03}.
	\end{proof}
	
	For the analysis of the incompressible Stokes problem, we will resort to the \textit{dev-div-inequality} (\textit{cf}.\ \cite[Thm. 3.1]{BNPS.2016} assuming a normal trace boundary condition on some portion of $\partial \Omega$ and \cite[Prop. 9.1.1]{BBF13} assuming a zero mean condition), which states that there exists a constant $c_{\textrm{DD}}>0$, called \textit{dev-div-constant},  depending only on the domain $\Omega$, such that for {every} ${\btau\hspace{-0.1em}\in \hspace{-0.1em}\mathbb{L}^2(\Omega)}$,~there~holds
  \begin{equation}\label{eq:dev_div}
		\smash{\|\btau\|_{\Omega}
			\leq c_{\textrm{DD}}\big\{\|\dev\btau\|_{\Omega} + \|\diver \btau\|_{\smash{\HminusDirichlet}}\big\} \,.}
	\end{equation} 
	Here,  
	the \emph{deviatoric part} of a tensor field $\btau\in \mathbb{L}^2(\Omega)$, defined by
	\begin{align*}
		\smash{\dev\btau\coloneqq \btau-\tfrac{1}{d}\textup{tr}(\btau) \BI\in \mathbb{L}^2_{\textrm{tr}}(\Omega)\,,}
	\end{align*}
	where $\smash{\BI\coloneqq (\delta_{ij})_{i,j=1,\ldots,d}\in\mathbb{R}^{d\times d}}$ denotes the \emph{identity matrix}, is the trace free part of  $\btau\in \mathbb{L}^2(\Omega)$.
	
  For the analysis of the Navier--Lam\'e  problem, we will resort to \emph{Korn's inequality} (\textit{cf}.\ \cite[Thm.\ 6.3-4]{Ciarlet1988}), which states that there exists a constant $c_{\mathrm{K}}>0$, the so-called \emph{Korn constant}, depending only on $\Omega$, such that for every $\bv\in \HoneDirichlet$, there holds 
	\begin{equation}\label{eq:korn}
		\|\nabla \bv\|_{\Omega}^2 \leq c_{\mathrm{K}} \|\symgrad{\bv}\|_{\Omega}^2\,,
	\end{equation}
	where the \textit{symmetric gradient (or strain tensor)} of a vector field (or displacement field)  $\bv\in \HoneDirichlet$ is defined by
	\begin{align*}
		\symgrad{\bv}\coloneqq \tfrac{1}{2}(\nabla \bv + \nabla \bv^\top)\in \mathbb{L}^2_{\textrm{sym}}(\Omega)\,.
	\end{align*}
	Note that, in the case that $\Gamma_D=\partial\Omega$, we have that $c_{\mathrm{K}}=2$ (\textit{cf}.\ \cite[p.\ 294]{Ciarlet1988}).
	
	\subsection{Triangulations}
	
Throughout the entire paper, we denote by $\{\mathcal{T}_h\}_{h>0}$ a family of triangulations of $\Omega\subseteq \mathbb{R}^d$, $d\ge 2$, (\textit{cf}.\  \cite[Def.\ 11.2]{EG21}), consisting of closed simplices $T$. 
	The subscript $h$ denotes the {\textit{averaged mesh-size}}, \textit{i.e.}, denoting by $\mathcal{N}_h$ the \emph{set of vertices} of $\mathcal{T}_h$, we define
	\begin{align*}
		h\coloneqq\big(\tfrac{\vert \Omega\vert}{\textup{card}(\mathcal{N}_h)}\big)^{\smash{\frac{1}{d}}}\,.
	\end{align*}
    
	We assume that the family  of triangulations $\{\mathcal{T}_h\}_{h>0}$ is \textit{shape regular}, \textit{i.e.}, denoting, for
	every $T \in \mathcal{T}_h$,
	by $h_T\coloneqq \textup{diam}(T)$ the diameter of $T$ and by
	$\rho_T\coloneqq \sup\{r>0\mid \exists x\in T\,:\,B_r^d(x)\subseteq T\}$ the supremum of diameters of inscribed balls in $T$, we assume that 
	there exists a constant $\omega_0>0$, called the \textit{chunkiness} of $\{\mathcal{T}_h\}_{h>0}$, which does not depend on $h>0$,~such~that~$\max_{T\in \mathcal{T}_h}{\big\{\frac{h_T}{\rho_T}\big\}}\le\omega_0$.
	Moreover, we introduce the \emph{sets of interior sides} $\mathcal{S}_h^{i}$, \emph{boundary sides} $\mathcal{S}_h^{\partial}$, and \emph{all sides} $\mathcal{S}_h$, which are connected to the degrees of freedom of the finite element spaces of interest and defined by
	\begin{align*}
		\mathcal{S}_h&\coloneqq \mathcal{S}_h^{i}\cup \mathcal{S}_h^{\partial}\,,\\
		\mathcal{S}_h^{i}&\coloneqq \big\{T\cap T'\mid T,T'\in\mathcal{T}_h\,,\text{dim}_{\mathcal{H}}(T\cap T')=d-1\big\}\,,\\
		\mathcal{S}_h^{\partial}&\coloneqq\big\{T\cap \partial\Omega\mid T\in \mathcal{T}_h\,,\text{dim}_{\mathcal{H}}(T\cap \partial\Omega)=d-1\big\}\,,
	\end{align*}
	where the Hausdorff dimension is defined by $\text{dim}_{\mathcal{H}}(\omega) \hspace{-0.1em}\coloneqq\hspace{-0.1em} \inf\{d'\hspace{-0.1em}\geq\hspace{-0.1em} 0\mid {\mathcal{H}^{d'}(\omega)\hspace{-0.1em}=\hspace{-0.1em}0}\}$ for all ${\omega \subseteq   \mathbb{R}^d}$. The \emph{set of Dirichlet sides} $\mathcal{S}_h^D$ and the \emph{set of Neumann~sides}~$\mathcal{S}_h^N$, respectively, are defined by
	\begin{align*}
		\mathcal{S}_h^{\mathrm{X}}\coloneqq\{S\in \mathcal{S}_h^{\partial}\mid \textup{int}(S)\subseteq \Gamma_\mathrm{X}\}\quad\text{ for } \mathrm{X}\in \{D,N\}\,,
	\end{align*}
    where we assume that the triangulation $\mathcal{T}_h$ is such that $\mathcal{S}_h^{\partial}=\mathcal{S}_h^D\dot{\cup}\mathcal{S}_h^N$.
	
	For $k\in \mathbb{N}_0$ and $T\in \mathcal{T}_h$, let $\mathbb{P}^k(T)$ denote the space of polynomials of maximal degree $k$ on $T$. Then, for $k\in \mathbb{N}_0$,
	the \emph{space of {element-wise} polynomial functions of maximal degree $k$} is defined by
	\begin{align*}
		\Pbroken{k} &\coloneqq  \big\{v_h\in L^\infty(\Omega)\mid v_h|_T\in\mathbb{P}^k(T)\text{ for all }T\in \mathcal{T}_h\big\}\,. 
	\end{align*}
	The (local) $L^2$-projection operator $\Pi_h\colon (L^1(\Omega))^{\ell}\to (\Pbroken{0})^{\ell}$, $\ell\in \{1,d,d\times d\}$, onto element-wise constant functions, vector or tensor fields, respectively, for every 
	$\bv\in (L^1(\Omega))^{\ell}$, is defined by 
	\begin{align*}
		\Pi_h \bv|_T\coloneqq \tfrac{1}{\vert T\vert }(\bv,1)_T\quad \text{ for all }T\in \mathcal{T}_h\,.
	\end{align*}
	The (local) $L^2$-projection  operator  $\pi_h\colon \hspace{-0.1em}(L^1(\cup\mathcal{S}_h))^{\ell}\hspace{-0.1em}\to \hspace{-0.1em}(\mathbb{P}^0(\mathcal{S}_h))^{\ell}$, $\ell\hspace{-0.1em}\in\hspace{-0.1em} {\{1,d,d\hspace{-0.1em}\times\hspace{-0.1em} d\}}$, where $\mathbb{P}^0(\mathcal{S}_h)$ is the \emph{space of side-wise constant functions} defined on $\cup\mathcal{S}_h$,  onto side-wise constant functions,  vector or tensor fields, respectively, for every $\bv_h\in (L^1(\cup\mathcal{S}_h))^{\ell}$, is defined by
	\begin{align*}
		\pi_h \bv_h|_S\coloneqq \tfrac{1}{\vert S\vert }(\bv,1)_S\quad \text{ for all }S\in \mathcal{S}_h\,.
	\end{align*} 
	For $k\in \mathbb{N}$, the \emph{element-wise gradient} operator 
	$\nabla_h\colon  (\Pbroken{k})^{\ell}\to(\Pbroken{k-1})^{\ell\times d}$, $\ell\in \{1,d,d\times d\}$, for every $\bv_h\in 
	(\Pbroken{k})^{\ell}$, is defined by $\nabla_h\bv_h|_T\coloneqq\nabla(\bv_h|_T)$ for all ${T\in  \mathcal{T}_h}$, and the \emph{element-wise divergence} operator
	$\textup{div}_h\colon  (\Pbroken{k})^{d\times\ell}\to(\Pbroken{k-1})^{\ell}$, $\ell\in \{1,d\}$, for every $\bv_h\in 
	(\Pbroken{k})^{d\times\ell}$, is defined by $(\textup{div}_h\bv_h|_T)_i\coloneqq\textup{div}(\mathbf{e}_i^\top\bv_h|_T)$\footnote{For every $i=1,\ldots,d$, we denote the $i$-th unit vector by $\mathbf{e}_i\coloneqq (\delta_{ij})_{j=1,\ldots,d}\in \mathbb{S}^{d-1}$.} for all ${T\in  \mathcal{T}_h}$ and $i=1,\ldots,\ell$.
	
  \subsubsection{The Crouzeix--Raviart element}\label{sec:CR}
	
	\hspace{5mm}The \emph{Crouzeix--Raviart space} (\textit{cf}.~\cite{CR73}) (with homogeneous Dirichlet boundary condition on $\Gamma_D$, respectively) is defined by
	\begin{align*}
		\CR &\coloneqq  \big\{\bv_h\in (\Pbroken{1})^d \mid \pi_h\jump{\bv_h}=\mathbf{0}\text{ q.e.\ in }\cup\mathcal{S}_h^{i}\big\}\,,\\
		\CRDirichlet &\coloneqq  \big\{\bv_h\in \CR \mid \pi_h\bv_h=\mathbf{0}\text{ q.e.\ in  }\cup\mathcal{S}_h^{D}\big\}\,.
	\end{align*}
	Here, for  every $\bv_h\hspace{-0.05em}\in\hspace{-0.05em} (\mathbb{P}^1(\mathcal{T}_h))^d$, the \textit{jump (across $\mathcal{S}_h$)} $\jump{\bv_h}\hspace{-0.05em}\in\hspace{-0.05em} (\mathbb{P}^1(\mathcal{S}_h))^d$ is defined by ${\jump{\bv_h}|_S\hspace{-0.05em}\coloneqq \hspace{-0.05em}\jump{\bv_h}_S}$ for all $S\in\mathcal{S}_h$, where, for every  $S\in\mathcal{S}_h$, the \textit{jump (across $S$)} $\jump{\bv_h}_S\in (\mathbb{P}^1(S))^d$ is defined by 
	\begin{align}
		\jump{\bv_h}_S\coloneqq  \begin{cases}
			\bv_h|_{T_+}-\bv_h|_{T_-}&\text{ if }
            \left\{\begin{aligned}
                 S\in\mathcal{S}_h^{i}\text{ and }T_+, T_-\in \mathcal{T}_h\\\text{ with }\partial T_+\cap\partial T_-=S\,,
            \end{aligned}\right.
           \\
			\bv_h|_T&\text{ if }S\in\mathcal{S}_h^{\partial}\text{ and }T\in \mathcal{T}_h\text{ with }S\subseteq \partial T\,.
		\end{cases} 
	\end{align}
	
	It is well-known that Crouzeix--Raviart functions do not satisfy a discrete Korn inequality analogous to \eqref{eq:korn} (\textit{cf}.\ \cite{Falk1991}). 
	To ensure the validity of a Korn inequality at the discrete level, a bounded and symmetric stabilisation bilinear form $s_h \colon \CR \times \CR \to \R$ is usually introduced (\textit{cf}.\ \cite{Bre.2003}), so that
	there exists a  constant $c^{\mathrm{disc}}_{\mathrm{K}}>0$, independent of $h>0$, such that for every~$\bv_h \in \CRDirichlet$,~there~holds 
	\begin{equation}\label{eq:discrete_korn}
		\smash{\|\bbC\nabla_h \bv_h\|_{\Omega}^2 \leq c^{\mathrm{disc}}_{\mathrm{K}} \big\{\|\bbC\symgraddisc{\bv_h}\|_{\Omega}^2 + s_h(\bv_h,\bv_h) \big\}\,,}
	\end{equation}
	where, for convenience, we include the elasticity tensor $\bbC\colon \mathbb{R}^{d\times d}\to \mathbb{R}^{d\times d}$, which is introduced later (\textit{cf}. \eqref{def:C}).
	Then,  an alternative inner product and the associated induced norm on the space $\CRDirichlet$, for every $\bv_h ,\bw_h\in \CRDirichlet$, are~defined~by
    \begin{subequations} 
	\begin{align}
        (\bv_h,\bw_h)_h &\coloneqq  (\bbC \symgraddisc{\bv_h},  \symgraddisc{\bw_h})_\Omega
			+ s_h(\bv_h,\bw_h)\,,\label{eq:discrete_product}\\
		\|\bv_h\|_h^2 &\coloneqq  \big( (\bv_h,\bv_h)_h\big)^{\smash{\frac{1}{2}}}\,.\label{eq:discrete_norm}
	\end{align}
    \end{subequations}
	A prototypical choice for $s_h $ in \eqref{eq:discrete_korn} would be a standard jump penalisation, for every $\bv_h,\bw_h\hspace{-0.15em}\in\hspace{-0.15em} \CR$ {defined} by 
	\begin{equation}
		s_h(\bv_h,\bw_h) \coloneqq
		\sum_{S \in \mathcal{S}_h^i\cup \mathcal{S}_h^D}{\int_S \frac{2\mu}{h_S} \jump{\bv_h}_S \jump{\bw_h}_S\,\mathrm{d}s}\,,
	\end{equation}
	where $h_S\hspace{-0.05em}\coloneqq\hspace{-0.05em}\textup{diam}(S)$ for all $S\hspace{-0.05em}\in \hspace{-0.05em}\mathcal{S}_h$ and, for convenience, we include the shear modulus ${\mu}$ (\textit{cf}. \eqref{def:C}). 
	
	A basis for  $\CR$ is given by $\boldsymbol{\varphi}_S^{i}\in \CR$, $S\in \mathcal{S}_h$, $i=1,\ldots,d$, satisfying  $\boldsymbol{\varphi}_S^{i}(x_{S'})=\delta_{S,S'} \mathbf{e}_i$ for all $S,S'\in \mathcal{S}_h$ and $i=1,\ldots,d$.
	Then, the canonical (quasi-)interpolation operator $\mathcal{I}_h^{cr}\colon \HoneDirichlet \to \CRDirichlet$ (\textit{cf}.\ \cite[Secs.\ 36.2.1--2]{EG21}), for every $\bv\in \smash{\HoneDirichlet}$ defined by
	\begin{align}
		\mathcal{I}_h^{cr}\bv\coloneqq  \sum_{i=1,\ldots,d}{\sum_{S\in \mathcal{S}_h}{\tfrac{1}{\vert S\vert}(\bv,\mathbf{e}_i)_S\,\boldsymbol{\varphi}^{i}_S}}\,,\label{CR-interpolant}
	\end{align}
	has structure-preserving properties, \emph{i.e.}, for every $\bv\in \smash{\HoneDirichlet}$, we have that
	\begin{subequations}\label{eq:ICR_preservation}
		\begin{alignat}{2}
			\nabla_h(\mathcal{I}_h^{cr}\bv)&=\Pi_h(\nabla \bv)&&\quad\text{ a.e.\ in } \Omega\,,\label{eq:ICR_preservation_grad}\\
			\textup{div}_h(\mathcal{I}_h^{cr}\bv)&=\Pi_h(\textup{div}\, \bv)&&\quad\text{ a.e.\ in }\Omega\,.\label{eq:ICR_preservation_div} 
		\end{alignat}
	\end{subequations}
	The gradient-preservation property \eqref{eq:ICR_preservation_grad} implies that for every $\bv\hspace{-0.15em}\in \hspace{-0.15em}\HoneDirichlet$, one has that
	\begin{equation}
		\|\nabla\bv - \nabla_h(\mathcal{I}_h^{cr}\bv)\|_{\Omega}^2
		=
		\sum_{T\in \tria}{\inf_{\BA_T \in (\mathbb{P}^0(T))^{d\times d}}
			\big\{\|\nabla \bv - \BA_T\|_T^2\big\}}\,,
	\end{equation}
	which for functions with the additional regularity $\bv \in \smash{\HoneDirichlet }\cap \mathbf{H}^{1+s}(\tria)\footnote{For $s>0$, we set $H^s(\mathcal{T}_h)\hspace{-0.15em}\coloneqq \hspace{-0.15em}\{v\hspace{-0.15em}\in\hspace{-0.15em} L^2(\Omega)\mid v|_T\hspace{-0.15em}\in\hspace{-0.15em} H^s(T)\;\forall T\hspace{-0.15em}\in\hspace{-0.15em} \mathcal{T}_h\}$, $\mathbf{H}^s(\tria)\hspace{-0.15em}\coloneqq\hspace{-0.15em} (H^s(\mathcal{T}_h))^d$, ${\mathbb{H}^s(\tria)\hspace{-0.15em}\coloneqq\hspace{-0.15em} (H^s(\mathcal{T}_h))^{d\times d}}$.}$, with $s\in (0,1]$, implies that
	\begin{equation}\label{eq:cr_best_approximation}
		\|\nabla \bv - \nabla_h (\mathcal{I}^{cr}_h \bv)\|_{\Omega}^2
		\leq
		c_{cr} \sum_{T\in\tria} h_T^{2s} {|\bv|_{\mathbf{H}^{1+s}(T)}^2}\,,
	\end{equation}
	where $c_{cr} = \smash{\frac{1}{\pi^2}}$ if $s=1$ and is otherwise a  constant $c_{cr}>0$ depending on 
	$\omega_0$ (\textit{cf}.\ {\cite[Thm. 3.2]{Beb.03}} for $s=1$ and \cite[Prop. 2.2]{Heu.14} for $s\in(0,1)$).
	Similarly,  for every $\bv \in \HoneDirichlet$ with $\textup{div}\,\bv\in H^{s}(\mathcal{T}_h)$, with $s\in (0,1]$,
	the divergence-preservation property \eqref{eq:ICR_preservation_div} leads to
	\begin{equation}\label{eq:cr_div_best_approximation}
		\|\diver \bv - \textup{div}_h (\mathcal{I}^{cr}_h \bv)\|_{\Omega}^2
		\leq
		c_{cr} \sum_{T\in\tria} h_T^{2s} {|\diver \bv|_{H^{s}(T)}^2}\,.
	\end{equation}

  \subsubsection{The Raviart--Thomas element}\label{sec:RT}
	
	\hspace{5mm}The \emph{Raviart--Thomas space} (\textit{cf}.\ \cite{RT75}) (with homogeneous normal boundary condition on  $\Gamma_N$, respectively) is defined by
	\begin{align*}
		\RT&\coloneqq  \left\{\btau_h\in (\Pbroken{1})^{d\times d} \;\left|\;
		\begin{aligned} 
			&\,(\btau_h|_T) \mathbf{n}_T=\textup{const}\text{ on }\partial T\text{ for all }T\in \mathcal{T}_h\,,\\ 
			&	\jump{\btau_h\mathbf{n}}_S=\mathbf{0}\text{ on }S\text{ for all }S\in \mathcal{S}_h^{i}
		\end{aligned}\right.\right\}\,,\\
		\RTNeumann& \coloneqq  \big\{\btau_h\in \RT \mid \btau_h\mathbf{n}=\mathbf{0}\text{ q.e.\ in  }\Gamma_N\big\}\,.
	\end{align*} 
	Here, for  every $\btau_h\in (\mathbb{P}^1(\mathcal{T}_h))^{d\times d}$, \emph{normal jump (across $\mathcal{S}_h$)} $\jump{\btau_h\bn}\in (\mathbb{P}^1(\mathcal{S}_h))^d$ is defined by $\jump{\btau_h\bn}|_S=\jump{\btau_h\bn}_S$ for all $S\in \mathcal{S}_h$, where, for every $S\in \mathcal{S}_h$, the \emph{normal jump (across $S$)} $\jump{\btau_h\bn}\in (\mathbb{P}^1(S))^d$ is defined by 
	\begin{align*}
		\jump{\btau_h \bn}_S\coloneqq\begin{cases}
			\btau_h|_{T_+}\bn_{T_+}+\btau_h|_{T_-}\bn_{T_-}&\text{ if }
            \left\{\begin{aligned}
                 S\in\mathcal{S}_h^{i}\text{ and }T_+, T_-\in \mathcal{T}_h\\\text{ with }\partial T_+\cap\partial T_-=S\,,
            \end{aligned}\right.
           \\
			\btau_h|_T\bn&\text{ if }S\in\mathcal{S}_h^{\partial}\text{ and }T\in \mathcal{T}_h\text{ with }S\subseteq \partial T\,.
		\end{cases}
	\end{align*} 
	
	For the Raviart--Thomas space, it is also possible to construct a structure-preserving interpolation operator. One can, \textit{e.g.}, employ the (quasi-)interpolation operator ${\mathcal{I}^{rt}_h\colon \HdivN\to \RTNeumann}$ from \cite{EGSV.2022}, which for every $\btau \in \HdivN$,  satisfies the  divergence-preserving property
	\begin{alignat}{2}
		\textup{div}\,(\mathcal{I}_h^{rt}\btau)&=\Pi_h(\textup{div}\,\btau)&&\quad\text{ a.e.\ in } \Omega\,.\label{eq:IRT_preservation_div}
	\end{alignat}

	Moreover, this interpolation operator is, in fact, a projector and satisfies local stability and approximation properties.
	But crucial for us is a key component in the derivation of an optimal approximation estimate.
	More precisely, for every $s\in [0,1]$, there exists a constant ${c_{rt}>0}$, depending only on $s$, 
	$d$, and the chunkiness $\omega_0$, such that for every $\btau \in \HdivN \cap \mathbb{H}^s(\tria)$, there holds (\textit{cf}.\ \cite[Thm. 3.6]{EGSV.2022})
	\begin{equation}\label{eq:rt_best_approximation}
		\inf_{\overset{\btau_h\in \RTNeumann}{\diver \btau_h = \Pi_h \diver \btau}}
\|\btau - \btau_h\|_{\Omega}^2
		\leq c_{rt}
		\left\{
		\sum_{T\in \tria} h_T^{2s} \|\btau\|^2_{\mathbb{H}^s(T)}
		+ \delta_{s<1} h_T^2 \|\diver \btau\|_T^2
		\right\}\,,
	\end{equation}
	where $\delta_{s<1}=1$ if $s <1$ and $\delta_{s<1}=0$ if $s=1$.
	
	\subsubsection{Discrete integration-by-parts formula}
	
	\hspace{5mm}For every $\bv_h\in \CR$ and $\BT_h\in \RT$, there holds the \textit{discrete integration-by-parts formula}
	\begin{align}
		(\nabla_h\bv_h,\Pi_h \BT_h)_\Omega=-(\Pi_h \bv_h,\,\textup{div}\,\BT_h)_\Omega+(\BT_n\mathbf{n},\pi_h \bv_h)_{\partial \Omega}\,.\label{eq:pi0}
	\end{align}
	
	In addition, crucial in our analysis is the  (orthogonal with respect to the inner product $(\cdot,\cdot)_{\Omega}$) \textit{discrete Helmholtz--Weyl decomposition} (\textit{cf}.\ \cite{BW21,BringmannKettelerSchedensack2025})
	\begin{align}
		\PbrokenM{0} = \textup{ker}(\textup{div}|_{\smash{\RTNeumann}})\oplus \nabla_h(\CRDirichlet) \,.\label{eq:decomposition}
	\end{align}
	Apart from that, the following orthogonality relations hold (\textit{cf}.\ \cite[Sec.\ 2.4]{BW21}):
	\begin{subequations}\label{eq:orthogonality}
		\begin{align}
			(\Pi_h \RTNeumann)^{\perp_{\mathbb{L}^2}}
			&=
			\nabla_h (\mathrm{ker}\,\Pi_h|_{\smash{\CRDirichlet}})  \,,\label{eq:orthogonality1} \\
			(\Pi_h \CRDirichlet)^{\perp_{\BL^2}}
			&=\diver(\mathrm{ker}\,\Pi_h|_{\RTNeumann}) \,.
			\label{eq:orthogonality2}
		\end{align}
	\end{subequations}

	\section{The incompressible Stokes  problem}\label{sec:stokes}
In this section, we consider the incompressible Stokes problem, first introduced by G. G. Stokes in 1845  in his work \cite{Stokes_2009} on the internal friction of fluids, which describes the motion of incompressible viscous Newtonian flows at low Reynolds numbers.
	
	\subsection{The primal problem}\label{sec:stokes_primal}
	
Let $\bff^*\in \HminusDirichlet$ be a given load and let $\bu_D \in \BH^{\frac{1}{2}}(\Gamma_D)$ be a given  Dirichlet boundary datum. For the Dirichlet datum, we assume that there exists an incompressible lifting to the interior of $\Omega$, \textit{i.e.},  
	there exists a vector field $\uDirichlet\in \BH^1(\Omega)$ with $\textup{div}\,\uDirichlet=0$ a.e.\ in $\Omega$ and $\uDirichlet=\bu_D $ q.e.\ in $\Gamma_D$.~In~the~following, we will introduce a formulation for a velocity vector field unknown $\bu \in \HoneDirichlet$. Then, the actual solution to the original problem will be given via $\bu_{\textup{orig}}\coloneqq\uDirichlet+\bu \in \uDirichlet+\HoneDirichlet$.
	
	More precisely, we consider an \textit{incompressible Stokes (minimisation) problem} that seeks a velocity vector field $\bu\in \HoneDirichlet$  that minimises the \emph{primal energy functional} $I\colon \HoneDirichlet \to  \Rext$, for every $\bv \in \HoneDirichlet$ defined by
	\begin{align}
		\smash{I(\bv)\coloneqq \tfrac{\nu}{2}\|\nabla(\bv + \uDirichlet)\|_{\Omega}^2
			+ I_{\{0\}}^{\Omega}(\diver \bv  )
			- \langle \bff^*, \bv \rangle_\Omega\,,}
		\label{eq:stokes_primal}
	\end{align}
	where $\nu>0$ is the \textit{kinematic viscosity} and $\smash{I_{\{0\}}^{\Omega}}\colon L^2(\Omega)\to \mathbb{R}\cup\{+\infty\}$ is the indicator functional, for every $\widehat{v}\in L^2(\Omega)$  defined by
	\begin{align}\label{eq:characteristic_omega}
		I_{\{0\}}^{\Omega}(\widehat{v})
		\coloneqq 
		\begin{cases}
			0&\text{ if }\widehat{v}=0\text{ a.e.\ in }\Omega\,,\\[-0.5mm]
			+\infty&\text{ else}\,.
		\end{cases}
	\end{align}
	Then, the effective domain\footnote{For a functional $F\colon X\to \mathbb{R}\cup\{+\infty\}$, the \emph{effective domain} is defined by $\textup{dom}(F)\coloneqq \{x\in X\mid F(x)<+\infty\}$.} of the primal energy functional  \eqref{eq:stokes_primal} is given via
	\begin{align*}
		\smash{K_{\bu}\coloneqq \textup{dom}(I)\coloneqq \big\{\bv\in \HoneDirichlet\mid \diver \bv=0\text{ a.e.\ in }\Omega\big\}\,.}
	\end{align*}
	In this section, we refer to the minimisation of \eqref{eq:stokes_primal} as the \textit{primal problem}. The existence of a unique minimiser $\bu\hspace*{-0.15em}\in \hspace*{-0.15em}\HoneDirichlet$, called \textit{primal solution},  follows from the direct method in the calculus of variations.   
	Since  \eqref{eq:stokes_primal} is not Fréchet differentiable, the optimality condition associated with the primal problem is not given via a variational equality (\textit{i.e.}, in $\HminusDirichlet$), but~in~a~\mbox{hydro-mechanical}~equality instead: 
	in fact, $\bu\in K_{\bu}$ is minimal for \eqref{eq:stokes_primal} if and only if for {every} ${\bv\in K_{\bu}}$, there holds
	\begin{align*}
		\nu(\nabla \bu,\nabla \bv)_{\Omega} = \langle \bff^*,\bv \rangle_\Omega\,.\\[-6.5mm]
	\end{align*}
	
	\subsection{The dual problem}
	
In this subsection, we derive and examine a (Lagrangian) dual problem (in the sense of  \cite[Chap. 49]{ZeiIII}) to the incompressible Stokes problem \eqref{eq:stokes_primal}. To this end, we introduce the Lagrangian $\mathcal{L}\colon \LtwoM\times \HoneDirichlet \to \R$, for every $(\btau,\bv)\in  \LtwoM\times \HoneDirichlet $ defined by
	\begin{align}\label{eq:lagrangian_stokes}
		\smash{\mathcal{L}(\btau,\bv) \coloneqq 
			-\tfrac{1}{2\nu} \|\dev\btau\|_{\Omega}^2
			+ (\btau, \nabla(\bv + \uDirichlet))_\Omega
			- \langle \load, \bv\rangle_\Omega\,. }
	\end{align}
	Note that the (Fenchel) conjugate functional (with respect to the~second argument) of the energy density $\phi\colon\Omega\times\mathbb{R}^{d\times d}\to \mathbb{R}\cup\{+\infty\}$, for a.e.\ $x\in \Omega$ and every $\BA\in \mathbb{R}^{d\times d}$ defined by
	\begin{align*}
		\phi(x,\BA)\coloneqq \begin{cases}
			\frac{\nu}{2}\vert \BA
			+\nabla\widehat{\bu}_D(x)\vert^2&\text{ if }\textrm{tr}\,\BA=0\,,\\
			+\infty &\text{ else}\,,
		\end{cases} 
	\end{align*}
	is given via $\smash{\phi^*\colon\Omega\times\mathbb{R}^{d\times d}\to \mathbb{R}}$, for a.e.\ $x\in \Omega$ and every $\BA\in \smash{\mathbb{R}^{d\times d}}$ defined by
	\begin{align*}
		\smash{\phi^*(x,\BA)\coloneqq \tfrac{1}{2\nu}\vert \dev \BA
			\vert^2-\BA:\nabla\widehat{\bu}_D(x)\,.}
	\end{align*}
	As a  consequence, according to \cite[Thm.\ 2]{Rockafellar1968}, for every $\bv\in  \HoneDirichlet $,~we~have~that
	\begin{align*}
		\tfrac{\nu}{2}\|\nabla(\bv + \uDirichlet)\|_{\Omega}^2
		+ I_{\{0\}}^{\Omega}(\diver \bv  )=\sup_{\btau \in \LtwoM}{\big\{-\tfrac{1}{2\nu} \|\dev\btau\|_{\Omega}^2
			+ (\btau, \nabla(\bv + \uDirichlet))_\Omega\big\}}\,,
	\end{align*}
	which implies that, for every $\bv\in \smash{\HoneDirichlet}$, we have that
	\begin{gather}
		\smash{I(\bv) =
			\sup_{\btau \in \LtwoM}{\big\{\mathcal{L}(\btau,\bv)\big\}}}\,.
	\end{gather}
	As the Lagrangian \eqref{eq:lagrangian_stokes} is concave-convex and $\mathcal{L}(\btau,\bv)\hspace{-0.12em}\to\hspace{-0.12em} -\infty$ ${(\|\btau\|_{\Omega}\hspace{-0.12em}\to\hspace{-0.12em}+\infty)}$ for all ${\bv\hspace{-0.12em}\in\hspace{-0.12em} \HoneDirichlet}$, 
	by the  min-max-theorem (\textit{cf}.\ \cite[Thm.\ 49.B]{ZeiIII}), 
	it admits a saddle point $(\BT,\bu)\in \LtwoM\times \HoneDirichlet$, \textit{i.e.}, we have that
	\begin{align}\label{eq:stokes_minmax}
		\min_{\bv\in \HoneDirichlet}{\max_{\btau\in \LtwoM}{\big\{\mathcal{L}(\btau,\bv)\big\}}}=\mathcal{L}(\BT,\bu)=\max_{\btau\in \LtwoM}{\min_{\bv\in \HoneDirichlet}{\big\{\mathcal{L}(\btau,\bv)\big\}}}\,.
	\end{align} 
	The optimality conditions associated with the corresponding saddle point problem amount to finding $(\BT,\bu)\in \LtwoM\times \HoneDirichlet$
	such that for every $(\btau,\bv)\in \LtwoM\times \HoneDirichlet$, there holds
	\begin{subequations} \label{eq:stokes_optimality_continuous}
		\begin{alignat}{2}
			-\tfrac{1}{\nu}(\dev\BT, \dev\btau)_\Omega
			+ (\nabla(\bu + \uDirichlet), \btau)_\Omega
			&	=0 \,,  \label{eq:stokes_optimality_continuous_1}\\
			(\BT,\nabla \bv)_\Omega &= \langle \load,\bv \rangle_\Omega\,,
		\end{alignat}
	\end{subequations}
	where \eqref{eq:stokes_optimality_continuous_1} is 
    equivalent to
	\begin{align}\label{eq:stokes_optimality_continuous_3}
		\dev\BT=\nu\nabla(\bu + \uDirichlet)\quad\text{ a.e.\ in }\Omega\,.
	\end{align}
	Then, according to \cite[Sec.\ 49.2]{ZeiIII}, the associated (Lagrangian) dual problem is given via the maximisation of the \emph{dual energy functional} $D\colon \LtwoM \to \Rext$, for every $\btau\in \LtwoM$ defined by
	\begin{equation}\label{eq:stokes_primal_0}
		D(\btau) \coloneqq \inf_{\bv\in\HoneDirichlet}{\big\{\mathcal{L}(\btau,\bv)\big\}}\,.
	\end{equation}
	The infimum in definition \eqref{eq:stokes_primal_0} is finite only if for every $\bv\in \HoneDirichlet$,~we~have~that
	\begin{equation}
		(\btau,\nabla \bv)_\Omega = \langle \load ,\bv \rangle_\Omega = (\bff,\bv)_\Omega + (\BF,\nabla \bv)_\Omega\,,
	\end{equation}
	where we used the representation \eqref{eq:dual_representation}, which is 
    equivalent to   $\btau-\BF\in \Hdiv$ and
	\begin{subequations}\label{eq:stokes_constraints}
		\begin{alignat}{2}
			\diver (\btau - \BF) &= -\bff &&\quad \text{ a.e. in }\Omega\,,\label{eq:stokes_stress_constraint}\\
			\langle (\btau - \BF)\bn, \widetilde{\bv} \rangle_{\partial\Omega} &= 0
			&&\quad\text{ for all }\bv\in \smash{\HNeumannDual}\,. \label{eq:Neumann_condition_weak}
		\end{alignat}
	\end{subequations}
	Recall that $\widetilde{\bv}\in \BH^{\frac{1}{2}}(\partial\Omega)$ denotes the zero extension of a vector field ${\bv\in \HNeumannDual}$. For this reason, it is enough to consider the restricted dual energy functional $D\colon \BF+\HdivN\to \Rext $, for every $\btau\in \BF+\HdivN$ admitting the explicit integral representation
	\begin{equation}\label{eq:stokes_dual}
		D(\btau) = 
		-\tfrac{1}{2\nu}\|\dev\btau\|_{\Omega}^2
		+ (\dev\btau,\nabla \uDirichlet)_\Omega- I^\Omega_{\{-\bff\}}(\diver (\btau-\BF))\,,
	\end{equation}
	where the indicator functional $I^\Omega_{\{-\bff\}}\colon \mathbf{L}^2(\Omega)\to \Rext$, for every $\widehat{\bv}\in \mathbf{L}^2(\Omega)$, is defined by
	\begin{align*}
		I^\Omega_{\{-\bff\}}(\widehat{\bv})\coloneqq \begin{cases}
			0&\text{ if }\widehat{\bv}=-\bff\text{ a.e.\ in  }\Omega\,,\\
			+\infty&\text{ else}\,.
		\end{cases}
	\end{align*} 
	Then, the effective domain of the negative of the dual energy functional \eqref{eq:stokes_dual} is given via
	\begin{align*}
		K_{\BT}\coloneqq \textup{dom}(-D)\coloneqq \big\{\btau\in \BF +\HdivN\mid \diver (\btau-\BF)=-\bff\text{ a.e.\ in }\Omega\big\}\,.
	\end{align*}
	By \eqref{eq:stokes_minmax}, there exists a unique maximiser $\BT \in \BF+ \HdivN$ of the dual energy functional \eqref{eq:stokes_dual}, called \emph{dual solution}, and a \emph{strong duality relation} applies, \textit{i.e.}, we have that
	\begin{equation}\label{eq:stokes_strong_duality}
		I(\bu) = D(\BT)\,.
	\end{equation}
  In the case of full Dirichlet boundary condition (\textit{i.e.}, $\Gamma_D\hspace{-0.1em}=\hspace{-0.1em}\partial \Omega$), the uniqueness of the {spherical} part of the stress (pressure) $\tfrac{1}{d}\mathrm{tr}(\BT)$ is usually enforced through imposing a zero-mean condition, \textit{i.e.}, imposing ${\int_\Omega \mathrm{tr}(\BT)\,\mathrm{d}x = 0}$.
	
	\begin{remark}\label{rem:Neumann}
		The condition \eqref{eq:Neumann_condition_weak} can be interpreted as Neumann boundary condition on $\Gamma_N$.
		In other words, in our framework, the Neumann boundary datum is encoded in the tensor field $\BF\in \LtwoM$.
		A relation of the form \eqref{eq:Neumann_condition_weak} with test functions in $\HNeumannDual$ is the correct representation of the Neumann boundary condition
		``$\BT\bn = \BF\bn$ on $\Gamma_N$'' in the weak setting (\textit{cf}.\ \cite[Sec.~31.3.3]{EG21b}).
	\end{remark}

	\subsection{The discrete primal problem}\label{sec:stokes_discrete_primal}
Similarly to the continuous setting, we assume that the Dirichlet boundary condition can be described by a vector field defined in the bulk $\uhDirichlet \in \CR$;
	we assume that $\textup{div}_h\uhDirichlet=0$ a.e. in $\Omega$. Such a vector field can be obtained by suitably interpolating $\uDirichlet$ onto $\CR$.
	In a similar manner to the continuous setting, the discrete formulation will, then, be framed in terms of a vector field $\bu_h \in \CRDirichlet$,
	but the desired approximation to the original problem will be ${\bu_{\textup{orig}}^h\hspace{-0.15em}\coloneqq\hspace{-0.15em}\bu_h + \uhDirichlet\hspace{-0.15em}\in\hspace{-0.15em} V^h}$. Furthermore, we denote element-wise constant approximations of the vector field $\bff\in \BL^2(\Omega)$ and the tensor field $\BF\in \mathbb{L}^2(\Omega)$ defining the load $\bff^*\in \HminusDirichlet$ by $\bff_h \in (\Pbroken{0})^d$ and $\BF_h \in \PbrokenM{0}$, respectively.
	
	Then, we consider a \emph{discrete incompressible Stokes (minimisation) problem} that seeks a discrete velocity vector field
	$\bu_h \in \CRDirichlet$ that minimises the \emph{discrete primal energy functional} $I_h \colon \CRDirichlet \to \Rext$, for every $\bv_h \in \CRDirichlet$ defined by
	\begin{align}\label{eq:stokes_primal_discrete}
    \begin{aligned} 
		I_h(\bv_h) &\coloneqq  
			\tfrac{\nu}{2}  \|\nabla_h (\bv_h + \uhDirichlet)\|_{\Omega}^2
			+
			\characteristic{0}^\Omega(\textup{div}_h\bv_h)
			\\&\quad- (\bff_h,\Pi_h\bv_h)_\Omega
			- (\BF_h,\nabla_h \bv_h)_\Omega\,.
            \end{aligned}
	\end{align}
	The effective domain of the discrete  primal energy functional \eqref{eq:stokes_primal_discrete} is given via
	\begin{align*}
		\smash{K^h_{\bu} \coloneqq
			\mathrm{dom}(I_h) = 
			\big\{\bv_h \in \CRDirichlet \mid \textup{div}_h \bv_h = 0\text{ a.e.\ in }\Omega\big\}\,.}
	\end{align*}In this section, we refer to the minimisation of \eqref{eq:stokes_primal_discrete} as the \textit{discrete primal problem}. The existence of a unique minimiser $\bu_h \in K^h_{\bu}$, called \textit{discrete primal solution}, follows from the direct method in the calculus of variations.   
	Since  \eqref{eq:stokes_primal_discrete} is not Frech\'et differentiable, the optimality condition associated with the discrete primal problem is not given via a discrete variational {equality}~(\textit{i.e.},~in~$(\CRDirichlet)^*$), but in a discrete hydro-mechanical equality instead: 
	in fact, $\bu_h \in K^h_{\bu}$ is minimal for \eqref{eq:stokes_primal_discrete} if and only if for every $\bv_h\in K_{\bu}^h$, there holds
	\begin{equation}
		\smash{\nu	(\nabla_h (\bu_h+\uhDirichlet) ,\nabla_h \bv_h)_\Omega = (\bff_h, \Pi_h\bv_h)_\Omega
			+ (\BF_h,\nabla_h\bv_h)_\Omega\,.}
	\end{equation}
	
	
	\subsection{The discrete dual problem}\label{sec:stokes_discrete_duality}
	
In this subsection, we derive and examine a (Lagrangian) dual problem (in the sense of  \cite[Chap. 49]{ZeiIII}) to the discrete incompressible Stokes  problem \eqref{eq:stokes_primal}. To this end, we introduce the Lagrangian $\mathcal{L}_h^0\colon\PbrokenM{0}\times \CRDirichlet \to \R$, for every $(\overline{\btau}_h,\bv_h)\in  \PbrokenM{0}\times \CRDirichlet $ defined by
	\begin{align}\label{eq:lagrangian_stokes_discrete}
		 \begin{aligned}\mathcal{L}_h^0(\overline{\btau}_h,\bv_h) &\coloneqq 
			-\tfrac{1}{2\nu} \|\dev\overline{\btau}_h\|_{\Omega}^2
			+ (\overline{\btau}_h, \nabla(\bv_h + \uhDirichlet))_\Omega
			\\&\quad- (\bff_h,\bv_h)_{\Omega}-(\BF_h,\nabla_h\bv_h)_{\Omega}\,.  \end{aligned}
	\end{align}
	Note that the (Fenchel) conjugate functional (with respect to the second argument) of the energy density  $\phi_h\colon\Omega\times\mathbb{R}^{d\times d}\to \mathbb{R}\cup\{+\infty\}$, for a.e.\ $x\in \Omega$ and every $\BA\in \mathbb{R}^{d\times d}$ defined by
	\begin{align*}
		\phi_h(x,\BA)\coloneqq \begin{cases}
			\frac{\nu}{2}\vert \BA
			+\nabla_h\uhDirichlet(x)\vert^2&\text{ if }\textrm{tr}\,\BA=0\,,\\
			+\infty &\text{ else}\,,
		\end{cases} 
	\end{align*}
	is given via $\phi^*_h\colon\Omega\times\mathbb{R}^{d\times d}\to \mathbb{R}$, for a.e.\ $x\in \Omega$ and every $\BA\in \mathbb{R}^{d\times d}$ defined by
	\begin{align*}
		\smash{\phi^*_h(x,\BA)\coloneqq \tfrac{1}{2\nu}\vert \dev \BA
			\vert^2-\BA:\nabla\widehat{\bu}_D^h(x)\,.}
	\end{align*}
	As a  consequence, applying element-wise definition of $\phi_h^*$, for every $\bv_h\in  \CRDirichlet $, we find that
	\begin{align*}
		&\tfrac{\nu}{2}\|\nabla_h(\bv_h + \uhDirichlet)\|_{\Omega}^2
		+ I_{\{0\}}^{\Omega}(\textrm{div}_h \bv_h  )\\&=\sup_{\overline{\btau}_h \in \PbrokenM{0}}{\big\{-\tfrac{1}{2\nu} \|\dev\overline{\btau}_h\|_{\Omega}^2
			+ (\overline{\btau}_h, \nabla_h(\bv_h + \uhDirichlet))_\Omega\big\}}\,,
	\end{align*}
	which implies that, for every $\bv_h\in \smash{\CRDirichlet}$, we have that
	\begin{gather}
		I_h(\bv_h) =
		\sup_{\overline{\btau}_h \in \PbrokenM{0}} {\big\{\mathcal{L}_h^0(\overline{\btau}_h,\bv_h)\big\}}\,.
	\end{gather}
	As the Lagrangian \eqref{eq:lagrangian_stokes_discrete} is concave-convex and $\mathcal{L}_h^0(\overline{\btau}_h,\bv_h)\hspace{-0.15em}\to \hspace{-0.15em}-\infty$ ${(\|\overline{\btau}_h\|_{\Omega}\hspace{-0.15em}\to\hspace{-0.15em} +\infty)}$ for all ${\bv_h\hspace{-0.12em}\in\hspace{-0.12em} \CRDirichlet}$, 
	by the min-max-theorem (\textit{cf}.\ \cite[Thm.\ 49.B]{ZeiIII}), 
	it admits a saddle point ${(\overline{\BT}_h,\bu_h)\hspace{-0.15em}\in\hspace{-0.15em} \PbrokenM{0}\hspace{-0.2em}\times\hspace{-0.2em} \CRDirichlet}$, \textit{i.e.}, we have that
	\begin{equation}\label{eq:stokes_minmax_discrete}
    \begin{aligned}
		\min_{\bv_h\in \CRDirichlet}{\max_{\overline{\btau}_h\in \PbrokenM{0}}{\big\{\mathcal{L}_h^0(\overline{\btau}_h,\bv_h)\big\}}}
    &=\mathcal{L}_h^0(\overline{\BT}_h,\bu_h)
  \\&=\max_{\overline{\btau}_h\in \PbrokenM{0}}{\min_{\bv_h\in \CRDirichlet}{\big\{\mathcal{L}_h^0(\overline{\btau}_h,\bv_h)\big\}}}\,.
    \end{aligned}
	\end{equation} 
	The optimality conditions associated with the corresponding saddle point problem amount to finding $(\overline{\BT}_h,\bu_h)\in \PbrokenM{0}\times \CRDirichlet$
	such that for every $(\overline{\btau}_h,\bv_h)\in \PbrokenM{0}\times \CRDirichlet$, there holds
	\begin{subequations} \label{eq:stokes_optimality_discrete}
		\begin{alignat}{2}
			-\tfrac{1}{\nu}(\dev\overline{\BT}_h, \dev\overline{\btau}_h)_\Omega
			+ (\nabla_h(\bu_h + \uhDirichlet), \overline{\btau}_h)_\Omega
			&	=0 \,,  \label{eq:stokes_optimality_discrete_1}\\
			(\overline{\BT}_h,\nabla_h \bv_h)_\Omega &= (\bff_h,\bv_h)_{\Omega}+(\BF_h,\nabla_h\bv_h)_{\Omega}\,,\label{eq:stokes_optimality_discrete_2}
		\end{alignat}
	\end{subequations}
	where \eqref{eq:stokes_optimality_discrete_1} is 
    equivalent to
	\begin{align}\label{eq:stokes_optimality_discrete_3}
		\dev\overline{\BT}_h=\nu \nabla_h(\bu_h + \uhDirichlet)\quad\text{ a.e.\ in }\Omega\,.
	\end{align}
	Then, according to \cite[Sec.\ 49.2]{ZeiIII}, the associated (Lagrangian) dual problem is given via the maximisation of the discrete dual energy functional   $D_h^0\colon \PbrokenM{0} \to \Rext$, for every $\overline{\btau}_h\in \PbrokenM{0}$ defined by 
	\begin{equation}\label{eq:stokes_dual_0_discrete}
		D_h^0(\overline{\btau}_h) \coloneqq \inf_{\bv_h\in\CRDirichlet} {\big\{\mathcal{L}_h^0(\overline{\btau}_h,\bv_h)\big\}}\,.
	\end{equation}
	The infimum in definition \eqref{eq:stokes_dual_0_discrete} is finite only if for every $\bv_h\in \CRDirichlet$, we have that
	\begin{equation}
		(\overline{\btau}_h,\nabla_h \bv_h)_\Omega = (\bff_h,\Pi_h\bv_h)_\Omega + (\BF_h,\nabla_h \bv_h)_\Omega\,,
	\end{equation}
	which is 
    equivalent to the existence of $\btau_h\in \BF_h+\RTNeumann$ such that $\overline{\btau}_h=\Pi_h\btau_h$ a.e. $\Omega$ and
	\begin{subequations}\label{eq:stokes_constraints_discrete}
		\begin{alignat}{2}
			\diver (\btau_h - \BF_h) &= -\bff_h &&\quad \text{ a.e.\ in }\Omega\,,\label{eq:stokes_stress_constraint_discrete}\\
			(\btau_h - \BF_h)\bn &= \mathbf{0}
			&&\quad\text{ q.e.\ in }\Gamma_N\,. \label{eq:Neumann_condition_weak_discrete}
		\end{alignat}
	\end{subequations}
	For this reason, it is enough to consider the restricted discrete dual energy functional $D_h^0\colon \BF_h+\Pi_h(\RTNeumann)\to \Rextd $, for every $\overline{\btau}_h\in \BF_h+\Pi_h(\RTNeumann)$ and $\btau_h\in \BF_h+\RT$ such that $\overline{\btau}_h=\Pi_h\btau_h$ a.e.\ in $\Omega$ admitting the explicit integral representation
	\begin{align}\label{eq:stokes_dual0_discrete}
		\begin{aligned} D_h^0(\overline{\btau}_h) &= 
		-\tfrac{1}{2\nu}\|\dev\Pi_h\btau_h\|_{\Omega}^2
		+ (\dev\Pi_h\btau_h,\nabla \uDirichlet)_\Omega\\&\quad- I^\Omega_{\{-\bff_h\}}(\diver (\btau_h-\BF_h))\,.
        \end{aligned}
	\end{align}
	We introduce the \textit{discrete dual energy functional} $D_h\colon \BF_h+\RTNeumann\to \Rextd $, for every $\btau_h \in \BF_h+\RTNeumann$ defined by $D_h(\btau_h)\coloneqq D_h^0(\Pi_h\btau_h)$,~\textit{i.e.},~we~have~that
	\begin{align}\label{eq:stokes_dual_discrete}
		\begin{aligned} 
			D_h(\btau_h)&\coloneqq -\tfrac{1}{2\nu}\|\dev\Pi_h\btau_h\|_{\Omega}^2
			+ (\dev\Pi_h\btau_h,\nabla_h \uDirichlet)_\Omega\\&\quad- I^\Omega_{\{-\bff_h\}}(\diver (\btau_h-\BF_h))\,.
		\end{aligned}
	\end{align}
	Then, the effective domain of the negative of the discrete dual energy functional \eqref{eq:stokes_dual_discrete} is given via
	\begin{align*}
		K_{\BT}^h\coloneqq \textup{dom}(-D_h)\coloneqq \big\{\btau_h\in \BF_h +\RTNeumann\mid \diver (\btau_h-\BF_h)=-\bff_h\text{ a.e.\ in }\Omega\big\}\,.
	\end{align*}
	By \eqref{eq:stokes_minmax_discrete}, there exists a unique maximiser $\overline{\BT}_h \in \PbrokenM{0}$ of the discrete dual energy functional \eqref{eq:stokes_dual_0_discrete} and a \emph{discrete strong duality relation} applies, \textit{i.e.}, we have that ${I_h(\bu_h) \hspace{-0.175em}=\hspace{-0.175em} D_h^0(\overline{\BT}_h)}$. Since $\textup{dom}(D_h^0)\hspace{-0.175em}\subseteq\hspace{-0.175em} \BF_h\hspace{-0.05em}+\hspace{-0.05em}\Pi_h(\RTNeumann)$, \hspace{-0.175mm}there \hspace{-0.175mm}exists \hspace{-0.175mm}$\BT_h \hspace{-0.175em}\in\hspace{-0.175em} \BF_h\hspace{-0.05em}+\hspace{-0.05em}\RTNeumann$ \hspace{-0.175mm}with \hspace{-0.175mm}$\overline{\BT}_h\hspace{-0.175em}=\hspace{-0.175em}\Pi_h\BT_h$ \hspace{-0.175mm}a.e.\ \hspace{-0.175mm}in \hspace{-0.175mm}$\Omega$ \hspace{-0.175mm}and, \hspace{-0.175mm}thus, we \hspace{-0.175mm}have \hspace{-0.175mm}that
	\begin{align}\label{eq:stokes_strong_duality_discrete}
		I_h(\bu_h) = D_h(\BT_h)\,.
	\end{align} 
	
	\begin{proposition}[(Inverse) Marini formula]\label{prop:stokes_marini} The following statements apply:
		\begin{itemize}[noitemsep,topsep=0pt,leftmargin=!,labelwidth=\widthof{(ii)}]
			\item[(i)]\hypertarget{prop:stokes_marini.i}{} Given the discrete primal solution $\bu_h\in K_{\bu}^h$ (\textit{i.e.}, the discrete velocity) and a discrete Lagrange multiplier $p_h\in \Pbroken{0}$ (\textit{i.e.}, the discrete pressure) such that  
			for every $\bv_h\in \CRDirichlet$, there holds
			\begin{equation}\label{eq:stokes_with_pressure}
				\smash{(\nu\nabla_h (\bu_h+\uhDirichlet) - p_h\BI  , \nabla_h \bv_h)_\Omega =
				( \bff_h ,\bv_h)_\Omega
				+ (\BF_h, \nabla_h \bv_h)_\Omega\, ,}
			\end{equation}
			the discrete dual solution $\BT_h\in K_{\BT}^h$ (\textit{i.e.}, the discrete stress) is immediately available via 
			\begin{align}\label{eq:stokes_marini}
				\smash{\BT_h= \nu\nabla_h (\bu_h + \uhDirichlet) - p_h \BI
					- \tfrac{1}{d}\bff_h \otimes (\mathrm{id}_{\Rd} -\Pi_h \mathrm{id}_{\Rd})\quad\text{ a.e.\ in }\Omega}\,;
			\end{align}
			\item[(ii)]\hypertarget{prop:stokes_marini.ii}{}  Given the discrete dual solution $\BT_h\in K_{\BT}^h$ (\textit{i.e.}, the discrete stress) and a discrete Lagrange multiplier $\overline{\bu}_h\in (\Pbroken{0})^d$ such that 
        for every $\btau_h\in \RTNeumann $, there holds 
			\begin{equation} \label{eq:stokes_with_multiplier} 
				\smash{-\tfrac{1}{\nu}(\dev\Pi_h\BT_h, \dev\Pi_h\btau_h)_\Omega
				- (\overline{\bu}_h, \diver\btau_h)_\Omega
				=-(\nabla_h\uhDirichlet,\dev \Pi_h\btau_h)_\Omega }\,,
			\end{equation}
			the discrete primal solution $\bu_h\in K_{\bu}^h$ (\textit{i.e.}, the discrete velocity) is immediately available via
			\begin{align}\label{eq:stokes_marini_inverse}
				\smash{\bu_h= \overline{\bu}_h+\tfrac{1}{\nu }\dev\Pi_h\BT_h(\mathrm{id}_{\Rd} -\Pi_h \mathrm{id}_{\Rd})\quad\text{ a.e.\ in }\Omega}\,.
			\end{align}
		\end{itemize}
	\end{proposition}
	
	\begin{remark}
		Such reconstruction formulas for the incompressible Stokes  problem were known in particular cases; e.g., \cite{CGS.2013} derived it for $d=2$, $\BF=\mathbf{0}$ a.e.\ in $\Omega$, and more regular Dirichlet boundary datum $\uDirichlet  \in \BH^1(\Gamma_D) \cap \BH^1(\Omega)$.
		Proposition \ref{prop:stokes_marini}, however, ensures that the formulas are valid for any $d\ge 2$,  general loads $\bff^*\in \mathbf{H}^{-1}_D(\Omega)$, and inhomogeneous mixed boundary conditions.
	\end{remark}
	
	\begin{proof}[of Proposition \ref{prop:stokes_marini}]
		\textit{ad (\hyperlink{prop:stokes_marini.i}{i}).} To begin with, we define tensor field
		\begin{align*}
			\smash{\widehat{\BT}_h}\coloneqq \nu\nabla_h (\bu_h + \uhDirichlet) - p_h \BI
			- \tfrac{1}{d}\bff_h \otimes (\mathrm{id}_{\Rd} -\Pi_h \mathrm{id}_{\Rd})\in \PbrokenM{1}\,,
		\end{align*}
		which is locally (on each element) a Raviart--Thomas tensor field with $\textup{div}_h\widehat{\BT}_h=-\bff_h$ a.e. in $\Omega$ and, by definition, satisfies
		\begin{align}\label{prop:stokes_marini.1}
			\smash{\Pi_h\smash{\widehat{\BT}_h}=\nu\nabla_h (\bu_h+\uhDirichlet) - p_h\BI \quad\text{ a.e.\ in }\Omega\,.}
		\end{align}
        Due to $\dev (p_h\BI)=0$ a.e.\ in $\Omega$, \eqref{prop:stokes_marini.1} implies that
		\begin{align}\label{prop:stokes_marini.2}
			\smash{\dev\Pi_h\smash{\widehat{\BT}_h}=\nu\nabla_h (\bu_h+\uhDirichlet) \quad\text{ a.e.\ in }\Omega\,.}
		\end{align}
		Due to $\diver(\smash{\widehat{\BT}_h}-\BT_h)=\mathbf{0}$ a.e.\ in $T$ for all $T\in \mathcal{T}_h$, we have that $\smash{\widehat{\BT}_h}-\BT_h\in \PbrokenM{0}$. Thus,  from \eqref{prop:stokes_marini.1}, the discrete integration-by-parts formula \eqref{eq:pi0}, and \eqref{eq:stokes_with_pressure}, for every $\bv_h\in \CRDirichlet$, it follows that
		\begin{align}\label{prop:stokes_marini.2.1}
            \begin{aligned}
			 (\widehat{\BT}_h-\BT_h+\BF_h,\nabla_h\bv_h)_{\Omega}& 
			 = (\Pi_h \widehat{\BT}_h-\Pi_h\BT_h+\BF_h,\Pi_h \nabla_h\bv_h)_{\Omega}\\&=
			 (\nu\nabla_h (\bu_h+\uhDirichlet) - p_h\BI  , \nabla_h \bv_h)_\Omega
                                                                                    \\&\quad +(\diver (\BT_h-\BF_h)  ,\Pi_h\bv_h)_\Omega
			 \\&=(\BF_h, \nabla_h \bv_h)_\Omega\,.
            \end{aligned}
		\end{align}
		Then, using the discrete Helmholtz--Weyl decomposition \eqref{eq:decomposition}, from \eqref{prop:stokes_marini.2.1}, we infer that 
		\begin{align*}
			\smash{\widehat{\BT}_h}-\BF_h-(\BT_h-\BF_h)=\smash{\widehat{\BT}_h}-\BT_h\in \smash{(\nabla_h(\CRDirichlet))^{\perp_{\mathbb{L}^2}}}=\textup{ker}(\textup{div}|_{\smash{\RTNeumann}})\,,
		\end{align*}
		and, due to $\BT_h-\BF_h\in \RTNeumann$, we find that $\smash{\widehat{\BT}_h}-\BF_h\in \RTNeumann$ with
		\begin{align}\label{prop:stokes_marini.3}
			\textup{div}(\smash{\widehat{\BT}_h}-\BF_h)=\textup{div}(\BT_h-\BF_h)=-\bff_h\quad\text{ a.e.\ in }\Omega\,,
		\end{align}
		\textit{i.e.}, we have that $\widehat{\BT}_h\hspace{-0.15em}\in\hspace{-0.15em} K_{\BT}^h$.
		In particular, due to \eqref{prop:stokes_marini.2} and \eqref{prop:stokes_marini.3}, we conclude that ${\widehat{\BT}_h=\BT_h}$ a.e.\ in $\Omega$.
		
		\textit{ad (\hyperlink{prop:stokes_marini.ii}{ii}).} To begin with, we define the vector field
		\begin{align*}
			\smash{\widehat{\bu}_h\coloneqq \overline{\bu}_h+\tfrac{1}{\nu }\dev\Pi_h\BT_h(\mathrm{id}_{\Rd} -\Pi_h \mathrm{id}_{\Rd})\in (\Pbroken{1})^d}\,,
		\end{align*}
    which is locally (on each element) a  
    $\mathbb{P}^1$-vector field, which,  
    by definition, satisfies
		\begin{subequations} 
			\begin{alignat}{2} \label{prop:stokes_marini.4}
				\nabla_h  \widehat{\bu}_h&=\tfrac{1}{\nu }\dev\Pi_h\BT_h&&\quad\text{ a.e.\ in }\Omega\,,\\
				\Pi_h  \widehat{\bu}_h&=\overline{\bu}_h&&\quad\text{ a.e.\ in }\Omega\,. \label{prop:stokes_marini.5}
			\end{alignat}
		\end{subequations}
		Due to $\nabla_h(\widehat{\bu}_h-\bu_h)=\mathbf{0}$ a.e.\ in $T$ for all $T\in  \mathcal{T}_h$ (\textit{cf}.\ \eqref{prop:stokes_marini.5}), we have that ${\widehat{\bu}_h-\bu_h\in (\Pbroken{0})^d}$. Thus,  from \eqref{prop:stokes_marini.5}, the discrete integration-by-parts formula \eqref{eq:pi0}, and \eqref{eq:stokes_with_multiplier}, for {every} ${\btau_h\in  \RTNeumann}$, it follows that
		\begin{align*}
			(\widehat{\bu}_h-\bu_h,\diver\btau_h)_{\Omega}& = (\Pi_h \widehat{\bu}_h-\Pi_h\bu_h,\diver\btau_h)_{\Omega}\\&=
			(\overline{\bu}_h , \diver \btau_h)_\Omega
			+(\nabla_h\bu_h  ,\dev\Pi_h\btau_h)_\Omega
			\\&=(\overline{\bu}_h , \diver \btau_h)_\Omega+(\tfrac{1}{\nu}\dev\Pi_h\BT_h-\nabla_h\uhDirichlet,\dev\Pi_h\btau_h)_{\Omega}
			\\&=0\,.
		\end{align*}
    Then, using  that $\diver(\RTNeumann)=(\mathbb{P}^0(\tria))^d$ 
    because $\Gamma_N\neq \partial\Omega$ (by assumption, we have that $\Gamma_D\neq \emptyset$), 
		we find that
		\begin{align*}
			\smash{\widehat{\bu}_h-\bu_h\in (\diver(\RTNeumann))^{\perp_{\BL^2}}}=\{\mathbf{0}\}\,,
		\end{align*}
		\textit{i.e.}, we conclude that $\widehat{\bu}_h=\bu_h$ a.e.\ in $\Omega$.
	\end{proof}  
	
	\subsection{\emph{A priori} error analysis}\label{sec:stokes_apriori}
	
Following ideas from \cite{BarGudKal24,ABKKTorsion}, we perform our \emph{a priori} error analysis in the spirit of an \emph{a posteriori} error analysis at the discrete level resorting
	to the derived discrete duality relations (\textit{cf}.\ Section \ref{sec:stokes_discrete_duality}). This procedure yields an \emph{a priori} error identity, which, to the best of the {authors}' {knowledge}, is new for the incompressible Stokes problem \eqref{eq:stokes_primal} and the associated {discretisation} \eqref{eq:stokes_primal_discrete}.
	
The first step in the procedure in \cite{BarGudKal24,ABKKTorsion} consists in deriving a \textit{`meaningful'} representation of the so-called \emph{discrete primal-dual gap estimator} $\eta_{\textup{gap},h}^2\colon K^h_{\bu} \times K^h_{\BT} \to [0,+\infty)$, for every $\bv_h\in K^h_{\bu}$ and $\btau_h\in K^h_{\BT}$ defined by
	\begin{align}\label{eq:discrete_primal_dual_gap_estimator_stokes}
		\eta_{\textup{gap},h}^2(\bv_h,\btau_h)\coloneqq I_h(\bv_h)-D_h(\btau_h)\,.
	\end{align} 
	Such a \textit{`meaningful'} representation of the discrete primal-dual gap estimator \eqref{eq:discrete_primal_dual_gap_estimator_stokes}, which shows that the latter precisely measures the violation of the discrete convex optimality relation \eqref{eq:stokes_optimality_discrete}, is derived in the following lemma.

	\begin{lemma}\label{lem:discrete_primal_dual_gap_estimator_stokes}
		For every $\bv_h\in K^h_{\bu}$ and $\btau_h\in  K^h_{\BT}$, we have that
		\begin{align*}
			\eta_{\textup{gap},h}^2(\bv_h,\btau_h)&\coloneqq   	\tfrac{\nu}{2} \|\nabla_h \bv_h+
			\nabla_h \smash{\widehat{\bu}_D^h}-\smash{\tfrac{1}{\nu}}\Pi_h \dev \btau_h\|_{\Omega}^2\,.
		\end{align*}
	\end{lemma} 
	
	\begin{proof}
		For every $\bv_h\in K^h_{\bu}$ and $\btau_h\in  K^h_{\BT}$,
		using \eqref{eq:stokes_stress_constraint_discrete}, the discrete integration-by-parts {formula} \eqref{eq:pi0} (since $\btau_h-\BF_h\in \RTNeumann$), \eqref{eq:Neumann_condition_weak_discrete},
		$ \btau_h:(\nabla_h \bv_h+\nabla_h\smash{\widehat{\bu}_D^h})=\dev \btau_h:(\nabla_h \bv_h+\nabla_h\smash{\widehat{\bu}_D^h})$ a.e. in $\Omega$ (since $\textrm{tr}\,\nabla_h (\bv_h+\smash{\widehat{\bu}_D^h})=\textup{div}_h(\bv_h 
		+\smash{\widehat{\bu}_D^h})=0$ a.e. in $\Omega$), and a binomial formula,
		we find that
		\begin{align*}
			I_h(\bv_h)-D_h(\btau_h)&= \tfrac{\nu}{2}\|\nabla_h \bv_h+\nabla_h\smash{\widehat{\bu}_D^h}\|_{\Omega}^2-(\bff_h,\Pi_h \bv_h)_{\Omega}-(\BF_h,\nabla_h \bv_h)_{\Omega}\\&\quad+\tfrac{1}{2\nu}\| \Pi_h \dev \btau_h\|_{\Omega}^2-(\btau_h,\nabla_h\smash{\widehat{\bu}_D^h})_{\Omega}\\&
			= \tfrac{\nu}{2}\| \nabla_h \bv_h+\nabla_h\smash{\widehat{\bu}_D^h}\|_{\Omega}^2+(\textup{div}\,(\btau_h-\BF_h),\Pi_h \bv_h)_{\Omega}-(\BF_h,\nabla_h \bv_h)_{\Omega}\\&\quad+\tfrac{1}{2\nu}\| \Pi_h \dev \btau_h\|_{\Omega}^2-(\btau_h,\nabla_h\smash{\widehat{\bu}_D^h})_{\Omega}
			\\&=\tfrac{\nu}{2}\| \nabla_h \bv_h+\nabla_h\smash{\widehat{\bu}_D^h}\|_{\Omega}^2-(\btau_h-\BF_h,\nabla_h \bv_h)_{\Omega}-(\BF_h,\nabla_h \bv_h)_{\Omega}\\&\quad+\tfrac{1}{2\nu}\| \Pi_h \dev \btau_h\|_{\Omega}^2-(\btau_h,\nabla_h\smash{\widehat{\bu}_D^h})_{\Omega}
			\\&=\tfrac{\nu}{2}\| \nabla_h \bv_h+\nabla_h\smash{\widehat{\bu}_D^h}\|_{\Omega}^2-(\Pi_h\dev \btau_h,\nabla_h \bv_h+\nabla_h\smash{\widehat{\bu}_D^h})_{\Omega}\\&\quad+\tfrac{1}{2\nu}\| \Pi_h \dev \btau_h\|_{\Omega}^2
			\\&= \tfrac{\nu}{2} \|\nabla_h \bv_h+
			\nabla_h \smash{\widehat{\bu}_D^h}-\smash{\tfrac{1}{\nu}}\Pi_h \dev \btau_h\|_{\Omega}^2\,,
		\end{align*}
		which is the claimed representation of the discrete primal-dual gap estimator \eqref{eq:discrete_primal_dual_gap_estimator_stokes}.
	\end{proof}
	
	The second step in the procedure in \cite{BarGudKal24,ABKKTorsion} consists 
	in deriving also \textit{`meaningful’} representations of the \emph{discrete optimal strong convexity measures} for the discrete primal energy functional \eqref{eq:stokes_primal_discrete} at the discrete primal solution $\bu_h \in K^h_{\bu}$, \textit{i.e.}, 
	\hspace{-0.175mm}$\rho_{I_h^{cr}}^2\colon K^h_{\bu} \to [0,+\infty)$, for every ${\bv_h \in K^h_{\bu}}$ {defined} by
	\begin{align}\label{def:discrete_optimal_primal_error}
		\smash{\rhoprimalh(\bv_h)\coloneqq I_h(\bv_h)-I_h(\bu_h)\,,}
	\end{align}
	and for (the negative of) the discrete dual energy functional \eqref{eq:stokes_dual_discrete} at the discrete dual solution $\BT_h\in K^h_{\BT}$, \textit{i.e.}, $\rhodualh\colon \hspace{-0.1em}K^h_{\BT}\hspace{-0.1em}\to\hspace{-0.1em} [0,+\infty)$,  
	for every $\btau_h\hspace{-0.1em}\in\hspace{-0.1em} K^h_{\BT}$~defined~by
	\begin{align}\label{def:discrete_optimal_dual_error}
		\smash{\rhodualh(\btau_h)\coloneqq- D_h(\btau_h)+D_h(\BT_h)\,.}
	\end{align}
	the sum of which, called \emph{discrete primal-dual total error}, \textit{i.e.}, $\rho_{\textup{tot},h}^2\colon K^h_{\bu}\times K^h_{\BT}\to [0,+\infty)$, for every $\bv_h\in K^h_{\bu}$ and $\btau_h\in K^h_{\BT}$ defined by
	\begin{align}\label{eq:discrete_primal_dual_error}
		\smash{\rho_{\textup{tot},h}^2(\bv_h,\btau_h)\coloneqq \rho_{I_h}^2(\bv_h)+\rho_{-D_h}^2(\btau_h)\,,}
	\end{align}
	represents a \emph{`natural'} error quantity in the discrete primal-dual gap identity (\textit{cf}.\ Theorem \ref{thm:discrete_prager_synge_identity_stokes}).
	
	Such \textit{`meaningful'} representations of the discrete optimal strong convexity measures \eqref{def:discrete_optimal_primal_error} and \eqref{def:discrete_optimal_dual_error} are derived in the following lemma.
	
	\begin{lemma}\label{lem:discrete_strong_convexity_measures_stokes}
		The following statements apply:
		\begin{itemize}[noitemsep,topsep=2pt,leftmargin=!,labelwidth=\widthof{(ii)}]
			\item[(i)]\hypertarget{lem:discrete_strong_convexity_measures_stokes.i}{} For every $\bv_h\in K^h_{\bu}$, we have that 
			\begin{align*}
				\smash{\rhoprimalh(\bv_h)=\tfrac{\nu}{2}\|\nabla_h \bv_h-\nabla_h \bu_h\|_{\Omega}^2\,;}
			\end{align*}
			\item[(ii)]\hypertarget{lem:discrete_strong_convexity_measures_stokes.ii}{} For every $\btau_h\in K^h_{\BT}$, we have that 
			\begin{align*}
				\smash{\rhodualh(\btau_h)=\tfrac{1}{2\nu}\|\Pi_h \dev \btau_h-\Pi_h \dev \BT_h\|_{\Omega}^2\,.}
			\end{align*}
		\end{itemize}
	\end{lemma}

	\begin{proof}[of Lemma \ref{lem:discrete_strong_convexity_measures_stokes}] 
		\emph{ad (\hyperlink{lem:discrete_strong_convexity_measures_stokes.i}{i}).} Noting that 
		$K^h_{\bu}$ is a Hilbert space and that the restricted discrete primal energy functional $I_h\colon K^h_{\bu}\to \R$ is a smooth quadratic functional, for every $\bv_h\in K^h_{\bu}$, a Taylor expansion yields that
		\begin{align*}
			I_h(\bv_h) - I_h(\bu_h)
			&= \langle \mathrm{D}I(\bu_h),\bv_h - \bu_h\rangle_{K^h_{\bu}}+
			\tfrac{1}{2}\langle \mathrm{D}^2I(\bu_h)(\bv_h - \bu_h), \bv_h - \bu_h\rangle_{K^h_{\bu}}
			\\&= \tfrac{\nu}{2} \|\nabla_h\bv_h - \nabla_h \bu_h\|_{\Omega}^2\,,
		\end{align*} 
		which is the claimed representation of the discrete optimal strong convexity measure \eqref{def:discrete_optimal_primal_error}.

		\emph{ad (\hyperlink{lem:discrete_strong_convexity_measures_stokes.ii}{ii}).} For every $\btau_h\in K^h_{\BT}$, using that $(\dev\BT_h - \dev\btau_h, \nabla_h \bu_h)_\Omega=(\BT_h - \btau_h, \nabla_h \bu_h)_\Omega =0$
		(which follows from $\textup{tr}(\nabla_h \bu_h)\hspace{-0.15em}=\hspace{-0.15em}\textup{div}_h\bu_h\hspace{-0.15em}=\hspace{-0.15em}0$~a.e.~in~$\Omega$, the discrete integration-by-parts {formula} \eqref{eq:pi0}, and \eqref{eq:stokes_constraints_discrete}), the 
		discrete convex optimality condition \eqref{eq:stokes_optimality_discrete}, and a binomial formula, we {get}
		\begin{align*}
			-D_h(\btau_h)+D_h(\BT_h)
            &=\tfrac{1}{2\nu}\|\Pi_h\dev \btau_h\|_{\Omega}^2-\tfrac{1}{2\nu}\|\Pi_h\dev \BT_h\|_{\Omega}^2
			\\& \quad
			+(\dev\BT_h- \dev\btau_h,\nabla_h(\uhDirichlet+\bu_h))_{\Omega}
			\\&=\tfrac{1}{2\nu}\|\Pi_h\dev \btau_h\|_{\Omega}^2-\tfrac{1}{2\nu}\|\Pi_h\dev \BT_h\|_{\Omega}^2
			\\&\quad+\tfrac{1}{\nu}(\Pi_h\dev(\BT_h- \btau_h),\Pi_h\dev \BT_h)_\Omega
			\\&=\tfrac{1}{2\nu}\|\Pi_h\dev \btau_h\|_{\Omega}^2+\tfrac{1}{2\nu}\|\Pi_h\dev \BT_h\|_{\Omega}^2\\&\quad-\tfrac{1}{\nu}(\Pi_h\dev\btau_h,\Pi_h\dev \BT_h)_\Omega
			\\&=\tfrac{1}{2\nu}\|\Pi_h\dev  \btau_h-\Pi_h\dev  \BT_h\|_{\Omega}^2\,,
		\end{align*}
		which is the claimed representation of the discrete optimal strong convexity measure \eqref{def:discrete_optimal_dual_error}.
	\end{proof}

	In the third step in the procedure in \cite{BarGudKal24,ABKKTorsion}, from the definitions \eqref{eq:discrete_primal_dual_gap_estimator_stokes}--\eqref{eq:discrete_primal_dual_error} and the discrete strong duality relation \eqref{eq:stokes_strong_duality_discrete}, it follows that an \emph{a posteriori} error identity on the discrete level, called \textit{discrete primal-dual gap identity}, applies.

	\begin{theorem}[Discrete primal-dual gap identity]\label{thm:discrete_prager_synge_identity_stokes}
		For every $\bv_h\in  K^h_{\bu}$ and $\btau_h\in K^h_{\BT}$, there holds
		\begin{equation}\label{eq:discrete_praguer_synge_identity_stokes}
			\smash{\rhototh(\bv_h,\btau_h)=\eta_{\textup{gap},h}^2(\bv_h,\btau_h)\,.}
		\end{equation}
	\end{theorem}
 
	\begin{proof}\let\qed\relax
		We combine the definitions \eqref{eq:discrete_primal_dual_gap_estimator_stokes}--\eqref{eq:discrete_primal_dual_error} and the discrete strong duality relation \eqref{eq:stokes_strong_duality_discrete}.
\end{proof} 

With the help of the discrete primal-dual gap identity \eqref{eq:discrete_praguer_synge_identity_stokes}, it is possible to derive an \emph{a priori}  error identity characterising the distance between the discrete solutions and the interpolants of the exact solutions in terms of a dual interpolation error. 
In particular, this leads to a quasi-optimal error estimate~that, in particular, holds for minimal regularity solutions of the incompressible Stokes problem. 

\begin{theorem}\label{thm:apriori_stokes}
	Let $\bff_h\coloneqq\Pi_h \bff\in  (\Pbroken{0})^d$, $\BF_h \coloneqq \Pi_h \BF\in (\Pbroken{0})^{d\times d}$, and $\widehat{\bu}_D^h \coloneqq\mathcal{I}^{cr}_h\widehat{\bu}_D\in V^h$. Then,
	the following statements apply:
	\begin{enumerate}[noitemsep,topsep=2pt,leftmargin=!,labelwidth=\widthof{(ii)}]
		\item[(i)]\hypertarget{thm:apriori_stokes.i}{}    There  holds the  \emph{a priori}  error identity
		\begin{equation}\label{eq:stotkes_apriori_identity}
			\hspace*{-4mm}\begin{aligned}
				\tfrac{\nu}{2}\|\nabla_h \mathcal{I}_h^{cr}\bu-\nabla_h \bu_h\|_{\Omega}^2&+\tfrac{1}{2\nu}\|\Pi_h\dev (\mathcal{I}_h^{rt}(\BT-\BF))-\Pi_h\dev (\BT_h-\BF_h)\|_{\Omega}^2\\&= \tfrac{1}{2\nu} \|\Pi_h \dev (\BT-\BF)-\Pi_h \dev \mathcal{I}_h^{rt}(\BT-\BF)\|_{\Omega}^2\,;
			\end{aligned}\hspace*{-1mm}
		\end{equation}
		
		\item[(ii)]\hypertarget{thm:apriori_stokes.ii}{} There  holds the quasi-optimal error estimate
		\begin{equation}\label{eq:stotkes_quasi-optimality}
			\begin{aligned}
				&\nu\|\nabla (\bu + \uDirichlet)- \nabla_h (\bu_h+\uhDirichlet)\|_{\Omega}^2+\tfrac{1}{\nu}\|\dev \BT - \Pi_h\dev \BT_h\|_{\Omega}^2\\ 
				&\leq 
				6\inf_{\bv_h \in K^h_{\bu}}{\big\{\nu\|\nabla(\bu+\uDirichlet) - \nabla_h(\bv_h+\uhDirichlet)\|_{\Omega}^2\big\}}
				\\&\quad +
				6\inf_{\btau_h \in K^h_{\BT}}{\big\{\tfrac{1}{\nu}\|\dev\BT - \dev\Pi_h \btau_h\|_{\Omega}^2\big\}}\,.
			\end{aligned}
		\end{equation}
	\end{enumerate}
\end{theorem}

\begin{remark}\label{rmk:quasi-optimal_stokes}
	To the best of the author's knowledge, Corollary \ref{thm:apriori_stokes} presents the first minimal regularity estimate for a Crouzeix--Raviart approximation of the Stokes problem with explicit constants, that considers a general load $\bff^*\in \mathbf{H}^{-1}_D(\Omega)$, and inhomogeneous mixed boundary conditions. 
	The only other quasi-optimality result is \cite[Thm.\ 4.2]{VZ.2019}, where the authors obtain the estimate
	\begin{equation}\label{eq:stotkes_quasi-optimality_2}
		\|\nabla \bu-\nabla_h\bu_h\|_{\Omega}
		\leq c_{\mathrm{qo}}
		\inf_{\bv_h\in \CRDirichlet}{\big\{\|\nabla\bu-\nabla_h\bv_h\|_{\Omega}\big\}}\,.
	\end{equation}
	(We do not deem results, \textit{e.g.}, from \cite{Lin.2014,LMNN.2018} to be genuinely of minimal regularity, since the oscillation of, \textit{e.g.}, $\Delta \bu$ in $\BL^2(\Omega)$ is assumed to be under control.)
	While the estimate \eqref{eq:stotkes_quasi-optimality_2} does not involve the dual best-approximation error (unlike \eqref{eq:stotkes_quasi-optimality}), it requires the implementation of a so-called smoothing operator $E_h \colon \CRDirichlet \to \BH_D^1(\Omega)$ on the right-hand side, \textit{i.e.}, $\langle \bff^*, E_h \bv_h \rangle_\Omega$ (which involves solving Stokes problems on macro-elements), and the constant $c_{\mathrm{qo}}>0$ is not explicit and may depend, \textit{e.g.}, on the shape-regularity. 
	Furthermore, their analysis was carried out for homogeneous Dirichlet boundary conditions only.
	However, the method from \cite{VZ.2019} is pressure robust.
\end{remark}

\begin{proof}[of Theorem \ref{thm:apriori_stokes}]
	
	\emph{ad (\hyperlink{thm:apriori_stokes.i}{i}).} 
	Recalling Lemma \ref{lem:discrete_primal_dual_gap_estimator_stokes} and Lemma \ref{lem:discrete_strong_convexity_measures_stokes}, the discrete primal-dual gap identity \eqref{eq:discrete_praguer_synge_identity_stokes}, for every $\bv_h\in K^h_{\bu}$ and $\btau_h\in K^h_{\BT}$, reads explicitly
	\begin{align}\label{eq:discrete_prager_synge_identity_stokes_2}
		\tfrac{\nu}{2}\|\nabla_h \bu_h -\nabla_h \bv_h\|_{\Omega}^2&+\tfrac{1}{2\nu}\|\Pi_h\dev\BT_h - \Pi_h\dev \btau_h\|_{\Omega}^2
		\\&= \tfrac{\nu}{2} \|\nabla_h(\bv_h + \uhDirichlet) - \tfrac{1}{\nu}\Pi_h \dev\btau_h\|_{\Omega}^2\,.
	\end{align}
	Using the divergence-preservation property \eqref{eq:ICR_preservation_div}, we see that $\mathcal{I}^{cr}_h \bu \in K^h_{\bu}$.
	Similarly, from the divergence-preservation property  \eqref{eq:IRT_preservation_div}, we obtain $\textup{div}\,(\mathcal{I}_h^{rt}(\BT-\BF))\hspace{-0.15em}=\hspace{-0.15em}
	-\bff_h$ a.e.\ in $\Omega$, \textit{i.e.}, we see that $\mathcal{I}_h^{rt}(\BT-\BF)+\BF_h\in K^h_{\BT}$. Then, 
	for  $\bv_h = \mathcal{I}^{cr}_h\bu\in K^h_{\bu}$ and $\btau_h = \mathcal{I}^{rt}_h(\BT-\BF) + \BF_h\in K^h_{\BT}$ in \eqref{eq:discrete_prager_synge_identity_stokes_2}, by the gradient-preservation property \eqref{eq:ICR_preservation_grad} and the convex optimality condition \eqref{eq:stokes_optimality_continuous_3},~we~find~that
	\begin{align*}
		\tfrac{\nu}{2}\|\nabla_h \mathcal{I}_h^{cr}\bu-\nabla_h \bu_h\|_{\Omega}^2&+\tfrac{1}{2\nu}\|\Pi_h\dev  \mathcal{I}_h^{rt} (\BT-\BF)-\Pi_h\dev  (\BT_h-\BF_h)\|_{\Omega}^2\\&= \tfrac{\nu}{2} \|\nabla_h\mathcal{I}_h^{cr}\bu+\nabla_h\widehat{\bu}_D^h - \tfrac{1}{\nu}\Pi_h \dev (\mathcal{I}_h^{rt} (\BT - \BF) + \BF_h)\|_{\Omega}^2
		\\&= \tfrac{1}{2\nu} \|\Pi_h\dev(\BT -\BF)- \Pi_h \dev (\mathcal{I}_h^{rt} (\BT - \BF) )\|_{\Omega}^2\,,
	\end{align*}  
	which is the claimed \emph{a priori} error identity \eqref{eq:stotkes_apriori_identity}.
	
	\emph{ad (\hyperlink{thm:apriori_stokes.ii}{ii}).} 
	Let $\bv_h\in K^h_{\bu}$ and $\btau_h\in K^h_{\BT}$ be fixed, but arbitrary. Then, using twice the discrete primal-dual gap identity \eqref{eq:discrete_praguer_synge_identity_stokes} and once the convex optimality condition \eqref{eq:stokes_optimality_continuous_3},  we find that 
	\begin{align}\label{eq:discrete_prager_synge_identity_stokes_3}
		\begin{aligned} 
			&\tfrac{\nu}{2}\|\nabla (\bu + \uDirichlet)- \nabla_h (\bu_h+\uhDirichlet)\|_{\Omega}^2+\tfrac{1}{2\nu}\|\dev \BT - \Pi_h\dev \BT_h\|_{\Omega}^2\\ 
			&\quad\leq\tfrac{\nu}{2} \left\{ \|\nabla(\bu+\uDirichlet)- \nabla_h(\bv_h+\uhDirichlet)\|_{\Omega} + 
			\|\nabla_h \bu_h - \nabla_h \bv_h\|_{\Omega}\right\}^2 \\
			&\qquad+ \tfrac{1}{2\nu} \left\{\|\dev\BT - \dev\Pi_h\btau_h\|_{\Omega} + \|\dev\Pi_h\btau_h - \dev\Pi_h\BT_h\|_{\Omega} \right\}^2 \\
			&\quad\leq \nu \|\nabla(\bu+\uDirichlet)- \nabla_h(\bv_h+\uhDirichlet)\|_{\Omega}^2
			+ \tfrac{1}{\nu} \|\dev\BT - \dev\Pi_h\btau_h\|_{\Omega}^2  \\
			&\qquad + \nu \|\nabla_h(\bv_h + \uhDirichlet) - \nabla(\bu+\uDirichlet) +\tfrac{1}{\nu}\dev\BT- \tfrac{1}{\nu}\dev\Pi_h\btau_h\|_{\Omega}^2 \\
			&\quad\leq 3\nu \|\nabla(\bu+\uDirichlet)- \nabla_h(\bv_h+\uhDirichlet)\|_{\Omega}^2
			+ \tfrac{3}{\nu} \|\dev\BT - \dev\Pi_h\btau_h\|_{\Omega}^2\,.
		\end{aligned}
	\end{align}
	Then, taking infima with respect to $\bv_h\in K^h_{\bu}$ and $\btau_h\in K^h_{\BT}$, respectively,~in~\eqref{eq:discrete_prager_synge_identity_stokes_3}, we conclude that the claimed quasi-optimal error estimate \eqref{eq:stotkes_quasi-optimality} applies. 
\end{proof}

Combining Theorem \ref{thm:apriori_stokes} with the approximation properties \eqref{eq:cr_best_approximation}~and~\eqref{eq:rt_best_approximation}~it~is, then, possible to extract the optimal error decay rate $s\in (0,1]$, under appropriate regularity {assumptions}, \textit{i.e.},
\begin{equation}
	\smash{\tfrac{\nu}{2}\|\nabla (\bu + \uDirichlet)- \nabla_h (\bu_h+\uhDirichlet)\|_{\Omega}^2+\tfrac{1}{2\nu}\|\dev \BT - \Pi_h\dev \BT_h\|_{\Omega}^2
		= \mathcal{O}(h_\mathcal{T}^{2s})\,.}
\end{equation}

We further would  like to point out that this approach to obtain quasi-optimal error estimates for Crouzeix--Raviart approximations seems to be new, even for the Poisson problem; in particular, for a general load $f^*\in H^{-1}_D(\Omega)$ and inhomogeneous mixed boundary conditions. For the sake of completeness, we outline this approach in this case, whose proof follows an analogous argument:

Let $f^*\in H^{-1}_D(\Omega)$ be a given load, which similarly to Lemma \ref{lem:dual_representation}, for every $v\in H^1_D(\Omega)$, can be represented as $\langle f^*,v\rangle_\Omega = (f,v)_\Omega + (\bff,\nabla v)_\Omega$,  for a function $f\in L^2(\Omega)$ and a vector field $\bff\in \BL^2(\Omega)$, and let $u_D\in H^{\frac{1}{2}}(\Gamma_D)$ be a given Dirichlet boundary datum, for which 
a lift to the interior $\Omega$ exists, \textit{i.e.}, 
a function 
$\widehat{u}_D\in H^1(\Omega)$ with $\widehat{u}_D=u_D$ q.e. in $\Gamma_D$.
Then, the \textit{Poisson (minimisation) problem} seeks to minimise the primal energy functional $I\colon H^1_D(\Omega) \to \R$, for every $v\in H^1_D(\Omega) $ defined by
\begin{equation}\label{def:poisson_primal}
	\smash{I(v) \coloneqq \tfrac{1}{2}\|\nabla(v+\widehat{u}_D)\|_{\Omega}^2 - \langle f^* ,v\rangle_\Omega\,.}
\end{equation}
A (Lagrangian) dual problem (in the sense of  \cite[Chap. 49]{ZeiIII}) 
is given via the maximisation 
of the dual energy functional $D\colon \bff+H_N(\textup{div};\Omega) \to \R\cup\{-\infty\}$, for every $\br\in \bff+H_N(\textup{div};\Omega) $ defined by
\begin{equation}\label{eq:dual_energy_laplace}
	\smash{D(\br) \coloneqq -\tfrac{1}{2}\|\br\|_{\Omega}^2 
		+ (\br,\nabla \widehat{u}_D)
		- \characteristic{-f}^{\Omega}(\diver(\br- \bff))\,.}
\end{equation}
By the direct method in the calculus of variations, the primal energy functional \eqref{def:poisson_primal} admits a unique minimiser $u\in H^1_D(\Omega)$ and the dual energy~functional~\eqref{eq:dual_energy_laplace} a unique maximiser $\bq\in \bff+H_N(\textup{div};\Omega)$, related via the convex optimality relation $\bq= \nabla (u+\widehat{u}_D)$ a.e.\ in $\Omega$.

Next, with a slight abuse of notation, let $\CRDirichlet$ and $\RTNeumann$ now be scalar- and vector-valued spaces. Moreover, 
let $f_h\hspace{-0.1em}\in\hspace{-0.1em} \Pbroken{0}$, $\bff_h\hspace{-0.1em}\in\hspace{-0.1em} (\Pbroken{0})^d$, and $\widehat{u}_D^h \hspace{-0.1em}\in \hspace{-0.1em}V^h$ be approximations of $f\hspace{-0.1em}\in\hspace{-0.1em} L^2(\Omega)$, ${\bff\hspace{-0.1em}\in\hspace{-0.1em} \BL^2(\Omega)}$, and $\widehat{u}_D\in H^1(\Omega)$. Then, the \textit{discrete Poisson (minimisation) problem} seeks to minimise the discrete primal energy functional $I_h\colon \CRDirichlet \to \R$, for every 
$v_h\in \CRDirichlet$ defined by
\begin{align}\label{def:poisson_primal_discrete}
	\smash{I_h(v_h) \coloneqq \tfrac{1}{2}\|\nabla_h(v_h+\widehat{u}^h_D)\|_{\Omega}^2 
		- (f_h,\Pi_hv_h)_\Omega - (\bff_h,\nabla_hv_h)_\Omega \,.}
\end{align}
A (Lagrangian) dual problem (in the sense of  \cite[Chap. 49]{ZeiIII}) 
is given via the maximisation of the discrete dual energy functional
$D_h\colon \bff_h + \RTNeumann \to \R\cup\{-\infty\}$, for every  $\br_h\in\bff_h + \RTNeumann$ defined by
\begin{align}\label{eq:dual_energy_laplace_discrete}
	\smash{D_h(\br_h) \coloneqq -\tfrac{1}{2}\|\Pi_h\br_h\|_{\Omega}^2 
		+ ( \Pi_h\br_h,\nabla_h \widehat{u}_D)
		- \characteristic{-f_h}^{\Omega}(\diver(\br_h- \bff_h))\,.}
\end{align}
By the direct method in the calculus of variations, the discrete primal energy functional \eqref{def:poisson_primal_discrete} admits a unique minimiser $u_h\hspace{-0.1em}\in\hspace{-0.1em} \CRDirichlet$ and the discrete dual energy functional \eqref{eq:dual_energy_laplace_discrete} a unique {maximiser} $\bq_h\in \bff_h+\RTNeumann$, related via the discrete convex optimality relation $\Pi_h \bq_h= \nabla_h(u_h + \widehat{u}^h_D)$ a.e. in $\Omega$.

Analogously to Theorem \ref{thm:apriori_stokes}, we can derive a quasi-optimal error estimate.
\begin{corollary}\label{thm:apriori_laplace} Let $f_h\coloneqq\Pi_h f\in  \Pbroken{0}$, $\bff_h \coloneqq \Pi_h \bff\in (\Pbroken{0})^d$, and $\widehat{u}_D^h \coloneqq\mathcal{I}^{cr}_h\widehat{u}_D\in V^h$. Then,
	the following statements apply:
	
	\begin{enumerate}[noitemsep,topsep=2pt,leftmargin=!,labelwidth=\widthof{(ii)}]
		\item[(i)]   There holds the \emph{a priori} error identity
		\begin{equation*}
			\begin{aligned}
				\|\nabla_h \mathcal{I}_h^{cr}u-\nabla_h u_h\|_{\Omega}^2 
        &+\|\Pi_h (\mathcal{I}_h^{rt}(\bq-\bff))-\Pi_h(\bq_h-\bff_h)\|_{\Omega}^2
      \\&
      = \|\Pi_h (\bq-\bff)-\Pi_h \mathcal{I}_h^{rt}(\bq-\bff)\|_{\Omega}^2\,;
			\end{aligned}
		\end{equation*}
		
		\item[(ii)] There holds the quasi-optimal error estimate
		\begin{equation*}
			\begin{aligned}
        &\|\nabla (u + \widehat{u}_D) - \nabla_h (u_h+\widehat{u}^h_D)\|_{\Omega}^2
        +\|\bq - \Pi_h\bq_h\|_{\Omega}^2 
      \\& \leq 
				6\inf_{\bv_h \in \CRDirichlet}{\big\{\|\nabla(u+\widehat{u}_D) - \nabla_h(v_h+\widehat{u}^h_D)\|_{\Omega}^2\big\}}
				+6\inf_{\substack{\br_h \in K_{\bq}^h}}{\big\{\|\bq - \Pi_h \br_h\|_{\Omega}^2\big\}}\,,\\[-1mm]
			\end{aligned}
		\end{equation*}
		where  $$\smash{K_{\bq}^h\coloneqq\textup{dom}(-D_h)\coloneqq\{\br_h\in \bff_h+\RTNeumann\mid \diver(\br_h-\bff_h)=-f_h\text{ a.e.\ in }\Omega\}}\,.$$
	\end{enumerate}
\end{corollary}

\begin{remark}\label{rmk:quasi-optimal_laplace}
	To the best of the authors' knowledge, Corollary \ref{thm:apriori_laplace} presents the first minimal regularity estimate in the context of Crouzeix--Raviart approximations of the Poisson problem with explicit constants, that considers a general load $f^*\in H^{-1}_D(\Omega)$, and inhomogeneous mixed boundary conditions. 
	Within the framework of convex duality theory, there exist results relating the discrete solution $(\bq_h,u_h)\in K_{\bq}^h \times \CRDirichlet$ to the interpolant of the exact solution (cf. \cite[eq. (1.34)]{BK.2024}), but do not cover general loads, and are generally written for {homogeneous} {boundary} {conditions}.
	The only other quasi-optimality result is \cite[Thm. 3.4]{VZ.2019.II}, where the authors obtain the {estimate}\vspace{-0.5mm}
	\begin{equation*}
		\|\nabla u-\nabla_hu_h\|_{\Omega}
		\leq c_{\mathrm{qo}}
		\inf_{v_h\in \CRDirichlet}{\big\{\|\nabla u-\nabla_hv_h\|_{\Omega}\big\}}\,.\\[-1mm]
	\end{equation*}
	While this estimate does not involve the dual best-approximation error (unlike Corollary \ref{thm:apriori_laplace}), it requires the implementation of a so-called smoothing operator $E_h \colon\hspace{-0.175em} \CRDirichlet \hspace{-0.175em}\to \hspace{-0.175em}H_D^1(\Omega)$ on the {right-hand} side, \textit{i.e.}, $\langle f^*, E_h v \rangle_\Omega$, and the constant $c_{\mathrm{qo}}\hspace{-0.15em}>\hspace{-0.15em}0$ is not explicit and can depend, \textit{e.g.}, on the {shape-regularity}. 
	Furthermore, their analysis was carried out for homogeneous Dirichlet boundary~conditions only.
\end{remark}

\subsection{\emph{A posteriori} error analysis}\label{sec:aposteriori_stokes} 

\hspace{5mm}After having performed an \emph{a priori} error analysis in the spirit of an \emph{a posteriori} error analysis at the discrete level in Section \ref{sec:stokes_apriori}, in this section,~eventually, we perform an \emph{a posteriori} error analysis. In doing so, 
we again follow ideas from \cite{BarGudKal24,ABKKTorsion} and, analogous to the \textit{a priori error} analysis at the discrete level, proceed in three steps.

The first step, again, consists in deriving a  \textit{`meaningful'} representation of the so-called
\emph{primal-dual gap estimator} ${\eta^2_{\textup{gap}}\colon K_{\bu}\times K_{\BT}\to \mathbb{R}}$, for every $\bv\in K_{\bu}$ and $\btau\in K_{\BT}$ defined by 
\begin{align}\label{eq:primal-dual.1}
	\begin{aligned}
		\gap(\bv,\btau)&\coloneqq I(\bv)-D(\btau)\,.
	\end{aligned}
\end{align}
In the context of \emph{a posteriori} error analysis, the central requirement of a representation of the primal-dual gap estimator \eqref{eq:primal-dual.1} being \textit{`meaningful'} consists in an integral representation with pointwise non-negative integrand, which allows the derivation of local mesh-refinement indicators, on the basis of which local mesh-refinement can be performed.
Such a \textit{`meaningful'} representation of the  primal-dual gap estimator \eqref{eq:primal-dual.1}, which is analogous to the discrete setting in Subsection \ref{sec:stokes_apriori} (\textit{cf}.\ Lemma \ref{lem:discrete_primal_dual_gap_estimator_stokes}) and shows that the latter precisely measures the violation of the convex optimality relation \eqref{eq:stokes_optimality_continuous_3}, is derived~in~the~\mbox{following}~lemma.

\begin{lemma}\label{lem:primal_dual_gap_estimator_stokes}
	For every $\bv\in K_{\bu}$ and $\btau\in K_{\BT}$, we have that
	\begin{equation*}
		\eta_{\textup{gap}}^2(\bv,\btau)= \tfrac{\nu}{2}\|\nabla (\bv+\uDirichlet)-\smash{\tfrac{1}{\nu}}\dev \btau\|_{\Omega}^2\,. 
	\end{equation*}
\end{lemma}


\begin{proof}
	For every $\bv\in K_{\bu}$ and $\btau\in K_{\BT}$, using \eqref{eq:stokes_stress_constraint}, the continuous integration-by-parts formula (since $\btau-\BF\in \HdivN$), \eqref{eq:Neumann_condition_weak},
	$ \btau:(\nabla \bv+\nabla\smash{\widehat{\bu}_D})=\dev \btau:(\nabla \bv+\nabla\smash{\widehat{\bu}_D})$ a.e.\ in $\Omega$ (since $\textrm{tr}\,\nabla (\bv+\smash{\widehat{\bu}_D})=\textup{div}\,(\bv 
	+\smash{\widehat{\bu}_D})=0$ a.e. in $\Omega$), and a binomial formula,
	we find that
	\begin{align*}
		I(\bv)-D(\btau)&= \tfrac{1}{2}\| \nabla (\bv+\uDirichlet)\|_{\Omega}^2-(\bff,\bv)_{\Omega}-(\BF,\nabla\bv)_{\Omega}\\&\quad+\tfrac{1}{2}\|\dev \btau\|_{\Omega}^2-(\dev\btau,\nabla\uDirichlet)_{\Omega}\\&
		= \tfrac{\nu}{2}\| \nabla (\bv+\uDirichlet)\|_{\Omega}^2+(\textup{div}\,(\btau-\BF),\bv)_{\Omega}-(\BF,\nabla\bv)_{\Omega}\\&\quad+\tfrac{1}{2\nu}\| \dev \btau\|_{\Omega}^2-(\dev\btau,\nabla\uDirichlet)_{\Omega}
		\\&
		= \tfrac{\nu}{2}\| \nabla (\bv+\uDirichlet)\|_{\Omega}^2-(\dev \btau,\nabla (\bv+\uDirichlet))_{\Omega}+\tfrac{1}{2\nu}\|\dev \btau \|_{\Omega}^2
		\\&
		= \tfrac{\nu}{2}\| \nabla (\bv+\uDirichlet)-\smash{\tfrac{1
			}{\nu}}\dev \btau\|_{\Omega}^2 \,,
	\end{align*}
	which is the claimed representation of the primal-dual gap estimator \eqref{eq:primal-dual.1}.
\end{proof}

The second step, again, consists in deriving \textit{`meaningful'} representations of the 
\emph{optimal strong convexity measures} for the primal energy functional~\eqref{eq:stokes_primal}~at the primal solution $\bu\hspace{-0.1em}\in\hspace{-0.1em} K_{\bu}$, \textit{i.e.}, 
$\rho_I^2\colon \hspace{-0.1em}K_{\bu}\hspace{-0.1em}\to\hspace{-0.1em} [0,+\infty)$, for~every~${\bv\hspace{-0.1em}\in \hspace{-0.1em}K_{\bu}}$~\mbox{defined}~by
\begin{align}\label{def:optimal_primal_error}
	\rho_I^2(\bv)\coloneqq I(\bv)-I(\bu)\,,
\end{align}
and for (the negative of) the dual energy functional \eqref{eq:stokes_dual} at the dual solution $\BT\in K_{\BT}$, \textit{i.e.}, $\rho_{-D}^2\colon K_{\BT}\to [0,+\infty)$,  
for every ${\btau\in K_{\BT}}$ defined by
\begin{align}\label{def:optimal_dual_error}
	\rho_{-D}^2(\btau)\coloneqq- D(\btau)+D(\BT)\,,
\end{align}
the sum of which, called \emph{primal-dual total error}, \textit{i.e.}, 	$\rho_{\textup{tot}}^2\colon K_{\bu}\times K_{\BT}\to [0,+\infty)$, for every $\bv\in K_{\bu}$ and $\btau\in K_{\BT}$ defined by 
\begin{align}\label{def:primal_dual_total_error}
	\rho_{\textup{tot}}^2(\bv,\btau)\coloneqq \rho_I^2(\bv)+\rho_{-D}^2(\btau)\,,
\end{align}
again, represents a \emph{`natural'} error quantity in the primal-dual gap identity (\textit{cf}.\ Theorem \ref{thm:prager_synge_identity_stokes}).\newpage

Such \textit{`meaningful'} representations of the optimal strong convexity measures \eqref{def:optimal_primal_error} and \eqref{def:optimal_dual_error}, which are analogous to the discrete setting in Subsection \ref{sec:stokes_apriori} (\textit{cf}.\ Lemma \ref{lem:discrete_strong_convexity_measures_stokes}), are derived in the following lemma.

\begin{lemma}
	\label{lem:strong_convexity_measures_stokes}
	The following statements apply:
	\begin{itemize}[noitemsep,topsep=2pt,leftmargin=!,labelwidth=\widthof{(ii)}]
		\item[(i)]\hypertarget{lem:strong_convexity_measures_stokes.i}{} For every $\bv\in K_{\bu}$, we have that
		\begin{align*}
			\rho_I^2(\bv)=\tfrac{\nu}{2}\|\nabla \bv-\nabla \bu\|_{\Omega}^2\,;
		\end{align*}
		\item[(ii)]\hypertarget{lem:strong_convexity_measures_stokes.ii}{} For every $\btau\in K_{\BT}$, we have that
		\begin{align*}
			\rho_{-D}^2(\btau)=\tfrac{1}{2\nu}\|\dev \btau-\dev \BT\|_{\Omega}^2\,.
		\end{align*}
	\end{itemize}
\end{lemma}

\begin{proof}
	\emph{ad (\hyperlink{lem:strong_convexity_measures_stokes.i}{i}).} 
  The proof is analogous to that of Lemma \ref{lem:discrete_strong_convexity_measures_stokes}(\hyperlink{lem:discrete_strong_convexity_measures_stokes.i}{i}).
%
	
	\emph{ad (\hyperlink{lem:strong_convexity_measures_stokes.ii}{ii}).} For every $\btau\in K_{\BT}$, using that $(\dev\BT - \dev\btau, \nabla \bu)_\Omega=(\BT - \btau, \nabla \bu)_\Omega =0$
	(which follows from $\textup{tr}(\nabla \bu)=\textup{div}\,\bu=0$ a.e.\ in $\Omega$, the continuous integration-by-parts formula, and \eqref{eq:stokes_constraints}), the 
	convex optimality condition \eqref{eq:stokes_optimality_continuous_3}, and a binomial formula, we find that
	\begin{align*}
		-D(\btau)+D(\BT)
		&=\tfrac{1}{2\nu}\|\dev \btau\|_{\Omega}^2-\tfrac{1}{2\nu}\|\dev \BT\|_{\Omega}^2+(\dev \BT-\dev \btau,\nabla\bu+\nabla\uDirichlet)_{\Omega}
		\\&=\tfrac{1}{2\nu}\|\dev \btau\|_{\Omega}^2-\tfrac{1}{2\nu}\|\dev \BT\|_{\Omega}^2+\tfrac{1}{\nu}(\dev \BT-\dev \btau,\dev \BT)_{\Omega}
		\\&=\tfrac{1}{2\nu}\|\dev \btau\|_{\Omega}^2-
		\tfrac{1}{\nu}(\dev \btau,\dev \BT)_{\Omega}+\tfrac{1}{2\nu}\|\dev \BT\|_{\Omega}^2
		\\&=\tfrac{1}{2\nu}\|\dev \btau-\dev \BT\|_{\Omega}^2\,,
	\end{align*} 
	which is the claimed representation of the optimal strong convexity measure \eqref{def:optimal_dual_error}.
\end{proof}

In the third step, from the definitions \eqref{eq:primal-dual.1}--\eqref{def:primal_dual_total_error} and the  strong duality relation \eqref{eq:stokes_strong_duality}, it follows that an \emph{a posteriori} error identity at the continuous level, called \textit{primal-dual gap identity}, applies.

\begin{theorem}[Primal-dual gap identity]\label{thm:prager_synge_identity_stokes}
	For every $\bv\in K_{\bu}$ and $\btau\in K_{\BT}$, there holds
	\begin{equation}\label{eq:prager_synge_identity_stokes}
		\rho_{\textup{tot}}^2(\bv,\btau)
		=\eta_{\textup{gap}}^2(\bv,\btau)\,.
	\end{equation}
\end{theorem}

\begin{proof} We combine the definitions \eqref{eq:primal-dual.1}, \eqref{def:optimal_primal_error}--\eqref{def:primal_dual_total_error}, and the strong duality relation \eqref{eq:stokes_strong_duality}. 
\end{proof} 

\begin{remark} 
  The identity \eqref{eq:prager_synge_identity_stokes} was previously derived in \cite{Rep.2002,Rep.2005,Kim.2014}, albeit under more restrictive conditions.
  Namely, the loads were taken in $\BL^2(\Omega)$, the stresses $\btau$ were required to be symmetric, and the condition $\diver \btau = -\bff^*$ was enforced in $K_{\bu}^*$.
  Ultimately this was used to derive an upper bound that did not assume $\bv$ in \eqref{eq:prager_synge_identity_stokes} to be divergence-free, but which required the introduction of an additional error estimator.
  In this paper, we will prove an upper and lower bound, where no additional estimator is required and the pressure error can, in fact, be controlled as well; see Section \ref{sec:stokes_revisited}.
\end{remark}

For the use in \emph{a posteriori}  error estimation, the error identity \eqref{eq:prager_synge_identity_stokes} requires admissible pairs $(\btau,\bv)\in K_{\BT}\times K_{\bu}$.
Obtaining an admissible stress is straightforward, as one can employ the Marini reconstruction \eqref{eq:stokes_marini}. Obtaining an admissible velocity vector field, however, is not entirely trivial, since the discrete solution is not exactly divergence-free. 
Thankfully, a post-processing of the velocity vector field that can deliver an admissible vector field has been developed in {\cite[Eq. (4.10)]{VZ.2019.II}}.
In Theorem \ref{thm:prager_synge_identity_stokes2}, we will present an alternative argument that enables to employ non-divergence-free vector fields the primal-dual gap estimator directly, for which, then, a simple node-averaging post-processing 
of the discrete velocity vector field suffices; the price to pay is that the primal-dual gap identity \eqref{eq:prager_synge_identity_stokes} becomes a primal-dual gap equivalence.

\section{The Navier--Lam\'e problem}\label{sec:navier-lame}
In this section, we consider the Navier--Lam\'e problem, first introduced by C.-L. Navier in 1821 (\textit{cf}.\ \cite{Navier1821}) and later cast in its modern form by G. Lamé in 1833 (\textit{cf}.\ \cite{Lame1833}), which describes the deformation of elastic bodies under external loads, with the Lamé parameters encoding resistance to compression and shear.

\subsection{The primal problem}\label{subsec:elasticity_primal}

\hspace{5mm}Let $\bff^*\in \HminusDirichlet$ be a given load and let $\bu_D \in \BH^{\frac{1}{2}}(\Gamma_D)$ be a given  Dirichlet {boundary} {datum}. For the Dirichlet boundary datum, we assume that there exists a lifting to the interior of $\Omega$, \textit{i.e.},  
there exists a vector field $\uDirichlet\in \BH^1(\Omega)$ with $\uDirichlet=\bu_D $ q.e.\ in $\Gamma_D$. 
By analogy with Section \ref{sec:stokes}, we will pose the problem for a displacement vector field $\bu\in \HoneDirichlet$, keeping in mind that the actual solution to the original problem will be given via $\bu_{\textup{orig}}\coloneqq\uDirichlet+\bu \in \uDirichlet+\HoneDirichlet$.

For given \textit{Lamé parameters} $\lambda,\mu\hspace{-0.1em} >\hspace{-0.1em} 0$, denoting the \textit{bulk modulus} and \textit{shear modulus}, {respectively}, by $\bbC\colon \Rdd\to \Rdd$
we denote the symmetric fourth-order \textit{elasticity tensor},  which, for every $\BA\in \Rdd$, is defined by
\begin{equation}\label{def:C}
	\bbC \BA \coloneqq 
  2\mu \BA + \lambda \mathrm{tr}(\BA)\BI\,.
\end{equation}
The tensor $\bbC$ is symmetric and positive definite, which ensures that the inverse $\bbC^{-1}$ also exists;
in this case, it takes the form 
\begin{equation}
	\bbC^{-1}\BB = \tfrac{1}{2\mu}\dev\BB
  + \tfrac{1}{2d\mu + d^2\lambda} \mathrm{tr}(\BB)\BI\,,
  \qquad \forall\, \BB\in\Rdd.
\end{equation}
In particular, the following norm equivalences apply:
\begin{subequations} 
	\begin{alignat}{3}
		2\mu |\BA|^2 &\leq \bbC\BA\fp \BA&& \leq (2\mu + d\lambda) |\BA|^2  &&\quad\text{for all }\BA\in \Rdd\,, \label{eq:elasticity_norm_equivalence1} \\
		\tfrac{1}{2\mu + d\lambda} |\BB|^2 &\leq \bbC^{-1}\BB\fp \BB&& \leq \tfrac{1}{2\mu} |\BB|^2  &&\quad\text{for all }\BB\in \Rdd\,. \label{eq:elasticity_norm_equivalence2}
	\end{alignat}
\end{subequations} 

As a consequence of \eqref{eq:elasticity_norm_equivalence1} and \eqref{eq:elasticity_norm_equivalence2}, it is possible to define norms induced by $\bbC$ and $\bbC^{-1}$, which are equivalent to $\|\cdot\|_{\Omega}$ and denoted by
\begin{subequations} \label{eq:norm_equiv}
	\begin{alignat}{2}
		\|\smash{\bbC^{\smash{\frac{1}{2}}}} \BA\|_{\Omega} &\coloneqq  ( (\bbC \BA, \BA)_\Omega)^{\smash{\frac{1}{2}}}&&\quad \text{ for all }\BA\in  \LtwoM\,,\label{eq:norm_equiv.1}\\
		\|\smash{\bbC^{-\smash{\frac{1}{2}}}} \BB\|_{\Omega} &\coloneqq  ( (\bbC^{-1} \BB, \BB)_\Omega)^{\smash{\frac{1}{2}}}&&\quad\text{ for all } \BB\in  \LtwoM\,.\label{eq:norm_equiv.2}
	\end{alignat}
\end{subequations}
Note \hspace{-0.15mm}that \hspace{-0.15mm}using \hspace{-0.15mm}the \hspace{-0.15mm}definition \hspace{-0.15mm}of \hspace{-0.15mm}the \hspace{-0.15mm}elasticity \hspace{-0.15mm}tensor \hspace{-0.15mm}$\bbC\colon \hspace{-0.15em}\Rdd\hspace{-0.15em}\to \hspace{-0.15em}\Rdd$ \hspace{-0.15mm}(\textit{cf}.\ \hspace{-0.15mm}\eqref{def:C}), 
\hspace{-0.15mm}one \hspace{-0.15mm}can \hspace{-0.15mm}{re-write} \hspace{-0.15mm}the Korn inequality \eqref{eq:korn} in terms of the equivalent norm \eqref{eq:norm_equiv.1}: more precisely, for every ${\bv\in  \HoneDirichlet}$, we have that
\begin{equation}\label{eq:korn2}
	\|\bbC^{\smash{\frac{1}{2}}} \nabla \bv\|_{\Omega}^2 \leq c_{\mathrm{K}} 
	\|\bbC^{\smash{\frac{1}{2}}}\symgrad{\bv}\|_{\Omega}^2 \,,
\end{equation}
where $c_{\mathrm{K}}>0$ is the same Korn constant as in the Korn inequality \eqref{eq:korn}.

Then, given the above definitions, we consider a \emph{Navier--Lam\'e (minimisation) problem} that seeks a displacement vector field $\bu \in \HoneDirichlet$ that minimises the \emph{primal energy functional} $I\colon \HoneDirichlet \to \R$, for every $\bv\in \HoneDirichlet$ defined by
\begin{equation}\label{eq:elasticity_primal}
	I(\bv) \coloneqq
	\tfrac{1}{2} (\bbC \symgrad{\bv+\uDirichlet}, \symgrad{\bv+\uDirichlet})_\Omega
	- \langle \load,\bv\rangle_\Omega\,.
\end{equation}
In this section, we refer to the minimisation of \eqref{eq:elasticity_primal} as the \textit{primal problem}. The existence of a unique minimiser $\bu\hspace*{-0.15em}\in\hspace*{-0.15em} \HoneDirichlet$, called \textit{primal solution}, follows from the direct method in the calculus of variations, where the necessary weak coercivity of the primal energy functional \eqref{eq:elasticity_primal} is crucially based on the Korn inequality \eqref{eq:korn2}. 
Since the primal energy functional \eqref{eq:elasticity_primal} is Frech\'et {differentiable}, the optimality condition associated with the primal problem is given via a variational equality (\textit{i.e.}, in $\HminusDirichlet$): 
in fact, $\bu\in \HoneDirichlet$ is minimal for \eqref{eq:elasticity_primal} if and only if for every $\bv\in \HoneDirichlet$, there holds 
\begin{equation}\label{eq:elasticity_primal_optimality}
	(\bbC \symgrad{\bu + \uDirichlet}, \symgrad{\bv})_\Omega 
	= \langle \load,\bv\rangle_\Omega\,.
\end{equation}

\subsection{The dual problem}\label{subsec:elasticity_dual}

\hspace{5mm}In this subsection, we derive and examine a (Lagrangian) dual problem (in the sense of \cite[Chap. 49]{ZeiIII}) to the Navier--Lam\'e problem \eqref{eq:elasticity_primal}. To this end, we introduce the Lagrangian  
$\mathcal{L}\colon \Ltwosym\times \HoneDirichlet \to \R$, for every $(\btau,\bv)\in \Ltwosym\times \HoneDirichlet $ defined by
\begin{align*}
	\mathcal{L}(\btau,\bv) \coloneqq 
	-\tfrac{1}{2} (\bbC^{-1}\btau,\btau)_\Omega
	+ (\btau, \nabla(\bv + \uDirichlet))_\Omega
	- \langle \load, \bv\rangle_\Omega\,.
\end{align*}
Note that the (Fenchel) conjugate functional (with respect to the second argument) of the energy density $\phi\colon \Omega\times\mathbb{R}^{d\times d}_{\textup{sym}}\to \R$, for a.e.\ $x\in \Omega$ and every $\BA\in \mathbb{R}^{d\times d}_{\textup{sym}}$ defined by
\begin{align*}
	\phi(x,\BA)\coloneqq \tfrac{1}{2}\bbC(\BA + \symgrad{\uDirichlet}(x)):(\BA + \symgrad{\uDirichlet}(x))\,,
\end{align*}
is given via $\phi^*\colon\Omega\times\mathbb{R}^{d\times d}_{\textup{sym}}\to \R$, for a.e.\ $x\in \Omega$ and every $\BA\in \mathbb{R}^{d\times d}_{\textup{sym}}$ defined by
\begin{align*}
	\phi^*(x,\BA)\coloneqq \tfrac{1}{2}\bbC^{-1}\BA:\BA-\BA:\symgrad{\uDirichlet}(x)\,.
\end{align*}
As a consequence, according to \cite[Thm.\ 2]{Rockafellar1968}, for every $\bv\in \HoneDirichlet$,~we~have~that
\begin{align*}
	\tfrac{1}{2}(\bbC\symgrad{\bv+\uDirichlet}, \symgrad{\bv+\uDirichlet})_{\Omega}=\sup_{\btau\in\Ltwosym }{\big\{-\tfrac{1}{2}(\bbC^{-1}\btau,\btau)_{\Omega}-(\btau,\symgrad{\uDirichlet})_{\Omega}\big\}}\,,
\end{align*}
which implies that, for every $\bv\in \HoneDirichlet$, we have that
\begin{align}
	I(\bv)=\sup_{\btau\in \Ltwosym}{\big\{\mathcal{L}(\btau,\bv)\big\}}\,.
\end{align} 
As the Lagrangian $\mathcal{L}\colon \Ltwosym\times \HoneDirichlet \to \R$ is concave-convex and $\mathcal{L}(\btau,\bv)\to -\infty$ ($\|\btau\|_{\Omega}\to+\infty$) for all $\bv\in \HoneDirichlet$, by the min-max-theorem (\textit{cf}.\ \cite[Thm.\ 49.B]{ZeiIII}), it admits a saddle point $(\bsigma,\bu)\in \Ltwosym\times \HoneDirichlet$, \textit{i.e.}, we have that
\begin{align}\label{eq:elasticity_minmax}
	\min_{\bv\in \HoneDirichlet}{\max_{\btau\in \Ltwosym}{\big\{\mathcal{L}(\btau,\bv)\big\}}}=\mathcal{L}(\bsigma,\bu)=\max_{\btau\in \Ltwosym}{\min_{\bv\in \HoneDirichlet}{\big\{\mathcal{L}(\btau,\bv)\big\}}}\,.
\end{align}
The optimality conditions associated with the corresponding saddle point problem amount to finding $(\bsigma,\bu)\in \Ltwosym \times \HoneDirichlet$
such that for every $(\btau,\bv)\in \Ltwosym \times \HoneDirichlet$, there holds 
\begin{subequations}\label{eq:elasticity_optimality_conditions}
	\begin{alignat}{2}
		-(\bbC^{-1}\bsigma, \btau)_\Omega
		+ (\symgrad{\bu + \uDirichlet}, \btau)_\Omega
		&	=0\,, \label{eq:elasticity_optimality_conditions.1} \\
		(\bsigma,\symgrad{\bv})_\Omega &= \langle \load,\bv \rangle_\Omega\,,\label{eq:elasticity_optimality_conditions.2}
	\end{alignat}
\end{subequations}
which is equivalent to
\begin{align}\label{eq:elasticity_optimality_conditions.3}
	\bsigma= \bbC\symgrad{\bu + \uDirichlet}\quad \text{ a.e.\ in }\Omega\,.
\end{align}
Then, according to \cite[Sec.\ 49.2]{ZeiIII}, the associated (Lagrangian) dual problem is given via the maximisation of the \emph{dual energy functional} $D\colon \Ltwosym\to \mathbb{R}\cup\{+\infty\}$, for every $\btau\in \Ltwosym$ defined by
\begin{align}\label{eq:elasticity_dual_0}
	D(\btau)\coloneqq \inf_{\bv\in \HoneDirichlet}{\big\{\mathcal{L}(\btau,\bv)\big\}}\,.
\end{align}
Note that we can swap the gradients for symmetric gradients, due to the symmetry of stresses, so that the infimum in definition \eqref{eq:elasticity_dual_0} is finite if only if for every $\bv\in \HoneDirichlet$, we have that
\begin{align*}
	(\bsigma,\nabla \bv)_{\Omega}=(\bsigma,\symgrad{\bv})_{\Omega}=\langle \bff^*,\bv\rangle_{\Omega}=(\bff,\bv)_{\Omega}+(\BF,\nabla \bv)_{\Omega}\,,
\end{align*}
where we used the representation \eqref{eq:dual_representation}, which is equivalent to $\btau-\BF\in \Hdiv$ and
\begin{subequations} \label{eq:elasticity_stress_constraint}
	\begin{alignat}{2}
		\diver (\btau - \BF) &=  -\bff &&\quad\text{ a.e.\ in }\Omega\,,\\
		\langle (\btau - \BF)\bn, \widetilde{\bv} \rangle_{\partial\Omega}
		&= 0
		&& \quad\text{  for all } \bv\in \smash{\HNeumannDual}\,.
	\end{alignat}
\end{subequations}
For this reason, it is enough to consider the \emph{restricted dual energy functional} $D\colon (\BF+\HdivN)\cap \Ltwosym\to\R\cup\{+\infty\}$, for every $\btau\in  (\BF+\HdivN)\cap \Ltwosym$ admitting the  explicit integral representation
\begin{equation} \label{eq:elasticity_dual}
	D(\btau) \coloneqq 
	-\tfrac{1}{2} (\bbC^{-1}\btau,\btau)_\Omega
	+ (\btau,\nabla\uDirichlet)_\Omega
	- \characteristic{-\bff}^\Omega(\diver(\btau - \BF))\,.
\end{equation}
Then, the effective domain of the negative of the dual energy functional \eqref{eq:elasticity_dual} is given via
\begin{align*}
	K_{\bsigma}\coloneqq \textup{dom}(-D)\coloneqq \big\{ \BF+\HdivN)\cap \Ltwosym\mid \textup{div}(\btau-\BF)=-\bff\text{ a.e.\ in }\Omega \big\}\,.
\end{align*}
By \eqref{eq:elasticity_minmax}, there exists a unique maximiser $\bsigma\in K_{\bsigma}$, called \textit{dual solution}, and a \textit{strong duality relation} applies, \textit{i.e.}, we have that
\begin{align}\label{eq:elasticity_strong_duality}
	I(\bu)=D(\bsigma)\,.
\end{align}

\subsection{The dual problem extended to non-symmetric tensor fields}

\hspace{5mm}Although the existence of the dual solution $\bsigma\hspace{-0.1em}\in \hspace{-0.1em}K_{\bsigma}$ in the previous subsection is {always} {ensured}, it is notoriously difficult to design conforming discretisations of the dual formulation \eqref{eq:elasticity_optimality_conditions}, where the discrete stresses are symmetric and simultaneously belong to $\BF+\HdivN$. In addition, at the level of \emph{a posteriori} error analyses, it is equally difficult to reconstruct equilibrated stresses that meet the symmetric requirement encoded in the 
dual energy functional \eqref{eq:elasticity_dual}. For this {reason}, in this subsection, we introduce an extension of the duality framework introduced in Subsection \ref{subsec:elasticity_primal} and Subsection \ref{subsec:elasticity_dual} to the non-symmetric case: more precisely,  in the primal {energy} {functional} \eqref{eq:elasticity_primal}, we replace the symmetric gradient with the full gradient, so that the symmetry requirement on stresses encoded
in the dual energy functional~\eqref{eq:elasticity_dual}~may~be~dropped.

To \hspace{-0.15mm}begin \hspace{-0.15mm}with, \hspace{-0.15mm}we \hspace{-0.15mm}introduce \hspace{-0.15mm}the \hspace{-0.15mm}\textit{extended \hspace{-0.15mm}primal \hspace{-0.15mm}problem} \hspace{-0.15mm}that \hspace{-0.15mm}seeks \hspace{-0.15mm}a \hspace{-0.15mm}vector \hspace{-0.15mm}field \hspace{-0.15mm}${\overline{\bu}\hspace{-0.15em}\in \hspace{-0.15em}\HoneDirichlet}$ \hspace{-0.15mm}that minimises the \textit{extended primal energy functional} $\overline{I}\colon \HoneDirichlet\to \mathbb{R}$, for every $\bv\in \HoneDirichlet$ defined by
\begin{align}\label{eq:elasticity_primal_extended}
	\overline{I}(\bv) \coloneqq \tfrac{1}{2}(\bbC\nabla(\bv+\uDirichlet),\nabla(\bv+\uDirichlet))_\Omega
	-\langle\load ,\bv\rangle_\Omega \,.
\end{align}
Note that the extended primal energy functional \eqref{eq:elasticity_primal_extended}, in fact, it not an extension of the primal energy functional \eqref{eq:elasticity_primal}, but  rather a modification, where the symmetric gradient is replaced with the full gradient.

Then, in order to identity a (Lagrangian) dual problem (in the sense of \cite[Chap. 49]{ZeiIII}) to the minimisation of \eqref{eq:elasticity_primal_extended}, we introduce the 
extended Lagrangian $\overline{\mathcal{L}}\colon \LtwoM \times \HoneDirichlet \to \R$, for every $(\btau,\bv) \in\LtwoM \times \HoneDirichlet $ defined by 
\begin{align}\label{eq:elasticity_lagrangian_extended}
	\overline{\mathcal{L}}(\btau,\bv) \coloneqq
	-\tfrac{1}{2} (\bbC^{-1}\btau,\btau)_\Omega
	+ (\btau, \nabla(\bv + \uDirichlet))_\Omega
	- \langle \load, \bv\rangle_\Omega\,.
\end{align}

By analogy with the procedure in Subsection \ref{subsec:elasticity_dual}, we find that
the (Lagrangian) dual problem (in the sense of \cite[Chap. 49]{ZeiIII}) to the minimisation of \eqref{eq:elasticity_primal_extended} is given via the maximisation of the \textit{extended dual energy functional} 
$\overline{D}\colon \BF+\HdivN\to \R\cup\{-\infty\}$, for every $\btau \in \BF+\HdivN$ defined by 
\begin{align}\label{eq:elasticity_dual_extended}
	\overline{D}(\btau) \coloneqq
	-\tfrac{1}{2} (\bbC^{-1}\btau,\btau)_\Omega
	+ (\btau,\nabla\uDirichlet)_\Omega
	- \characteristic{-\bff}^\Omega(\diver(\btau - \BF))\,.
\end{align}
Note that the extended dual energy functional \eqref{eq:elasticity_dual_extended}, in fact, is an extension of the dual energy functional \eqref{eq:elasticity_dual}, inasmuch as for every $\btau \in (\BF+\HdivN)\cap\Ltwosym$, we have that
\begin{equation}\label{eq:elasticity_dual_extension}
	\overline{D}(\btau) = D(\btau)\,.
\end{equation}
Eventually, by the min-max-theorem (\textit{cf}.\ \cite[Thm.\ 49.B]{ZeiIII}), there exists a saddle point $(\overline{\bsigma},\overline{\bu})\in (\BF+\HdivN)\times \HoneDirichlet $ of the extended Lagrangian \eqref{eq:elasticity_lagrangian_extended}, which is equivalently characterised by the following optimality relations: for every $\btau \in \LtwoM$ and $\bv\in \HoneDirichlet$, there holds
\begin{subequations}
	\begin{alignat}{2}
		-(\bbC^{-1}\overline{\bsigma}, \btau)_\Omega
		+ (\nabla(\overline{\bu} + \uDirichlet), \btau)_\Omega
		&	=0 \,,\label{eq:elasticity_mixed_nonsym_optimality1} \\
		(\overline{\bsigma},\nabla\bv)_\Omega &= \langle \load,\bv \rangle_\Omega \,.
	\end{alignat}
\end{subequations}
In particular, the vector field $\overline{\bu}\in \HoneDirichlet$ is the unique minimiser of the extended primal energy functional \eqref{eq:elasticity_primal_extended}, called \textit{extended primal solution}, the tensor field $\overline{\bsigma}\in \BF+\HdivN$ is the unique maximiser of the extended dual energy functional \eqref{eq:elasticity_dual_extended}, called \emph{extended dual solution}, and
a \emph{strong duality relation} applies, \textit{i.e.}, we have that 
\begin{equation} \label{eq:elasticity_strong_duality_nonsym}
	\overline{I} (\overline{\bu}) = \overline{D}(\overline{\bsigma})\,.
\end{equation} 

\subsection{The discrete primal problem}

\hspace{5mm}Let $\uhDirichlet\hspace{-0.1em} \in\hspace{-0.1em} \CR$, $\bff_h \hspace{-0.1em}\in\hspace{-0.1em} (\Pbroken{0})^d$, and $\BF_h \hspace{-0.1em}\in\hspace{-0.1em} \PbrokenM{0}$ approximations of $\uDirichlet\hspace{-0.1em}\in\hspace{-0.1em} \HoneDirichlet$, ${\bff\hspace{-0.1em}\in\hspace{-0.1em} \BL^2(\Omega)}$, and $\BF\in \mathbb{L}^2(\Omega)$, respectively. Then, we consider \emph{discrete Navier--Lam\'e {(minimisation)} problem} that seeks a discrete displacement vector field $\bu_h\in\CRDirichlet$ that minimises the \textit{discrete primal energy functional}  $I_h\colon \CRDirichlet \to \R$, for every $\bv_h\in \CRDirichlet$ defined by
\begin{equation}\label{eq:elasticity_primal_discrete}
	\begin{aligned} 
		I_h(\bv_h) &\coloneqq \tfrac{1}{2}\|\bv_h + \uhDirichlet\|_h^2
		- (\bff_h,\Pi_h\bv_h)_\Omega
		- (\BF_h,\nabla_h \bv_h)_\Omega\,.
	\end{aligned}
\end{equation}
In this section, we refer to the minimisation of \eqref{eq:elasticity_primal_discrete} as the \textit{discrete primal problem}. The existence of a unique minimiser $\bu_h\in \CRDirichlet$, called \textit{discrete primal solution}, follows from the direct method in the calculus of variations, where the needed weak coercivity of the discrete primal energy functional \eqref{eq:elasticity_primal_discrete} is crucially based on the discrete Korn inequality   \eqref{eq:discrete_korn} and recalling the definition of the discrete norm 
$\|\cdot\|_h$ (\textit{cf}.\ \eqref{eq:discrete_norm}). 
Since the discrete primal energy functional \eqref{eq:elasticity_primal_discrete} is Fréchet differentiable,  $\bu_h\in \CRDirichlet$ is minimal for \eqref{eq:elasticity_primal_discrete} if and only if for every $\bv_h\in \CRDirichlet$, there holds 
\begin{equation}\label{eq:elasticity_primal_optimality_discrete}
    (\bu_h + \uhDirichlet,\bv_h)_h=
	(\bff_h,\Pi_h\bv_h)_\Omega
	+ (\BF_h,\nabla_h\bv_h)_\Omega \,.
\end{equation}
Due to the presence of the stabilization term $s_h$ in the discrete primal energy functional \eqref{eq:elasticity_primal_discrete}, we do not expect to find an associated (Lagrangian) dual problem (in the sense of  \cite[Chap. 49]{ZeiIII}).

\subsection{\emph{A priori} error analysis}

\hspace*{5mm}Owing to the lack of an identified (Lagrangian) dual problem (in the sense of \cite[Chap. 49]{ZeiIII}) to the minimisation of \eqref{eq:elasticity_primal_optimality_discrete}, deriving simultaneously an \emph{a priori} error identity as well as quasi-optimal error estimate involving both the primal and the dual approximation --as in Theorem \ref{thm:discrete_prager_synge_identity_stokes}-- is out of reach. Nevertheless, we are still in the position to derive a quasi-optimal error estimate, which does not control a dual approximation error, but takes into account a dual approximation  error with respect to the \textit{set of admissible discrete stresses without symmetry requirement}, \textit{i.e.}, 
\begin{equation}
	\smash{\overline{K}}^h_{\bsigma} \coloneqq
	\big\{\btau_h \in \BF_h + \RTNeumann \mid \diver (\btau_h - \BF_h) = -\bff_h \text{ a.e.\ in }\Omega\big\}\,,
\end{equation}
on the right-hand side.

\begin{theorem} \label{thm:apriori_elasticity}
	Let $\widehat{\bu}_D^h\coloneqq \mathcal{I}_h^{cr}\widehat{\bu}_D\in V^h$, $\bff_h\coloneqq \Pi_h\bff\in (\Pbroken{0})^d$, and $\BF_h\coloneqq \Pi_h\BF\in (\Pbroken{0})^{d\times d}$. Then, the following statements apply:
	\begin{itemize}[noitemsep,topsep=2pt,leftmargin=!,labelwidth=\widthof{(ii)}]
		\item[(i)] 
        For every $\bv_h\in \CRDirichlet$ and $\btau_h \in \smash{\overline{K}}^h_{\bsigma}$, there holds the \emph{a priori} error estimate 
		\begin{equation}\label{eq:elasticity_apriori1}
			\begin{aligned}
				\|\bu_h - \bv_h\|_h^2
				&\leq
				8\,\big\{	\|\smash{\bbC^{\smash{\frac{1}{2}}}}(\symgraddisc{\bv_h + \uhDirichlet} - \smash{\bbC^{-1}}\sym(\btau_h))\|_{\Omega}^2 \\
				&\quad
				+ s_h(\bv_h + \uhDirichlet,\bv_h + \uhDirichlet)+ c_{\mathrm{K}}^{\mathrm{disc}}\|\smash{\bbC^{-\smash{\frac{1}{2}}}}\skewM(\btau_h)\|_{\Omega}^2\big\}\,;
			\end{aligned}
		\end{equation}
		\item[(ii)]  There holds the error estimate 
		\begin{equation}\label{eq:elasticity_apriori2}
			\begin{aligned}
        \hspace*{-2mm}\|\bu \hspace{-0.1em}+\hspace{-0.1em} \uDirichlet - \bu_h - \uhDirichlet\|_h
         &\lesssim
                \inf_{\bv_h \in \CRDirichlet}{\big\{\|\bu\hspace{-0.1em} +\hspace{-0.1em} \uDirichlet \hspace{-0.1em}- \hspace{-0.1em}\bv_h \hspace{-0.1em}- \hspace{-0.1em}\uhDirichlet\|_h\big\}}
         \\&\;
         + \inf_{\btau_h\in \smash{\overline{K}}^h_{\bsigma}}{\big\{\|\smash{\bbC^{-\smash{\frac{1}{2}}}}(\bsigma \hspace{-0.1em}-\hspace{-0.1em} \btau_h)\|_{\Omega}\big\}}
				\,.\hspace*{-2mm}
			\end{aligned}
		\end{equation}
	\end{itemize} 
\end{theorem}
\begin{remark}
    \begin{itemize}[noitemsep,topsep=2pt,leftmargin=!,labelwidth=\widthof{(ii)}]
        \item[(i)] To the best of the authors' knowledge, an error estimate such as 
	\eqref{eq:elasticity_apriori2} is new.
	Traditional approaches, such as \cite[Thm. 3.1]{HL.2003b}, have additional {regularity} {assumptions} (\textit{e.g.}, $\bu\in \BH^{\frac{3}{2}+\varepsilon}(\Omega)$ for some $\varepsilon>0$), and are often developed for homogeneous boundary conditions and a load $\bff^*\in \HminusDirichlet$ belonging to $\BL^2(\Omega)$, \textit{i.e.}, $\BF=\mathbf{0}$ a.e.\ in $\Omega$.
  It seems that the only other minimal regularity quasi-optimality result for a Crouzeix--Raviart discretisation involving the symmetric gradient is that of \cite{DHKZ.2025},
  which deals with the (non-linear) incompressible system, but which requires the application of a smoothing operator to trial and test functions.
  For a minimal regularity result (albeit including oscillation terms) for an alternative locking-free non-conforming method see also \cite{CT.2025}.

        \item[(ii)] Note that the first term in the right-hand side of the \emph{a priori} error estimate \eqref{eq:elasticity_apriori1} is a discrete counterpart of the primal-dual gap estimator at the continuous level (compare with \eqref{eq:elasticity_gap}), which measures the extent to which the optimality relation \eqref{eq:elasticity_optimality_conditions.3} (constitutive law) is violated.
	The other terms take care of the kernel of rigid body motions and the possible non-symmetry of the discrete stress $\btau_h\in \smash{\overline{K}}^h_{\bsigma}$, respectively.
    \end{itemize}
	
\end{remark} 

\begin{proof}[of Theorem \ref{thm:apriori_elasticity}]
	
	\emph{ad \eqref{eq:elasticity_apriori1}.}
	Let $\bv_h \in \CRDirichlet$ and $\btau_h\in \smash{\overline{K}}^h_{\bsigma}$ be fixed, but arbitrary. Then, since the discrete primal energy functional \eqref{eq:elasticity_primal_discrete} is a smooth quadratic functional,  $\mathrm{D}I_h(\bv_h)=R_{\smash{\CRDirichlet}}(\bv_h+\widehat{\bu}_D^h)-(\bff_h,\Pi_h(\cdot))_{\Omega}-(\BF_h,\nabla_h(\cdot))_{\Omega}$ in $(\CRDirichlet)^*$, where $R_{\smash{\CRDirichlet}}\colon \CRDirichlet\to (\CRDirichlet)^*$ is the Riesz isomorphism of $(\CRDirichlet,
    (\cdot,\cdot)_h)$, and  $\mathrm{D}^2I_h\equiv R_{\smash{\CRDirichlet}}$,  a Taylor expansion yields that
	\begin{align}\label{thm:apriori_elasticity.1}
		\begin{aligned} 
			0&\geq I_h(\bu_h) - I_h(\bv_h)
			\\&= \langle\mathrm{D}I_h(\bv_h),\bu_h - \bv_h\rangle_{\smash{\CRDirichlet}}+
			\tfrac{1}{2}\langle\mathrm{D}^2I_h(\bv_h)(\bu_h - \bv_h),\bu_h - \bv_h\rangle_{\smash{\CRDirichlet}}
			\\&= (\bv_h+\widehat{\bu}_D^h,\bu_h-\bv_h)_h
			\\&\quad -(\bff_h,\bu_h-\bv_h)_{\Omega}-(\BF_h,\nabla_h(\bu_h-\bv_h))_{\Omega}
			+\tfrac{1}{2}\|\bu_h - \bv_h\|_h^2\,.
		\end{aligned}
	\end{align}
	Next, using that $\textup{div}\,(\btau_h-\BF_h)=-\bff_h$ a.e.\ in $\Omega$ and the discrete integration-by-parts formula \eqref{eq:pi0} (since $\btau_h-\BF_h\in \RTNeumann$), we obtain
	\begin{align}\label{thm:apriori_elasticity.2}
		\begin{aligned} 
			-(\bff_h,\bu_h-\bv_h)_{\Omega}-(\BF_h,\nabla_h(\bu_h-\bv_h))_{\Omega}&=(\textup{div}\,(\btau_h-\BF_h),\bu_h-\bv_h)_{\Omega}
                                                                         \\&\quad-(\BF_h,\nabla_h(\bu_h-\bv_h))_{\Omega}
			\\&=(\btau_h,\nabla_h(\bv_h-\bu_h))_{\Omega}
			\\&=(\mathbb{C}^{-1}\sym(\btau_h),\mathbb{C}\symgraddisc{\bv_h-\bu_h)}_{\Omega}\\&\quad+(\skewM(\btau_h),\nabla_h(\bv_h-\bu_h))_{\Omega}\,.
		\end{aligned}\hspace*{-7.5mm}
	\end{align}
	Then, using \eqref{thm:apriori_elasticity.2} in \eqref{thm:apriori_elasticity.1}, Hölder's inequality, and the discrete Korn's inequality \eqref{eq:discrete_korn}, we {arrive} at 
	\begin{align*}
		\tfrac{1}{2}\|\bu_h - \bv_h\|_h^2&\leq  (\symgraddisc{\bv_h+\widehat{\bu}_D^h}-\mathbb{C}^{-1}\sym(\btau_h),\bbC\symgraddisc{\bv_h-\bu_h}) 
                                   \\&\quad 
                                   +s_h(\bv_h + \widehat{\bu}_D^h, \bu_h - \bv_h) 
    -(\skewM(\btau_h),\nabla_h(\bv_h-\bu_h))_{\Omega}\\&\leq
		\big\{ \|\bu_h - \bv_h\|_h^2 + \|\bbC\symgraddisc{\bu_h - \bv_h}\|_{\Omega}^2 \big\}^{\smash{\frac{1}{2}}}
		\\&\qquad\times\big\{ 
		\| \smash{\bbC^{\smash{\frac{1}{2}}}}(\symgraddisc{\bv_h + \uhDirichlet} - \smash{\bbC^{\smash{-1}}}\sym(\btau_h))\|_{\Omega}^2 \\
		&\quad\qquad+ 
		s_h(\bv_h + \uhDirichlet,\bv_h + \uhDirichlet)
		+ c_{\mathrm{K}}^{\mathrm{disc}} \|\smash{\bbC^{-\smash{\frac{1}{2}}}}\skewM(\btau_h)\|_{\Omega}^2
		\big\}^{\smash{\frac{1}{2}}}
		\\&
		\leq 2^{\frac{1}{2}}\|\bu_h - \bv_h\|_h
		\big\{ 
		\| \smash{\bbC^{\smash{\frac{1}{2}}}}(\symgraddisc{\bv_h + \uhDirichlet} - \smash{\bbC^{\smash{-1}}}\sym(\btau_h))\|_{\Omega}^2 \\
		&\qquad\qquad\qquad\qquad+ 
		s_h(\bv_h + \uhDirichlet,\bv_h + \uhDirichlet)
		\\
		&\qquad\qquad\qquad\qquad+ c_{\mathrm{K}}^{\mathrm{disc}} \|\smash{\bbC^{-\smash{\frac{1}{2}}}}\skewM(\btau_h)\|_{\Omega}^2
		\big\}^{\smash{\frac{1}{2}}}\,,
	\end{align*}
	which implies the claimed \emph{a priori} error estimate \eqref{eq:elasticity_apriori1}.
	
		%

\emph{ad \eqref{eq:elasticity_apriori2}.} Let $\bv_h \in \CRDirichlet$ and $\btau_h\in \smash{\overline{K}}^h_{\bsigma}$ be fixed, but arbitrary. Using the convex optimality condition \eqref{eq:elasticity_optimality_conditions.3},  \eqref{eq:elasticity_apriori1}, that $s_h(\bu + \uDirichlet,\bw_h)=0$ for all $\bw_h\in \CRDirichlet$, that $\skewM(\bsigma)=\bzero$ a.e.\ in $\Omega$, and that $\max\{\vert \mathbb{C}^{\smash{\frac{1}{2}}}\sym(\BA)\vert,\vert \mathbb{C}^{\smash{\frac{1}{2}}}\textup{skew}\,(\BA)\vert\}\leq \vert \mathbb{C}^{\smash{\frac{1}{2}}}\BA\vert$ for all $\BA\in \R^{d\times d}$,  we find that
\begin{align}\label{thm:apriori_elasticity.3}
\hspace*{-1.5mm}\begin{aligned} 
\|\bu + \uDirichlet - \bu_h - \uhDirichlet\|_h
&\leq
\|\bu -\uDirichlet -\bv_h -\uhDirichlet\|_h
+ 
\|\bv_h- \bu_h\|_h \\
&  \leq
\|\bu + \uDirichlet - \bv_h - \uhDirichlet\|_h
\\&\quad+ 
\sqrt{8}\, \big\{
\|\smash{\bbC^{\smash{\frac{1}{2}}}}(\symgraddisc{\bv_h +\uhDirichlet}- \bbC^{-1}\sym(\btau_h))\|_{\Omega}  
\\
&\qquad\qquad 		+ s_h(\bv_h+\uhDirichlet,\bv_h+\uhDirichlet)^{\smash{\frac{1}{2}}} 
\\
&\qquad\qquad + (c_{\mathrm{K}}^{\mathrm{disc}})^{\smash{\frac{1}{2}}}\|\smash{\bbC^{-\smash{\frac{1}{2}}}}\skewM(\sigma-\btau_h)\|_{\Omega}
\big\}
\\
&  \leq
\|\bu -\uDirichlet - \bv_h - \uhDirichlet\|_h
\\&\quad+ 
\sqrt{8}\, \big\{
\|\smash{\bbC^{\smash{\frac{1}{2}}}}(\symgraddisc{\bv_h + \uhDirichlet - \bu-\uDirichlet}\|_{\Omega}  
\\&\quad\qquad+ s_h(\bv_h+\uhDirichlet-(\bu +\uDirichlet),\bv_h+\uhDirichlet-(\bu +\uDirichlet))^{\smash{\frac{1}{2}}}
\\
  &\quad\qquad	+ \{(c_{\mathrm{K}}^{\mathrm{disc}})^{\smash{\frac{1}{2}}}+1\}\|\smash{\bbC^{-\smash{\frac{1}{2}}}}(\bsigma-\btau_h)\|_{\Omega}
\big\}\,.
\end{aligned}\hspace*{-6mm}
\end{align}
Then, taking infima with respect to $\bv_h\in\CRDirichlet$ and $\btau_h\in \smash{\overline{K}}^h_{\bsigma}$, respectively, in \eqref{thm:apriori_elasticity.3}, we conclude that the claimed quasi-optimal error estimate \eqref{eq:elasticity_apriori2} applies.  
\end{proof}

Combining Theorem \ref{thm:apriori_stokes} with the approximation properties \eqref{eq:cr_best_approximation}--\eqref{eq:rt_best_approximation} it is, then, possible to extract the optimal error decay rate $s\in (0,1]$, under
appropriate regularity {assumptions}, \textit{i.e.},
\begin{equation}
\|\bu + \uDirichlet - \bu_h - \uhDirichlet\|_h^2
= \mathcal{O}(h_\mathcal{T}^{2s})\,,
\end{equation}
where as usual, thanks to the approximation property \eqref{eq:cr_div_best_approximation}, the approximation is robust in the almost incompressible regime $\frac{\lambda}{\mu}\gg 1$.

\subsection{\emph{A posteriori} error analysis} \label{sec:aposteriori_elasticity}
\hspace{5mm}The goal of this section is to prove a primal-dual gap identity analogous to Theorem \ref{thm:prager_synge_identity_stokes}\eqref{eq:prager_synge_identity_stokes}. 
As previously mentioned, one major difficulty is that for a given discrete primal solution $\bu_h\in \CRDirichlet$, it is challenging to reconstruct a discrete stress $\bsigma_h^*\hspace{-0.175em}\in \hspace{-0.175em}\smash{\overline{K}}^h_{\bsigma}$ that is {exactly} (or even weakly) {symmetric};
in fact, the later reconstructed discrete stress $\bsigma_h^*\in \smash{\overline{K}}^h_{\bsigma}$ obtained via a Marini-type formula (\textit{cf}.\ \eqref{prop:marini_elasticity.1}) will neither be exactly nor weakly symmetric and, consequently,  not admissible for the dual energy functional \eqref{eq:elasticity_dual},
preventing one to repeat the argument from Theorem \ref{thm:prager_synge_identity_stokes}\eqref{eq:prager_synge_identity_stokes} \emph{{verbatim}}.
Instead, the analysis in this section will make use of the extended dual energy functional \eqref{eq:elasticity_dual_extended} and the conformity property \eqref{eq:elasticity_dual_extension}.
Of course, if a symmetric post-processing were be available, then the corresponding \emph{a posteriori} error identity could be used;
for a Prager--Synge identity for the Navier--Lam\'e problem \eqref{eq:elasticity_primal} derived in the context of exactly symmetric stresses, we refer to \cite{LS.2023}.

By analogy with Section \ref{sec:aposteriori_stokes}, we start by introducing the \textit{extended primal-dual gap estimator} $\gapext\colon \HoneDirichlet \times \overline{K}_{\bsigma}\to \R$, for every $(\bv,\btau) \in \HoneDirichlet \times \overline{K}_{\bsigma}$ defined by
\begin{equation}\label{eq:elasticity_gap}
\gapext (\bv,\btau) 
= 
\tfrac{1}{2}\|\smash{\bbC^{\smash{\frac{1}{2}}}}(\symgrad{\bv+\uDirichlet} - \smash{\bbC^{-1}}\btau)\|_{\Omega}^2\,, 
\end{equation} 
where the \emph{set of admissible stresses without symmetry requirement}~is~\mbox{defined}~by
\begin{equation}
\overline{K}_{\bsigma} \coloneqq
\{\btau \in \BF + \HdivN \mid \diver(\btau - \BF) = -\bff \text{ a.e. in }\Omega\}\,.
\end{equation}
Similarly to  Section \ref{sec:aposteriori_stokes}, we assume that $\bff=\bff_h\in (\Pbroken{0})^d$ as well as $\BF=\BF_h\in (\Pbroken{0})^{d\times d}$; in the general case, the \emph{a posteriori} error estimates, in addition, will include data oscillation terms (\textit{cf}.\ Section \ref{sec:oscillation}).
Note that, unlike the primal-dual gap estimator \eqref{eq:primal-dual.1} for the Stokes problem, here the \emph{a posteriori} error estimator \eqref{eq:elasticity_gap} is \emph{defined} as the deviation from the convex optimality relation \eqref{eq:elasticity_optimality_conditions.3}. Its precise relation to energy differences is established in the following lemma.

\begin{lemma}\label{lem:primal_dual_gap_estimator_elasticity}
For every $\bv\in \HoneDirichlet$ and $\btau \in \overline{K}_{\bsigma} $, we have that
\begin{equation*}
\gapext(\bv,\btau) =
I(\bv) - \overline{D}(\btau)
+ (\skewM(\btau),\nabla(\bv+\uDirichlet))_\Omega\,. 
\end{equation*}
\end{lemma}

\begin{remark}
We remark that in case that $\skewM(\btau)=\bzero$ a.e.\ in $\Omega$, then the extended primal-dual gap estimator \eqref{eq:elasticity_gap} is, in fact, the duality gap, since, in this case, we have that $\overline{D}(\btau)=D(\btau)$.
\end{remark}

\begin{proof}[of Lemma \ref{lem:primal_dual_gap_estimator_elasticity}]
For every $\bv\in \HoneDirichlet$ and $\btau \in \overline{K}_{\bsigma}$, using \eqref{eq:stokes_stress_constraint}, the continuous integra\-tion-by-parts formula (since $\btau-\BF\in \HdivN$), \eqref{eq:Neumann_condition_weak}, and a binomial formula, we find that
\begin{align*}
I(\bv) - \overline{D}(\btau)
&=
\tfrac{1}{2}(\bbC\symgrad{\bv+\uDirichlet},\symgrad{\bv+\uDirichlet})_\Omega
- (\bff,\bv)_\Omega - (\BF, \nabla \bv)_\Omega \\
&\quad
+ \tfrac{1}{2}(\bbC^{-1}\btau,\btau)_\Omega
- (\btau,\nabla \uDirichlet)_\Omega
\\ 
&=
\tfrac{1}{2}(\bbC\symgrad{\bv+\uDirichlet},\symgrad{\bv+\uDirichlet})_\Omega
+ (\diver(\btau - \BF),\bv)_\Omega - (\BF, \nabla \bv)_\Omega \\
&\quad
+ \tfrac{1}{2}(\bbC^{-1}\btau,\btau)_\Omega
- (\btau,\nabla \uDirichlet)_\Omega
\\ 
&=
\tfrac{1}{2}(\bbC\symgrad{\bv+\uDirichlet},\symgrad{\bv+\uDirichlet})_\Omega
- (\btau, \nabla(\bv+\uDirichlet)) 
+ \tfrac{1}{2}(\bbC^{-1}\btau,\btau)_\Omega
\\ 
&=
\tfrac{1}{2}\|\smash{\bbC^{\smash{\frac{1}{2}}}}(\symgrad{\bv+\uDirichlet} - \bbC^{-1}\btau)\|_{\Omega}^2
- (\skewM(\btau),\nabla(\bv+\uDirichlet))_\Omega\,,
\end{align*}
which is the claimed representation of the extended primal-dual gap estimator \eqref{eq:elasticity_gap}.
\end{proof}

Next, by analogy with Section \ref{sec:aposteriori_stokes}, we continue with introducing the \emph{(exten\-ded) strong convexity measures} for the primal energy functional \eqref{eq:elasticity_primal} at the primal solution $\bu\hspace{-0.1em}\in\hspace{-0.1em} K_{\bu}$, \textit{i.e.}, $\rhoprimal\colon \HoneDirichlet\hspace{-0.1em} \to\hspace{-0.1em} [0,+\infty) $, for every~${\bv\hspace{-0.1em}\in\hspace{-0.1em} K_{\bu}}$~\mbox{defined}~by
\begin{equation}\label{def:elasticity_primal_error}
\rhoprimal(\bv) \coloneqq
\tfrac{1}{2}\|\smash{\bbC^{\smash{\frac{1}{2}}}}\symgrad{\bv-\bu}\|_{\Omega}^2\,,
\end{equation}
and of (the negative of) the extended dual energy functional \eqref{eq:elasticity_dual_extended} at the dual solution $\bsigma \in \overline{K}_{\bsigma} $, \textit{i.e.}, $\rhodualext\colon \overline{K}_{\bsigma} \to \R$, for every $\btau\in \overline{K}_{\bsigma}$ defined by 
\begin{equation}\label{def:elasticity_dual_error_extended}
\rhodualext(\btau) \coloneqq
\tfrac{1}{2}\|\smash{\bbC^{-\smash{\frac{1}{2}}}}(\btau-\bsigma)\|_{\Omega}^2\,.
\end{equation}
Note that, unlike the optimal strong convexity measures \eqref{def:optimal_primal_error} and \eqref{def:optimal_dual_error} for the Stokes {problem}, here the (extended) strong convexity measures \eqref{def:elasticity_primal_error}~and~\eqref{def:elasticity_dual_error_extended}, respectively, are \emph{defined} as the
desired error measures for the primal energy functional \eqref{eq:elasticity_primal} and the extended dual energy functional \eqref{eq:elasticity_dual_extended}. 
Their precise relations to energy differences is established in the following lemma.


\begin{lemma}\label{lem:strong_convexity_measures_elasticity}
The following representations apply:
\begin{itemize}[noitemsep,topsep=2pt,leftmargin=!,labelwidth=\widthof{(ii)}]
\item[(i)]\hypertarget{lem:strong_convexity_measures_elasticity.i}{} For every $\bv\in \HoneDirichlet$, we have that 
\begin{equation*}
\rhoprimal(\bv) 
= I(\bv) - I(\bu)\,; 
\end{equation*}

\item[(ii)]\hypertarget{lem:strong_convexity_measures_elasticity.ii}{} For every $\btau\in \overline{K}_{\bsigma}$, we have that
\begin{equation*}
\rhodualext(\btau) 
=
-\overline{D}(\btau) + \overline{D}(\bsigma)
+ (\skewM(\btau),\nabla(\bu+\uDirichlet))_\Omega\,.
\end{equation*}
\end{itemize}
\end{lemma}
\begin{proof}
\emph{ad (\hyperlink{lem:strong_convexity_measures_elasticity.i}{i}).} 
Since $I\colon \HoneDirichlet\to \mathbb{R}$ is a quadratic functional, $\mathrm{D}I(\bu)=\mathbf{0}^*$ in $\HminusDirichlet$, and $\mathrm{D}^2I\equiv R_{\smash{\HoneDirichlet}}$, where $R_{\smash{\HoneDirichlet}}\colon\hspace{-0.175em}\HoneDirichlet\hspace{-0.175em}\to\hspace{-0.175em} \HminusDirichlet$ is the Riesz isomorphism of $(\HoneDirichlet,(\bbC^{\frac{1}{2}}\symgrad{\cdot},\bbC^{\frac{1}{2}}\symgrad{\cdot})_{\Omega})$, for every $\bv\in \HoneDirichlet$, a Taylor expansion yields that
\begin{align*}
I(\bv) - I(\bu)
&= \langle \mathrm{D}I(\bu),\bv-\bu\rangle_{\HoneDirichlet}+\tfrac{1}{2}\langle \mathrm{D}^2I(\bu)(\bv-\bu),\bv-\bu\rangle_{\smash{\HoneDirichlet}}
\\&=
\tfrac{1}{2}\|\smash{\bbC^{\smash{\frac{1}{2}}}}\symgrad{\bv-\bu}\|_{\Omega}^2\,,
\end{align*}
which is the claimed representation of the strong convexity measure \eqref{def:elasticity_primal_error}.

\emph{ad (\hyperlink{lem:strong_convexity_measures_elasticity.ii}{ii})}.
For $\btau\in \overline{K}_{\bsigma}$, 
using that $(\btau-\bsigma,\nabla \bu)_{\Omega}=0$ (which follows from the continuous integra\-tion-by-parts formula together with $\btau,\bsigma\in \overline{K}_{\bsigma}$), that
$\skewM(\bsigma)=\bzero$ a.e.\ in $\Omega$ (since $\bsigma\in K_{\bsigma}$) and the convex optimality condition \eqref{eq:elasticity_optimality_conditions.3},  we find that
\begin{align*}
-\overline{D}(\btau) + \overline{D}(\bsigma)
&=
\tfrac{1}{2}(\bbC^{-1}\btau,\btau)_\Omega
-\tfrac{1}{2}(\bbC^{-1}\bsigma,\bsigma)_\Omega
- (\btau - \bsigma,\nabla\uDirichlet)_\Omega \\
&=
\tfrac{1}{2}(\bbC^{-1}\btau,\btau)_\Omega
-\tfrac{1}{2}(\bbC^{-1}\bsigma,\bsigma)_\Omega
- (\btau - \bsigma,\nabla(\bu +\uDirichlet))_\Omega
\\
&=
\tfrac{1}{2}(\bbC^{-1}\btau,\btau)_\Omega
-\tfrac{1}{2}(\bbC^{-1}\bsigma,\bsigma)_\Omega
- (\btau - \bsigma,\symgrad{\bu +\uDirichlet})_\Omega
\\&\quad- (\skewM(\btau),\nabla(\bu +\uDirichlet))_\Omega
\\
&=
\tfrac{1}{2}\|\smash{\bbC^{-\smash{\frac{1}{2}}}}(\btau-\bsigma)\|_{\Omega}^2
- (\skewM(\btau),\nabla(\bu+\uDirichlet))_\Omega\,,
\end{align*}
which is the claimed representation of the  strong convexity measure \eqref{def:elasticity_dual_error_extended}.
\end{proof}

Now let us introduce the \emph{extended primal-dual total error}, \textit{i.e.},  $\rhototext \colon \HoneDirichlet \times \overline{K}_{\bsigma} \to \R$, for every $\bv\in \BH^1_D(\Omega)$ and $\btau\in \overline{K}_{\sigma}$ defined by
\begin{equation*}
\rhototext(\bv,\btau) \coloneqq
\rhoprimal(\bv) + \rhodualext(\btau)\,.
\end{equation*}
We are now in the position  to derive  an \emph{a posteriori}  error equivalence.
Note that, similarly to the \emph{a priori} error estimate in Theorem \ref{thm:apriori_elasticity}, we obtain an equivalence as opposed to an identity, 
since the admissible stresses $\btau\in \overline{K}_{\bsigma}$ is not assumed to be symmetric.

\begin{theorem}\label{thm:prager_synge_identity_elasticity}
For every $\bv\in \HoneDirichlet$ and $\btau\in \overline{K}_{\bsigma}$, we have the following identity
and equivalence
\begin{equation}\label{eq:elasticity_aposteriori1}
\begin{aligned} 
\rhototext(\bv,\btau)= \gapext(\bv,\btau)
+ (\skewM(\btau),\nabla(\bu-\bv))_\Omega\sim \gapext(\bv,\btau)\,,
\end{aligned}
\end{equation}
where the equivalence explicitly written reads  
\begin{equation}\label{eq:elasticity_aposteriori2}
\|\smash{\smash{\bbC^{\smash{\frac{1}{2}}}}}\symgrad{\bv-\bu}\|_{\Omega}^2
+
\|\smash{\smash{\bbC^{-\smash{\frac{1}{2}}}}}(\btau-\bsigma)\|_{\Omega}^2
\sim
\|\smash{\smash{\bbC^{\smash{\frac{1}{2}}}}}(\symgrad{\bv+\uDirichlet} - \smash{\bbC^{-1}}\btau)\|_{\Omega}^2\,,
\end{equation}
the implicit constant in the estimate $\gtrsim$ is given via $2$ and the implicit  in the estimate $\lesssim $ depends only on $\max_{T\in \mathcal{T}_h}{\{c_{\mathrm{K},T}^2\}}$, \textit{i.e.}, the maximum of element-wise Korn constants $c_{\mathrm{K},T}^2>0$ (\textit{cf}.~\cite[Ineq. (3.2) \& Thm. 3.1]{Kim.2011}). 
\end{theorem}

\begin{proof}

\emph{ad identity in \eqref{eq:elasticity_aposteriori1}.}
Recalling the conformity property \eqref{eq:elasticity_dual_extension} and the strong duality relation \eqref{eq:elasticity_strong_duality},
the identity in \eqref{eq:elasticity_aposteriori1} is a direct consequence of Lemma \ref{lem:primal_dual_gap_estimator_elasticity} together with Lemma \ref{lem:strong_convexity_measures_elasticity}.

\emph{ad equivalence in \eqref{eq:elasticity_aposteriori1}.} 
Next, let us establish the claimed equivalence:

\emph{ad $\lesssim$.} For every $\bv\in \HoneDirichlet$ and $\btau\in \overline{K}_{\bsigma}$, abbreviating $\BA_h\coloneqq \Pi_h\skewM\nabla(\bu-\bv) \in (\Pbroken{0})^{d\times d}$ and 
using that $\skewM(\btau)\perp_{\mathbb{L}^2} \skewM\nabla(\bu-\bv)-\BA_h$, that 
$\skewM(\btau):\nabla(\bu-\bv)=\btau:\skewM\nabla(\bu-\bv)$ a.e.\ in $\Omega$, that 
$\symgrad{\bv+\uDirichlet}\perp_{\mathbb{L}^2}\mathbb{C}\skewM\nabla(\bu-\bv)$,  and Hölder's inequality, 
we find that 
\begin{align*}
	( \skewM(\btau),\nabla(\bu-\bv))_\Omega
	&= (\symgrad{\bv+\uDirichlet} - \bbC^{-1}\btau, \bbC \,\skewM\nabla(\bu-\bv))_\Omega \\
	&= (\symgrad{\bv+\uDirichlet} - \bbC^{-1}\btau, \bbC \,(\skewM\nabla(\bu-\bv)-\BA_h))_\Omega \\
	&\leq \|\bbC^{\smash{\frac{1}{2}}}(\symgrad{\bv+\uDirichlet} - \bbC^{-1}\btau)\|_{\Omega}	\\&\quad\times\| \bbC^{\smash{\frac{1}{2}}}(\skewM\nabla(\bu-\bv)-\BA_h)\|_{\Omega}\,,
\end{align*}
where, by the $\mathbb{L}^2$-best-approximation property of $\Pi_h$, 
denoting by 
$$
\mathrm{RM}(T)\coloneqq \{\BA(\cdot)+\bb\colon T\to \mathbb{R}^{d\times d}\mid \BA\in \mathbb{R}^{d\times d}\,,\; \bb\in \mathbb{R}^d\,,\;\BA^{\top}=-\BA\}\,,
$$
the \emph{space of rigid body motions on $T$}, an element-wise application of the Korn inequality  (\textit{cf}.\ \cite[Ineq. (3.2)]{Kim.2011}) yields that
\begin{align*}
	&\| \bbC^{\smash{\frac{1}{2}}}(\skewM\nabla(\bu-\bv)-\BA_h)\|_{\Omega}\\&=\bigg\{\sum_{T\in \mathcal{T}_h}{\| \bbC^{\smash{\frac{1}{2}}}(\skewM\nabla(\bu-\bv)-\BA_h)\|_T^2}\bigg\}^{\smash{\frac{1}{2}}}
	\\&\leq \bigg\{\sum_{T\in \mathcal{T}_h}{	\inf_{\brho_T\in \mathrm{RM}(T)}{\big\{\|\bbC^{\smash{\frac{1}{2}}}
		(\skewM\nabla(\bu-\bv)-\brho_T)\|_T^2\big\}}}\bigg\}^{\smash{\frac{1}{2}}} 
			\\&\leq \bigg\{\sum_{T\in \mathcal{T}_h}{	c_{\mathrm{K},T}^2\|\bbC^{\smash{\frac{1}{2}}} \symgrad{\bu-\bv})\|_T^2}\bigg\}^{\smash{\frac{1}{2}}} \,,
\end{align*}
where $c_{\mathrm{K},T}^2>0$, $T\in \mathcal{T}_h$, denote the element-wise Korn constants (\textit{cf}.\ \cite[Thm.\ 3.1]{Kim.2011}),
so that, an application of Young's inequality yields that the claimed upper bound applies.

\emph{ad $\gtrsim$.} Using the convex optimality condition \eqref{eq:elasticity_optimality_conditions.3}, we obtain
\begin{align*}
\|\smash{\bbC^{\smash{\frac{1}{2}}}}(\symgrad{\bv+\uDirichlet} - \smash{\bbC^{-1}}\btau)\|_{\Omega}^2
&=
\|\smash{\bbC^{\smash{\frac{1}{2}}}}(\symgrad{\bv+\uDirichlet} - \symgrad{\bu + \uDirichlet} + \smash{\bbC^{-1}}\bsigma - \smash{\bbC^{-1}}\btau)\|_{\Omega}^2
\\ 
&\leq
2\big\{ \|\smash{\bbC^{\smash{\frac{1}{2}}}}(\symgrad{\bv - \bu}\|_{\Omega}^2   + \|\smash{\bbC^{-1}}(\bsigma - \btau)\|_{\Omega}^2 \big\}\,,
\end{align*}
which yields the claimed lower bound. 
\end{proof}

\begin{remark}[on the element-wise Korn constants in Theorem \ref{thm:prager_synge_identity_elasticity}]\label{rem:local_Korn_constants}
    According to \cite[Thm. 3.1]{Kim.2011}, for every $T\hspace{-0.15em}\in\hspace{-0.15em} \mathcal{T}_h$,
    we have that $c_{\mathrm{K},T}^2\hspace{-0.15em}\leq\hspace{-0.15em} \tfrac{2}{\sin^2(\frac{\theta_{\smash{\min,T}}}{4})}$,
    where $\theta_{\min,T}\hspace{-0.15em}\in\hspace{-0.15em} (0,2\pi)$ is the minimum {angle} in $T$.
\end{remark}

\begin{remark}[on related results to Theorem \ref{thm:prager_synge_identity_elasticity}]

    Several works have developed Prager--Synge-type \emph{a posteriori} error estimators for the Navier--Lam\'e problem: 
    
    The results in \cite{DM.2013} are structurally closest to the present approach in that a \emph{single} \emph{a posteriori} estimator is derived (without requiring symmetry of the reconstructed discrete stresses) and the analysis involves element-wise Korn constants (\textit{cf}.\ Remark \ref{rem:local_Korn_constants}).  The results in \cite{DM.2013}, however, are established for conforming discretisations with polynomial degree $k \ge 2$. 
    
    A related framework is presented in \cite{BKMS.2021}, where a displacement-pressure formulation is analysed and a Prager--Synge-type \emph{a posteriori} estimator is constructed based on a weakly symmetric reconstruction $\bsigma_h^* \in \BH(\mathrm{div};\Omega)$. 
    The framework in  \cite{BKMS.2021} likewise assumes polynomial degree $k \ge 2$, proves local efficiency, and demonstrates robustness with respect to the incompressible limit. 

The work \cite{Kim.2012} employs a Crouzeix--Raviart discretisation and reconstructs stresses in the next-to-lowest Raviart--Thomas space. 
The resulting reconstruction is weakly symmetric, but the analysis requires separate estimators for the symmetry defect and for the non-conformity of $\bu_h\in V_D^h$, and does not establish efficiency. The reconstruction strategy from \cite{Kim.2012} could, in principle, also be combined with the present estimator without affecting the validity of the analysis, which remains applicable under minimal regularity assumptions on the displacement vector field ${\bu \in \BH^1_D(\Omega)}$.

\end{remark}

For the applicability of the derived \emph{a posteriori}  error equivalence, it is necessary to have a computationally inexpensive way to reconstruct, from the discrete displacement field $\bu_h\in K_{\bu}^h$, solving \eqref{eq:elasticity_primal_optimality_discrete}, an 
admissible discrete tensor field $\bsigma_h^* \in K_{\bsigma}$. Although we have not derived a discrete dual problem to minimisation of the discrete primal energy functional \eqref{eq:elasticity_primal_discrete}, {similar} to Proposition \ref{prop:stokes_marini}, we are able to establish a Marini-type formula for such an admissible discrete tensor field.

\begin{proposition}[Marini-type formula]\label{prop:marini_elasticity}
Given the discrete primal solution $\bu_h\in K_{\bu}^h$ (\textit{i.e.}, the discrete displacement field) and a discrete lift  $\br_h\in \CRDirichlet$ of the jump stabilization term $s_h(\bu_h + \uhDirichlet, \cdot)$, for every $\bv_h\in \CRDirichlet$ defined by 
\begin{align}\label{eq:lifting}
(\nabla_h \br_h, \nabla_h \bv_h)_\Omega
=
s_h(\bu_h + \uhDirichlet, \bv_h)\,,
\end{align}
the discrete tensor field 
\begin{align}\label{prop:marini_elasticity.1}
\smash{\bsigma_h^*\coloneqq \bbC\symgraddisc{\bu_h + \uhDirichlet} +
\nabla_h \br_h 
- \tfrac{2}{d+1}\sym (\bff_h \otimes (\mathrm{id}_{\R^d} - \Pi_h \mathrm{id}_{\R^d}))\in (\Pbroken{1})^{d\times d}\,,}
\end{align}
satisfies $\bsigma_h^*\in K_{\bsigma}^h$ with 
\begin{subequations}\label{prop:marini_elasticity.2}
	\begin{alignat}{2}\label{prop:marini_elasticity.2.1}
		\Pi_h\bsigma_h^*&=\bbC\symgraddisc{\bu_h + \uhDirichlet} +
		\nabla_h \br_h&&\quad \text{ a.e.\ in }\Omega\,,\\
		\textup{div}(\bsigma_h^*-\BF_h)&=-\bff_h&&\quad\text{ a.e.\ in }\Omega\,.\label{prop:marini_elasticity.2.2}
	\end{alignat}
\end{subequations}
\end{proposition}

\begin{proof}
	To begin with, we note that $\textup{div}_h \bsigma_h^* = -\tfrac{2}{d+1}(\tfrac{d}{2}\bff_h + \tfrac{1}{2}\bff_h )
		= - \bff_h$ a.e.\ in $\Omega$.
	On the other hand, by the surjectivity of the row-wise divergence operator $\textup{div}\colon \RTNeumann\to (\Pbroken{0})^d$ (since $\vert \Gamma_D\vert>0$), there exists a tensor field $\widehat{\bsigma}_h\in \RTNeumann$ with $\diver \widehat{\bsigma}_h=-\bff_h$ a.e.\ in $\Omega$.  As a consequence, we have that
	$\textup{div}_h (\bsigma_h^*-\widehat{\bsigma}_h)=\mathbf{0}$ a.e.\ in $\Omega$ and, thus, $\bsigma_h^*-\widehat{\bsigma}_h\in  (\Pbroken{0})^{d\times d}$. 
    By the discrete convex optimality condition
	\eqref{eq:elasticity_primal_optimality_discrete} and the discrete integration-by-parts formula \eqref{eq:pi0}, for {any} ${\bv_h\hspace{-0.15em}\in\hspace{-0.15em} \CRDirichlet}$, we have that
	\begin{align*}
		(\bsigma_h^* -\BF_h- \widehat{\bsigma}_h, \nabla_h{\bv_h})_\Omega=(\bff_h,\Pi_h\bv_h)_{\Omega}-(\Pi_h  \widehat{\bsigma}_h,\nabla_h{\bv_h})_{\Omega}=0\,.
	\end{align*}
	Then, using the discrete Helmholtz--Weyl decomposition \eqref{eq:decomposition}, we find that 
	\begin{align*}
		\bsigma_h^*-\BF_h-\widehat{\bsigma}_h\in \smash{(\nabla_h(\CRDirichlet))^{\perp_{\mathbb{L}^2}}}=\textup{ker}(\textup{div}|_{\smash{\RTNeumann}})\,,
	\end{align*}
	and, due to $\widehat{\bsigma}_h\in \RTNeumann$, we infer that $	\bsigma_h^*-\BF_h\in \RTNeumann$ and that $\textup{div}(	\bsigma_h^*-\BF_h)=\diver \widehat{\bsigma}_h=-\bff_h$ a.e. in $\Omega$, so that, in summary, we conclude that $	\bsigma_h^*\in K_{\bsigma}^h$ with \eqref{prop:marini_elasticity.2}.
\end{proof}

\begin{remark}[on Proposition \ref{prop:marini_elasticity}]\label{rem:marini_elasticity}
	\begin{itemize}[noitemsep,topsep=2pt,leftmargin=!,labelwidth=\widthof{(ii)}]
		\item[(i)] In practice, the linear problem  \eqref{eq:lifting} may be solved inexactly without affecting the performance significantly;
		\item[(ii)] The lack of symmetry in the Marini-type formula \eqref{prop:marini_elasticity.1} can be quantified as
		\begin{equation*}
			\smash{\|\skewM{\bsigma_h^*}\|_{\Omega} ^2
			\sim 
			\|\skewM(\nabla_h \br_h)\|_{\Omega} ^2
			\sim 
			s_h(\bu_h + \uhDirichlet, \bu_h + \uhDirichlet)\,,}
		\end{equation*}
		and, thus, is related to the $\BH^1(\Omega)$-conformity of $\bu_{\textup{orig}}^h\coloneqq\bu_h + \uhDirichlet\in \CR$.
    This suggests that, in practice, one could omit the computation of the lifting $\br_h$ and still get a reliable estimator by including instead the residual jump estimator $s_h(\bu_h + \uhDirichlet, \bu_h + \uhDirichlet)$.
		\item[(iii)] An alternative Marini-type formula to \eqref{prop:marini_elasticity.1}, which does not have a symmetric affine part, is 
		\begin{align*}
			\smash{\bsigma_h^*\coloneqq \bbC\symgraddisc{\bu_h + \uhDirichlet} +
			\nabla_h \br_h 
			- \tfrac{1}{d}\bff_h \otimes (\mathrm{id}_{\R^d} - \Pi_h \mathrm{id}_{\R^d})\in (\Pbroken{1})^{d\times d}\,.}
		\end{align*} 
	\end{itemize} 
\end{remark}

\begin{remark}\label{rem:marini_elasticity2}
The Marini-type formula \eqref{prop:marini_elasticity.1} seems to be new;
in the context of Crouzeix--Raviart discretisations of linear elasticity,
the only somewhat similar reconstruction procedure was developed in \cite{Kim.2012},
where a stress is reconstructed in the next-to-lowest order Raviart--Thomas space for $d=2$, homogeneous Neumann boundary conditions, and load in $\BL^2(\Omega)$;
we note that their analysis is more involved, requiring additional estimators and proving only reliability.
Other works developed for other discretisations, such as \cite{AR.2010,AR.2011,NWW.2008,Kim.2011,RdPE.2017,OSW.2001,PCL.2012,LS.2024},
present reconstruction procedures based on local solves that employ finite elements, which allow the enforcement of symmetry (either weakly or strongly); in order to enforce symmetry, most of these works employ at least quadratic elements.
For an approach based also on a Prager--Synge identity that does not require symmetry (which, however, assumes that the polynomial degree is at least 2), we refer to \cite{DM.2013}.
In summary, the reconstruction formula \eqref{prop:marini_elasticity.1} provides a lowest-order $H(\textup{div};\Omega)$-conforming approximation of the stress with computational cost comparable to a primal solve.
\end{remark}

\subsection{The Stokes problem revisited}\label{sec:stokes_revisited}

\hspace{5mm}Incidentally, it is possible to carry out an argument analogous Section \ref{sec:aposteriori_elasticity} for the incompressible Stokes  problem \eqref{eq:stokes_primal}, and obtain an alternative result to the \emph{a posteriori} identity \eqref{eq:prager_synge_identity_stokes}, which does not require primal approximations $\bv\in \BH^1_D(\Omega)$ to be exactly divergence-free;
this simplifies the implementation in that we can employ the discrete primal solution $\bu_h\in V_D^h$ after an inexpensive post-processing (\textit{e.g.}, via nodal averaging)  
directly.
The clear downside, of course, is that we obtain an \emph{a posteriori} error equivalence, as opposed to an \emph{a posteriori} error identity.

To begin with, we introduce a \emph{(to non-divergence-free vector field) extended primal problem} that seeks a vector field $\overline{\bu}\in\BH^1_D(\Omega)$ that minimises the \emph{extended primal energy functional} $\overline{I}\colon \BH^1_D(\Omega)\to \mathbb{R}$, for every $\bv\in \BH^1_D(\Omega)$ defined by
\begin{equation}\label{eq:stokes_primal_extended} 
\overline{I}(\bv) \coloneqq \tfrac{\nu}{2}\|\nabla(\bv + \uDirichlet)\|_{\Omega}^2 - \langle \load ,\bv\rangle_\Omega\,.
\end{equation}

Next, by analogy with Section \ref{sec:aposteriori_elasticity}, we introduce the \emph{(extended) strong convexity measure} $\rho^2_{\overline{I}}\colon \HoneDirichlet \to [0,+\infty) $, for every $\bv\in \HoneDirichlet$ 
\emph{defined} by
\begin{align}\label{eq:stokes_optimal2}
\rho^2_{\overline{I}}(\bv) &\coloneqq \tfrac{\nu}{2}\|\nabla \bv - \nabla \bu\|_{\Omega}^2\,,
\end{align}
where $\bu\in K_{\bu}$ denotes the unique primal solution of the incompressible Stokes  problem \eqref{eq:stokes_primal}.
The dual energy functional \eqref{eq:stokes_dual} with unique dual solution $\BT\in K_{\BT}$ and the corresponding optimal strong convexity measure \eqref{def:optimal_dual_error} remain as before (\textit{cf}.\  Lemma \ref{lem:strong_convexity_measures_stokes}), 
so that the \emph{extended primal-dual total error} $\rhototext\colon \HoneDirichlet\times K_{\BT} \to [0,+\infty) $, for every $\bv\in \HoneDirichlet$ and $\btau\in K_{\BT} $, is defined by 
\begin{align}
    \rhototext(\bv,\btau)\coloneqq \rho^2_{\overline{I}}(\bv)+\rho^2_{D}(\btau)\,.
\end{align}
Eventually, by analogy with Section \ref{sec:aposteriori_elasticity}, we introduce the \emph{extended primal-dual gap estimator} $\gapext\colon\HoneDirichlet\times K_{\BT} \to [0,+\infty) $, for every $\bv\in \HoneDirichlet$ and $\btau\in K_{\BT} $ \emph{defined} by  
\begin{equation}\label{eq:stokes_gap2}
\gapext(\bv,\btau)\coloneqq \tfrac{\nu}{2}\|\nabla (\bv+\uDirichlet)-\smash{\tfrac{1}{\nu}}\dev \btau\|_{\Omega}^2\,. 
\end{equation}
Their precise relations to energy differences is established in the following lemma.

\begin{lemma}
The following statements apply:
\begin{itemize}[noitemsep,topsep=2pt,leftmargin=!,labelwidth=\widthof{(ii)}]
\item[(i)]\hypertarget{eq:stokes_gap2.i}{} For every $\btau \in K_{\BT}$ and $\bv\in \HoneDirichlet$, we  have that
\begin{equation*}
\gapext(\bv,\btau) = \overline{I}(\bv)-D(\btau)
+ (\tfrac{1}{d}\mathrm{tr}(\btau), \nabla \bv)_\Omega\,;
\end{equation*}
\item[(ii)]\hypertarget{eq:stokes_gap2.ii}{} 
For every $\bv \in \HoneDirichlet$, we have that 
\begin{equation*}
\rho^2_{\overline{I}}(\bv) = 
\overline{I}(\bv)-I(\bu) +
(\tfrac{1}{d}\mathrm{tr}(\BT)\BI, \nabla\bv)_{\Omega}\,.
\end{equation*}
\end{itemize}
\end{lemma}
\begin{proof}
\emph{ad (\hyperlink{eq:stokes_gap2.i}{i})}. For every $\bv\in \HoneDirichlet$, using \eqref{eq:stokes_stress_constraint}, the continuous integration-by-parts formula (since $\BT-\BF\in \HdivN$), \eqref{eq:Neumann_condition_weak},  $\BT=\dev\BT+\tfrac{1}{d}\textup{tr}(\BT)$ a.e.\ in $\Omega$, $(\BT,\nabla \bu)_{\Omega}=(\dev\BT,\nabla \bu)_{\Omega}$ (since $\textup{tr}(\nabla \bu)=\diver\bu=0$ a.e.\ in $\Omega$), the convex optimality relation \eqref{eq:stokes_optimality_continuous_3}, and a binomial formula, we find that
\begin{align*}
\overline{I}(\bv)\hspace{-0.125em}-\hspace{-0.125em}I(\bu)&=\tfrac{\nu}{2}\|\nabla (\bv\hspace{-0.125em}+\hspace{-0.125em}\uDirichlet)\|_{\Omega}^2-\tfrac{\nu}{2}\|\nabla (\bu\hspace{-0.125em}+\hspace{-0.125em}\uDirichlet)\|_{\Omega}^2\hspace{-0.125em}-\hspace{-0.125em}(\bff,\bv\hspace{-0.125em}-\hspace{-0.125em}\bu)_{\Omega}-(\BF,\nabla(\bv\hspace{-0.125em}-\hspace{-0.125em}\bu))_{\Omega}\\&=\tfrac{\nu}{2}\|\nabla (\bv\hspace{-0.125em}+\hspace{-0.125em}\uDirichlet)\|_{\Omega}^2\hspace{-0.125em}-\hspace{-0.125em}\tfrac{\nu}{2}\|\nabla (\bu\hspace{-0.125em}+\hspace{-0.125em}\uDirichlet)\|_{\Omega}^2+(\textup{div}\,(\BT\hspace{-0.125em}-\hspace{-0.125em}\BF),\bv\hspace{-0.125em}-\hspace{-0.125em}\bu)_{\Omega}\hspace{-0.125em}-\hspace{-0.125em}(\BF,\nabla(\bv\hspace{-0.125em}-\hspace{-0.125em}\bu))_{\Omega}
\\&=\tfrac{\nu}{2}\|\nabla (\bv\hspace{-0.125em}+\hspace{-0.125em}\uDirichlet)\|_{\Omega}^2\hspace{-0.125em}-\hspace{-0.125em}\tfrac{\nu}{2}\|\nabla (\bu\hspace{-0.125em}+\hspace{-0.125em}\uDirichlet)\|_{\Omega}^2\hspace{-0.125em}-\hspace{-0.125em}(\BT,\nabla(\bv\hspace{-0.125em}-\hspace{-0.125em}\bu))_{\Omega}
\\&=\tfrac{\nu}{2}\|\nabla (\bv\hspace{-0.125em}+\hspace{-0.125em}\uDirichlet)\|_{\Omega}^2\hspace{-0.125em}-\hspace{-0.125em}\tfrac{\nu}{2}\|\nabla (\bu\hspace{-0.125em}+\hspace{-0.125em}\uDirichlet)\|_{\Omega}^2\hspace{-0.125em}-\hspace{-0.125em}(\dev\BT,\nabla(\bv\hspace{-0.125em}-\hspace{-0.125em}\bu))_{\Omega} 
\hspace{-0.125em}-\hspace{-0.125em} (\tfrac{1}{d}\tr(\BT)\BI, \nabla\bv)_\Omega 
\\&=\tfrac{\nu}{2}\|\nabla (\bv\hspace{-0.125em}+\hspace{-0.125em}\uDirichlet)\|_{\Omega}^2\hspace{-0.125em}-\hspace{-0.125em}\tfrac{\nu}{2}\|\nabla (\bu\hspace{-0.125em}+\hspace{-0.125em}\uDirichlet)\|_{\Omega}^2\hspace{-0.125em}-\hspace{-0.125em}\nu(\nabla(\bu+\uDirichlet),\nabla(\bv\hspace{-0.125em}-\hspace{-0.125em}\bu))_{\Omega} 
\hspace{-0.125em}-\hspace{-0.125em} (\tfrac{1}{d}\tr(\BT)\BI, \nabla\bv)_\Omega
\\&=\tfrac{\nu}{2}\|\nabla (\bv\hspace{-0.125em}+\hspace{-0.125em}\uDirichlet)\|_{\Omega}^2\hspace{-0.125em}-\hspace{-0.125em}\nu(\nabla(\bu\hspace{-0.125em}+\hspace{-0.125em}\uDirichlet),\nabla(\bv\hspace{-0.125em}+\hspace{-0.125em}\uDirichlet))_{\Omega}\hspace{-0.125em}+\hspace{-0.125em}\tfrac{\nu}{2}\|\nabla (\bu\hspace{-0.125em}+\hspace{-0.125em}\uDirichlet)\|_{\Omega}^2\hspace{-0.125em}- \hspace{-0.125em}(\tfrac{1}{d}\tr(\BT)\BI, \nabla\bv)_\Omega 
\\&=\tfrac{\nu}{2}\|\nabla(\bv\hspace{-0.125em}-\hspace{-0.125em}\bu)\|_{\Omega}^2 \hspace{-0.125em}-\hspace{-0.125em} (\tfrac{1}{d}\tr(\BT)\BI, \nabla\bv)_\Omega \,,
\end{align*}

\emph{ad (\hyperlink{eq:stokes_gap2.ii}{ii}).} For every $\btau \in K_{\BT}$ and $\bv\in \HoneDirichlet$, using \eqref{eq:stokes_stress_constraint}, the continuous integration-by-parts formula (since $\btau-\BF\in \HdivN$), \eqref{eq:Neumann_condition_weak}, $\btau=\dev\btau+\tfrac{1}{d}\textup{tr}(\btau)$ a.e.\ in $\Omega$, $(\btau,\nabla \widehat{\bu}_D)_{\Omega}=(\dev\btau,\nabla \widehat{\bu}_D)_{\Omega}$ (since $\textup{tr}(\nabla \widehat{\bu}_D)=\diver\widehat{\bu}_D=0$ a.e.\ in $\Omega$), and a binomial formula, we find that
\begin{align*}
\overline{I}(\bv)-D(\btau)&= \tfrac{\nu}{2}\| \nabla (\bv\hspace{-0.125em}+\hspace{-0.125em}\uDirichlet)\|_{\Omega}^2-(\bff,\bv)_{\Omega}\hspace{-0.125em}-\hspace{-0.125em}(\BF,\nabla\bv)_{\Omega}
\hspace{-0.125em}+\hspace{-0.125em}\tfrac{1}{2\nu}\|\dev \btau\|_{\Omega}^2\hspace{-0.125em}-\hspace{-0.125em}(\dev\btau,\nabla\uDirichlet)_{\Omega}\\&
= \tfrac{\nu}{2}\| \nabla (\bv\hspace{-0.125em}+\hspace{-0.125em}\uDirichlet)\|_{\Omega}^2\hspace{-0.125em}+\hspace{-0.125em}(\textup{div}\,(\btau\hspace{-0.125em}-\hspace{-0.125em}\BF),\bv)_{\Omega}\hspace{-0.125em}-\hspace{-0.125em}(\BF,\nabla\bv)_{\Omega}
\hspace{-0.125em}+\hspace{-0.125em}\tfrac{1}{2\nu}\| \dev \btau\|_{\Omega}^2-(\dev\btau,\nabla\uDirichlet)_{\Omega}
\\&
= \tfrac{\nu}{2}\| \nabla (\bv\hspace{-0.125em}+\hspace{-0.125em}\uDirichlet)\|_{\Omega}^2\hspace{-0.125em}-\hspace{-0.125em}(\dev \btau,\nabla (\bv\hspace{-0.125em}+\hspace{-0.125em}\uDirichlet))_{\Omega}
\hspace{-0.125em}-\hspace{-0.125em} (\tfrac{1}{d}\tr(\btau), \nabla \bv)_\Omega \hspace{-0.125em}+\hspace{-0.125em}\tfrac{1}{2\nu}\|\dev \btau \|_{\Omega}^2
\\&
= \tfrac{\nu}{2}\| \nabla (\bv\hspace{-0.125em}+\hspace{-0.125em}\uDirichlet)\hspace{-0.125em}-\hspace{-0.125em}\smash{\tfrac{1
}{\nu}}\dev \btau\|_{\Omega}^2 \hspace{-0.125em} -\hspace{-0.125em} (\tfrac{1}{d}\tr(\btau), \nabla \bv)_\Omega\,,
\end{align*}
which is the claimed representation of the extended primal-dual gap estimator \eqref{eq:stokes_gap2}.
\end{proof}

Eventually, we are in position to establish an analogous equivalence to \eqref{eq:elasticity_aposteriori2}.

\begin{theorem}\label{thm:prager_synge_identity_stokes2}
For every $\bv\hspace{-0.1em}\in \hspace{-0.1em}\HoneDirichlet$ and $\btau\hspace{-0.1em}\in\hspace{-0.1em} K_{\BT}$, we have the following identity and equivalence
\begin{equation}\label{eq:stokes_aposteriori1}
\begin{aligned}
\rhototext(\bv,\btau) &= \gapext(\bv,\btau)
+ \tfrac{1}{d}(\tr(\BT - \btau)\BI,\nabla\bv)_\Omega
\sim \gapext(\bv,\btau)\,,
\end{aligned}
\end{equation}
where the equivalence  explicitly written reads
\begin{equation}\label{eq:stokes_aposteriori2}
\tfrac{\nu}{2}\|\nabla \bv - \nabla \bu\|_{\Omega}^2
+
\tfrac{1}{2\nu}\|\dev \btau - \dev \BT\|_{\Omega}^2
\sim
\tfrac{\nu}{2}\|\nabla \bv + \nabla \uDirichlet - \smash{\tfrac{1}{\nu}}\dev \btau\|_{\Omega}^2\,,
\end{equation}
the implicit constants in $\sim$ only depends on the dev-div constant $c_{\mathrm{DD}}>0$.
\end{theorem}
\begin{remark}
Note that another consequence of the dev-div inequality \eqref{eq:dev_div} is that, in the \emph{a posteriori} error estimate \eqref{eq:stokes_aposteriori2}, we can replace $\|\dev(\btau - \BT)\|_{\Omega}^2$ with the full norm $\|\btau - \BT\|_{\Omega}^2$, yielding then an estimate for the pressure $p= -\frac{1}{d}\tr\BT$.
\end{remark}
\begin{proof}[of Theorem \ref{thm:prager_synge_identity_stokes2}]
The proof is analogous to that of Theorem \ref{thm:prager_synge_identity_elasticity}, noting that $\overline{I}(\bu)=I(\bu)$ (since $\textup{div}\,\bu=0$ a.e.\ in $\Omega$) and that
\begin{equation*}
\tfrac{1}{d}(\tr(\BT - \btau)\BI,\nabla \bv)_\Omega
=
\tfrac{1}{d}(\tr(\BT - \btau)\BI,\nabla \bv+\nabla \uDirichlet - \tfrac{1}{\nu}\dev \btau)_\Omega\,,
\end{equation*}
where we used that $\tr(\BT - \btau)\BI\perp_{\mathbb{L}^2}\nabla \uDirichlet - \tfrac{1}{\nu}\dev \btau$  as $\diver \uDirichlet = 0$ a.e.\ in $\Omega$.
\end{proof}

\subsection{Data oscillation}\label{sec:oscillation}

\hspace{5mm}In general, the data $\bff\in \BL^2(\Omega)$ and $\BF\in \mathbb{L}^2(\Omega)$ will not be element-wise constant and, {thus},  the previously obtained results do not apply immediately.
However, a similar argument based on modified energies  yields an \emph{a posteriori} error estimate valid also for the general loads. For {simplicity}, let us return to the Poisson  problem \eqref{def:poisson_primal}.

Given that the flux reconstruction $\bq_h\in \bff_h+\RTNeumann$ satisfies $\diver(\bq_h - \bff_h) = -f_h$ a.e.\ in $\Omega$, it is not admissible for the dual energy functional \eqref{eq:dual_energy_laplace}, but admissible for the modified dual energy functional $\overline{D}_h\colon \bff_h+H_N(\textup{div};\Omega)\to \mathbb{R}\cup\{-\infty\}$, for every $\br_h\in \bff_h+H_N(\textup{div};\Omega)$ defined by
\begin{equation}
\overline{D}_h(\br_h) \coloneqq -\tfrac{1}{2}\|\br_h\|_{\Omega}^2 
+ (\br_h,\nabla \widehat{u}_D)
- \characteristic{-f_h}^{\Omega}(\diver(\br_h- \bff_h))\,.
\end{equation}
Then, with analogous arguments as above, for every $v\in H^1_D(\Omega)$ and $\br_h\in \bff_h+H_N(\textup{div};\Omega)$ with $\diver(\bq_h - \bff_h) = -f_h$ a.e.\ in $\Omega$, we  arrive at the modified primal-dual gap identity
\begin{align}
\begin{aligned}
\tfrac{1}{2}\|\nabla u - \nabla v\|^2_\Omega 
+ \tfrac{1}{2}\|\bq - \br\|^2_\Omega
&=
\tfrac{1}{2}\|\nabla(u+\widehat{u}_D) - \br\|^2_\Omega 
\\&\quad+(f-f_h, u-v)_\Omega
\\&\quad + (\bff - \bff_h, \nabla(u-v))_\Omega\,.
\end{aligned}
\end{align}
Recalling that, by the local Poincaré inequality (\textit{cf}.\ \cite[Lem.\ 3.24]{EG21}), for every $T\hspace{-0.15em}\in\hspace{-0.15em} \mathcal{T}_h$ and ${v\hspace{-0.15em}\in\hspace{-0.15em} H^1(T)}$, we have that $\|v - \Pi_h v\|_T \leq 
\frac{h_T}{\pi} \|\nabla v\|_T$,
for every $v\in H^1_D(\Omega)$ and $\br_h\in \bff_h+H_N(\textup{div};\Omega)$ with $\diver(\bq_h - \bff_h) = -f_h$ a.e.\ in $\Omega$, 
Young's inequality yields the \emph{a posteriori} error estimate
\begin{equation}
\tfrac{1}{2}\|\nabla u - \nabla v\|^2_\Omega 
+ \tfrac{1}{2}\|\bq - \br\|^2_\Omega
\lesssim
\tfrac{1}{2}\|\nabla(u+\widehat{u}_D) - \br\|^2_\Omega 
+ \sum_{T\in \tria} \mathrm{osc}^2_T(f,\bff)\,,
\end{equation}
where the \textit{local data oscillation term} 
is defined by $\mathrm{osc}_T^2(f,\bff) \hspace{-0.15em}\coloneqq \hspace{-0.15em}
\tfrac{h_T^2}{\pi^2}\|f-f_h\|^2_T 
+ \| \bff  - \bff_h\|^2_T$ for all ${T\hspace{-0.15em}\in \hspace{-0.15em}\mathcal{T}_h}$.

\section{Numerical experiments}\label{sec:experiments}

\hspace{5mm}In this section, we 
present numerical experiments that support the theoretical findings obtained in the previous sections.
All numerical experiments were implemented in \texttt{firedrake} (\textit{cf}.~\cite{Firedrake}).
The arising linear systems were solved using the \texttt{MUMPS} library (\textit{cf}.\ \cite{MUMPS}).

Even though the \emph{a posteriori}  error identities/estimates derived in this work are of global {nature},
all derived \emph{a posteriori} error estimators have integral representations with non-negative {integrands} (\textit{cf}.\ Lemma \ref{lem:discrete_primal_dual_gap_estimator_stokes} and Lemma \ref{lem:primal_dual_gap_estimator_elasticity}, respectively), measuring the violation of the convex optimality relations \eqref{eq:stokes_optimality_continuous_3} and \eqref{eq:elasticity_optimality_conditions.3}, respectively.
This  allows the definition of local mesh-refinement {indicators} that can be employed for local mesh-refinement. 
More precisely, for the incompressible Stokes problem (\textit{cf}.\ Section \ref{sec:stokes_primal}) and the Navier--Lam\'e  problem (\textit{cf}.\ Section \ref{subsec:elasticity_primal}), we intro\-duce the following local mesh-refinement indicators:
\begin{itemize}[noitemsep,topsep=2pt,leftmargin=!,labelwidth=\widthof{$\bullet$}]
\item[$\bullet$] \emph{Incompressible Stokes problem:} For every $T\in \mathcal{T}_h$, we introduce the local refinement indicator $\gaplocal\colon K_{\bu}\times K_{\BT}\to \mathbb{R}$, for every $\bv\in K_{\bu}$ and $\btau\in  K_{\BT}$ defined by
\begin{align}\label{def:indicator_stokes}
\gaplocal(\bv,\btau)
&= 
\tfrac{\nu}{2}\|\nabla \bv + \nabla \uDirichlet - \tfrac{1}{\nu}\dev \btau \|^2_T\,;
\end{align}
\item[$\bullet$] \emph{Navier--Lam\'e problem:} For every $T\in \mathcal{T}_h$, we introduce the local refinement indicator $\gapextlocal\colon \HoneDirichlet\times \overline{K}_{\bsigma}\to \mathbb{R}$, for every $\bv\in \HoneDirichlet$ and $\btau\in \overline{K}_{\bsigma}$ defined by
\begin{align}\label{def:indicator_elasticity}
\gapextlocal(\bv,\btau)
&= 
\tfrac{1}{2}\|\smash{\bbC^{\smash{\frac{1}{2}}}}(\symgrad{\bv+\uDirichlet} - \smash{\bbC^{-1}}\btau)\|^2_T\,.
\end{align}
\end{itemize}

The numerical experiments not necessarily make use of element-wise constant data ${\bff\hspace*{-0.175em} \in\hspace*{-0.175em} (\Pbroken{0})^d}$ and $\BF\in \PbrokenM{0}$. Instead, in all numerical experiments, we employ $\bu_D^h\coloneqq \mathcal{I}_h^{cr}\bu_D\in \CR$, $\bff_h\coloneqq \Pi_h\bff\in (\Pbroken{0})^d$, and $\BF_h\coloneqq \Pi_h\BF\in (\Pbroken{0})^{d\times d}$. 
For this reason, 
in view of Section \ref{sec:oscillation}, for every $T\in \mathcal{T}_h$, we introduce the \emph{local data oscillation estimator}
\begin{equation}\label{def:osc_exp}
\osclocal(\bff,\BF) \coloneqq \tfrac{h_T^2}{\pi^2}\|\bff -\bff_h\|^2_T
+ \|\BF - \BF_h\|^2_T\,.
\end{equation}

In all numerical experiments that make use of local mesh-refinement (based on \eqref{def:indicator_stokes} or \eqref{def:indicator_elasticity}), we implement the following local mesh-refinement algorithm.

\begin{algorithm}\label{alg:adaptive}
Choose an initial triangulation $\mathcal{T}_0$, a maximum number of iterations ${k_{\mathrm{max}}\in \mathbb{N}}$,
an error tolerance $\varepsilon_{\mathrm{STOP}}\hspace{-0.05em}>\hspace{-0.05em}0$, and a refinement parameter $\theta\hspace{-0.05em}\in\hspace{-0.05em} (0,1]$.
Then, for $k\hspace{-0.05em}\in \hspace{-0.05em}\{1,2,\ldots, k_{\mathrm{max}}\}$, perform the following iteration loop:

\begin{description}[noitemsep,topsep=1pt,labelwidth=\widthof{('\textsc{Estimate}')},leftmargin=!,font=\normalfont\itshape]
\item[('\textsc{Solve}')]\hypertarget{Solve}{}
Compute a discrete primal solution $\bu_{h_k} \hspace{-0.2em}\in \hspace{-0.175em} V_D^{h_k}$ and a discrete dual solution ${\BT_{h_k} \hspace{-0.2em} \in \hspace{-0.175em} \BF_{h_k}}$ $ + \Sigma_N^{h_k}$ (resp. $\bsigma_{h_k}^*\in \BF_{h_k} + \Sigma_N^{h_k}$) via the reconstruction formula \eqref{eq:stokes_marini} (resp. \eqref{prop:marini_elasticity.1}). Post-process the discrete primal solution $\bu_{h_k} \in V_D^{h_k}$ into an admissible vector field $\widehat{\bu}_{h_k}\in \BH^1_D(\Omega)$ by nodal averaging (\textit{cf}.\ \cite[Sec.\ 22.2]{EG21}).

\item[('\textsc{Estimate}')]\hypertarget{Estimate}{}
For each $T\in \mathcal{T}_{h_k}$, compute the local refinement indicator $\eta^2_T\coloneqq\eta^2_T(\widehat{\bu}_{h_k}, \BT_{h_k})+\osclocal(\bff,\BF)$ ($\eta^2_T\coloneqq\overline{\eta}^2_T(\widehat{\bu}_{h_k}, \bsigma_{h_k}^*)+\osclocal(\bff,\BF)$, respectively).
If ${\sum_{T\in \mathcal{T}_{h_k}}{\eta^2_T}< \varepsilon_{\mathrm{STOP}}}$, then \textup{STOP}.
Otherwise, continue with Step (\hyperlink{Mark}{'\textsc{Mark}'});

\item[('\textsc{Mark}')]\hypertarget{Mark}{}
Mark all elements $T\in \mathcal{T}_{h_k}$ for refinement which satisfy a max-marking condition, \textit{i.e.},  there holds
\begin{equation*}
\eta^2_T \geq \theta 
\max_{\widehat{T}\in \mathcal{T}_{h_k}}{\big\{\eta^2_{\widehat{T}}\big\}}\,;
\end{equation*}

\item[('\textsc{Refine}')]\hypertarget{Refine}{}
Employ the bisection algorithm from \cite{ANP.2000}
to obtain the next triangulation $\mathcal{T}_{h_{k+1}}$.
Increase $k\mapsto k+1$ and continue with Step (\hyperlink{Solve}{'\textsc{Solve}'}).
\end{description}
\end{algorithm}
Unless otherwise specified, we will employ $\theta = \frac{1}{2}$ in the experiments with local mesh-refinement.
Note that $\theta = 0$ corresponds to uniform mesh-refinement, \textit{i.e.}, each element is bisection refined.\pagebreak

\subsection{Incompressible Stokes problem: \emph{a priori} error analysis}

In this first numerical experment, we  review the validity of the discrete primal-dual gap identity \eqref{eq:discrete_praguer_synge_identity_stokes}. To this end, let the computational domain be given via $\Omega\! \coloneqq\! (0,1)^2$ with {Dirichlet} {boundaries}  at the left and right walls, \textit{i.e.}, $\Gamma_D \coloneqq \{0\}\times (0,1) \cup \{1\}\times (0,1)$, as well as 
Neumann boundaries at the top and bottom walls, \textit{i.e.}, 
$
\Gamma_N \coloneqq  (0,1)\times  \{0\} \cup(0,1)\times  \{1\}$, (\textit{cf}.\ Figure \ref{fig:unit_square}). Then, as manufactured (original) velocity vector field  $\bu_{\textup{orig}}\in \BH^1(\Omega)$ and kinematic pressure $p_{\textup{orig}}\in L^2(\Omega)$, we consider a Taylor--Green vortex (\textit{cf}.\ \cite{TaylorGreen1937}), for every $x=(x_1,x_2)\in \Omega$ defined by
\begin{subequations}\label{eq:exp1_solution}
\begin{align}
\bu_{\textup{orig}}(x) &\coloneqq \smash{(\sin(\pi x_1)\cos(\pi x_2), -\cos(\pi x_1)\sin(\pi x_2))^\top\,, }\\
p_{\textup{orig}}(x) &\coloneqq \smash{\tfrac{1}{4}\{\cos(2\pi x_1)+ \sin(2\pi x_2)\}\,,}
\end{align}
\end{subequations}
respectively. Then, the load is set as $\bff = -\nu\Delta\bu_{\textup{orig}}+ \nabla p_{\textup{orig}}\in \BL^2(\Omega)$, where $\nu = \frac{1}{2}$, 
the Dirichlet boundary datum $\widehat{\bu}_D\in \BH^1(\Omega)$ such that $\diver \widehat{\bu}_D=0$ a.e.\ in $\Omega$ and $\widehat{\bu}_D=\bu_{\textup{orig}}$ q.e.\ in $\Gamma_D$, and the Neumann boundary datum 
$\BF\in \mathbb{L}^2(\Omega)$ such that ${\BF\bn=\nu \nabla \bu_{\textup{orig}}\bn}$ q.e. in $\Gamma_N$ (\textit{cf}. {Remark} \ref{rem:Neumann}).

Starting with a triangulation $\mathcal{T}_h$ with 200 elements,  the discrete incompressible Stokes problem (\textit{cf}.\ Section \ref{sec:stokes_discrete_primal}) is solved introducing a discrete kinematic pressure, which serves a Lagrange multiplier to enforce the divergence constraint.
More precisely, we seek  $(\bu_h, p_h) \in \CRDirichlet \times \mathbb{P}^0(\tria)$ 
such that for every $(\bv_h,q_h)\in \CRDirichlet\times \mathbb{P}^0(\tria)$, there holds
\begin{alignat*}{2}
\nu (\nabla_h (\bu_h+\bu_D^h), \nabla_h \bv_h)_\Omega 
- (p_h,\textup{div}_h\bv_h)_\Omega &= (\bff_h,\Pi_h\bv_h)_\Omega \,, \\ 
(\textup{div}_h \bu_h, q_h)_\Omega &= 0 \,.
\end{alignat*}
The discrete dual solution/stress $\BT_h\in K_{\BT}^h$ is found using the reconstruction formula \eqref{eq:stokes_marini}. Then, we generate random vector fields $\bv_h^i\in K_{\bu}^h$, $i\in \{1,2,3\}$, and tensors $\btau_h^i \in \R^{2\times 2}$, $i\in\{1,2,3\}$, (\textit{i.e.}, $\BT_h + \btau_h^i \in \RTNeumann$ with $\diver(\BT_h + \btau_h^i) = -\bff_h$ a.e.\ in $\Omega$ for all $i\in \{1,2,3\}$); each via sampling from a beta distribution. 
In order to quantify the validity of the discrete primal-dual gap identity \eqref{eq:discrete_praguer_synge_identity_stokes}, for each refinement level $k\in \{1,\ldots,6\}$ and sampling step $i\in \{1,2,3\}$,
we compute the \textit{relative identity errors} (\textit{cf}.\ Table \ref{tbl:discrete_prager_synge})\vspace*{-0.5mm}
\begin{equation}
\mathtt{err}^{\mathrm{iden}}_{k,i} \coloneqq 
\frac{|\eta_{\mathrm{gap},h_k}^2(\bv_{h_k}^i, \BT_{h_k} + \btau^i_{h_k}) - \rho_{\mathrm{tot},h_k}^2(\bv_{h_k}^i, \BT_{h_k} + \btau^i_{h_k})|}{\rho_{\mathrm{tot},h_k}^2(\bv_{h_k}^i, \BT_{h_k} + \btau_{h_k}^i)}\,. 
\end{equation}
Moreover, for each refinement level $k\hspace{-0.15em}\in\hspace{-0.15em} \{1,\ldots,6\}$,
Table \ref{tbl:discrete_prager_synge}  contains the primal error ${\texttt{err}_k^{\bu}\hspace{-0.15em}\coloneqq\hspace{-0.15em}\rho_I(\bu_{h_k})}$ and the dual error $\texttt{err}_k^{\BT}\coloneqq \rho_{-D}(\BT_{h_k})$,  along with experimental rates of convergence computed via $\text{\texttt{EOC}}_k(\texttt{err}_k)
\coloneqq 
(\log(\texttt{err}_k)-\log(\texttt{err}_{k-1}))/(\log(h_k)-\log(h_{k-1}))$, where $\texttt{err}_k\in \{\texttt{err}_k^{\bu},\texttt{err}_k^{\BT}\}$. In it,
we observe that  
the discrete primal-dual gap identity \eqref{eq:discrete_praguer_synge_identity_stokes} is satisfied and a linear convergence rate, which confirms to what is expected, given the smoothness of the original solution \eqref{eq:exp1_solution}.

\begin{center}
\begin{table}[H]
\renewcommand{\arraystretch}{0.05} 
%
  \centering
\begin{tabular}{@{}ccccccc@{}}
\toprule[2.0pt]
\texttt{num\_dof}$_k$ & $\texttt{err}_k^{\bu} $ & $\texttt{EOC}_k(\texttt{err}_k^{\bu})$ & $\texttt{err}_k^{\BT}$ & $\texttt{EOC}_k(\texttt{err}_k^{\BT})$\vphantom{$X^{X^X}$} & $i$ & $\mathtt{err}^{\mathrm{iden}}_{k,i}$\\[1.25mm]
\toprule[1.5pt]
\multirow{3}{*}{840} & \multirow{3}{*}{0.1830} & \multirow{3}{*}{-} & \multirow{3}{*}{0.1573} & \multirow{3}{*}{-} & 1 & $8.9\times 10^{-4}$\vphantom{$X^{X^x}$} \\ 
&&&&&2 & $7.0\times 10^{-3}$ \\
&&&&&3 & $6.9\times 10^{-4}$ \\[0.5mm]
\midrule
\multirow{3}{*}{3 280} & \multirow{3}{*}{0.0920} & \multirow{3}{*}{1.0091} & \multirow{3}{*}{0.0783} & \multirow{3}{*}{1.0238} & 1 & $3.0\times 10^{-3}$\vphantom{$X^{X^x}$}\\ 
&&&&&2 & $1.8\times 10^{-3}$ \\
&&&&&3 & $1.8\times 10^{-4}$\\[0.5mm]
\midrule
\multirow{3}{*}{12 960} & \multirow{3}{*}{0.0461} & \multirow{3}{*}{1.0054} & \multirow{3}{*}{0.0391} & \multirow{3}{*}{1.0114} & 1 & $4.7\times 10^{-5}$\vphantom{$X^{X^x}$} \\
&&&&&2& $1.5\times 10^{-4}$ \\
&&&&&3& $2.6\times 10^{-4}$ \\[0.5mm]
\midrule
\multirow{3}{*}{51 520} & \multirow{3}{*}{0.0231} & \multirow{3}{*}{1.0029} & \multirow{3}{*}{0.0195} & \multirow{3}{*}{1.0054} & 1 & $7.9\times 10^{-6}$\vphantom{$X^{X^x}$} \\
&&&&&2& $4.8\times 10^{-5}$ \\
&&&&&3& $3.5\times 10^{-5}$ \\[0.5mm]
\midrule
\multirow{3}{*}{205 440} & \multirow{3}{*}{0.0115} & \multirow{3}{*}{1.0015} & \multirow{3}{*}{0.0098} & \multirow{3}{*}{1.0027} & 1 & $8.2\times 10^{-8}$\vphantom{$X^{X^x}$}\\ 
&&&&&2& $1.1\times 10^{-5}$ \\
&&&&&3 & $2.8\times 10^{-5}$\\[0.5mm]
\midrule
\multirow{3}{*}{820 480} & \multirow{3}{*}{0.0058} & \multirow{3}{*}{1.008} & \multirow{3}{*}{0.0049} & \multirow{3}{*}{1.0013} & 1& $6.3\times 10^{-7}$\vphantom{$X^{X^x}$} \\
&&&&&2 & $2.3\times 10^{-7}$\\
&&&&&3& $7.5\times 10^{-7}$ \\[0.5mm]
\bottomrule[2pt]
\end{tabular} \vspace*{-1.5mm}
\captionof{table}{Primal, dual, and relative identity errors.}
\label{tbl:discrete_prager_synge}
\end{table}
\end{center}
\vspace{-2.5cm}

\begin{center}
\begin{figure}
  \centering

\tikzset {_c5juh8s6p/.code = {\pgfsetadditionalshadetransform{ \pgftransformshift{\pgfpoint{0 bp } { 0 bp }  }  \pgftransformscale{1 }  }}}
\pgfdeclareradialshading{_vov8r40tj}{\pgfpoint{0bp}{0bp}}{rgb(0bp)=(1,1,1);
rgb(2.142857142857143bp)=(1,1,1);
rgb(14.285714285714285bp)=(0.89,0.89,0.89);
rgb(14.285714285714285bp)=(0.89,0.89,0.89);
rgb(18.482142857142858bp)=(0.86,0.86,0.86);
rgb(25bp)=(0.82,0.82,0.82);
rgb(400bp)=(0.82,0.82,0.82)}


\tikzset{
pattern size/.store in=\mcSize, 
pattern size = 5pt,
pattern thickness/.store in=\mcThickness, 
pattern thickness = 0.3pt,
pattern radius/.store in=\mcRadius, 
pattern radius = 1pt}
\makeatletter
\pgfutil@ifundefined{pgf@pattern@name@_ttb5rjq11}{
\pgfdeclarepatternformonly[\mcThickness,\mcSize]{_ttb5rjq11}
{\pgfqpoint{0pt}{-\mcThickness}}
{\pgfpoint{\mcSize}{\mcSize}}
{\pgfpoint{\mcSize}{\mcSize}}
{
\pgfsetcolor{\tikz@pattern@color}
\pgfsetlinewidth{\mcThickness}
\pgfpathmoveto{\pgfqpoint{0pt}{\mcSize}}
\pgfpathlineto{\pgfpoint{\mcSize+\mcThickness}{-\mcThickness}}
\pgfusepath{stroke}
}}
\makeatother


\tikzset{
pattern size/.store in=\mcSize, 
pattern size = 5pt,
pattern thickness/.store in=\mcThickness, 
pattern thickness = 0.3pt,
pattern radius/.store in=\mcRadius, 
pattern radius = 1pt}
\makeatletter
\pgfutil@ifundefined{pgf@pattern@name@_egox6bz96}{
\pgfdeclarepatternformonly[\mcThickness,\mcSize]{_egox6bz96}
{\pgfqpoint{0pt}{-\mcThickness}}
{\pgfpoint{\mcSize}{\mcSize}}
{\pgfpoint{\mcSize}{\mcSize}}
{
\pgfsetcolor{\tikz@pattern@color}
\pgfsetlinewidth{\mcThickness}
\pgfpathmoveto{\pgfqpoint{0pt}{\mcSize}}
\pgfpathlineto{\pgfpoint{\mcSize+\mcThickness}{-\mcThickness}}
\pgfusepath{stroke}
}}
\makeatother


\tikzset{
pattern size/.store in=\mcSize, 
pattern size = 5pt,
pattern thickness/.store in=\mcThickness, 
pattern thickness = 0.3pt,
pattern radius/.store in=\mcRadius, 
pattern radius = 1pt}
\makeatletter
\pgfutil@ifundefined{pgf@pattern@name@_1epxomdb1}{
\pgfdeclarepatternformonly[\mcThickness,\mcSize]{_1epxomdb1}
{\pgfqpoint{0pt}{-\mcThickness}}
{\pgfpoint{\mcSize}{\mcSize}}
{\pgfpoint{\mcSize}{\mcSize}}
{
\pgfsetcolor{\tikz@pattern@color}
\pgfsetlinewidth{\mcThickness}
\pgfpathmoveto{\pgfqpoint{0pt}{\mcSize}}
\pgfpathlineto{\pgfpoint{\mcSize+\mcThickness}{-\mcThickness}}
\pgfusepath{stroke}
}}
\makeatother


\tikzset{
pattern size/.store in=\mcSize, 
pattern size = 5pt,
pattern thickness/.store in=\mcThickness, 
pattern thickness = 0.3pt,
pattern radius/.store in=\mcRadius, 
pattern radius = 1pt}
\makeatletter
\pgfutil@ifundefined{pgf@pattern@name@_7mca2zvg8}{
\pgfdeclarepatternformonly[\mcThickness,\mcSize]{_7mca2zvg8}
{\pgfqpoint{0pt}{-\mcThickness}}
{\pgfpoint{\mcSize}{\mcSize}}
{\pgfpoint{\mcSize}{\mcSize}}
{
\pgfsetcolor{\tikz@pattern@color}
\pgfsetlinewidth{\mcThickness}
\pgfpathmoveto{\pgfqpoint{0pt}{\mcSize}}
\pgfpathlineto{\pgfpoint{\mcSize+\mcThickness}{-\mcThickness}}
\pgfusepath{stroke}
}}
\makeatother
\tikzset{every picture/.style={line width=0.75pt}} 

\begin{tikzpicture}[x=0.75pt,y=0.75pt,yscale=-1,xscale=1]

\path  [shading=_vov8r40tj,_c5juh8s6p] (60.25,50.19) -- (180,50.19) -- (180,169.94) -- (60.25,169.94) -- cycle ; 
\draw   (60.25,50.19) -- (180,50.19) -- (180,169.94) -- (60.25,169.94) -- cycle ; 

\draw  [draw opacity=0][pattern=_ttb5rjq11,pattern size=3.75pt,pattern thickness=0.75pt,pattern radius=0pt, pattern color={shamrockgreen}] (180,50.19) -- (190.25,50.19) -- (190.25,170) -- (180,170) -- cycle ;
\draw  [draw opacity=0][pattern=_egox6bz96,pattern size=3.75pt,pattern thickness=0.75pt,pattern radius=0pt, pattern color={shamrockgreen}] (50,50.19) -- (60.25,50.19) -- (60.25,170) -- (50,170) -- cycle ;
\draw  [draw opacity=0][pattern=_1epxomdb1,pattern size=3.75pt,pattern thickness=0.75pt,pattern radius=0pt, pattern color={denim}] (60.25,39.75) -- (180,39.75) -- (180,50.19) -- (60.25,50.19) -- cycle ;
\draw  [draw opacity=0][pattern=_7mca2zvg8,pattern size=3.75pt,pattern thickness=0.75pt,pattern radius=0pt, pattern color={denim}] (60.25,169.94) -- (180,169.94) -- (180,180.39) -- (60.25,180.39) -- cycle ;
\draw [color={rgb, 255:red, 206; green, 206; blue, 206 }  ,draw opacity=1 ] [dash pattern={on 0.75pt off 0.75pt on 0.75pt off 0.75pt}]  (20.25,170) -- (60.25,170) ;
\draw [color={rgb, 255:red, 206; green, 206; blue, 206 }  ,draw opacity=1 ] [dash pattern={on 0.75pt off 0.75pt on 0.75pt off 0.75pt}]  (20.25,50.19) -- (60.25,50.19) ;
\draw [color={rgb, 255:red, 206; green, 206; blue, 206 }  ,draw opacity=1 ] [dash pattern={on 0.75pt off 0.75pt on 0.75pt off 0.75pt}]  (60.25,169.94) -- (60.25,209.94) ;
\draw [color={rgb, 255:red, 206; green, 206; blue, 206 }  ,draw opacity=1 ] [dash pattern={on 0.75pt off 0.75pt on 0.75pt off 0.75pt}]  (180,169.94) -- (180,209.94) ;
\draw [color={rgb, 255:red, 128; green, 128; blue, 128 }  ,draw opacity=1 ]   (20.25,50.19) -- (20.25,170) ;
\draw [shift={(20.25,170)}, rotate = 270] [color={rgb, 255:red, 128; green, 128; blue, 128 }  ,draw opacity=1 ][line width=0.75]    (0,5.59) -- (0,-5.59)   ;
\draw [shift={(20.25,50.19)}, rotate = 270] [color={rgb, 255:red, 128; green, 128; blue, 128 }  ,draw opacity=1 ][line width=0.75]    (0,5.59) -- (0,-5.59)   ;
\draw [color={rgb, 255:red, 128; green, 128; blue, 128 }  ,draw opacity=1 ]   (60.25,209.94) -- (180,209.94) ;
\draw [shift={(180,209.94)}, rotate = 180] [color={rgb, 255:red, 128; green, 128; blue, 128 }  ,draw opacity=1 ][line width=0.75]    (0,5.59) -- (0,-5.59)   ;
\draw [shift={(60.25,209.94)}, rotate = 180] [color={rgb, 255:red, 128; green, 128; blue, 128 }  ,draw opacity=1 ][line width=0.75]    (0,5.59) -- (0,-5.59)   ;
\draw   (60.25,50.25) -- (180,50.25) -- (180,170) -- (60.25,170) -- cycle ;

\draw (192.67,100) node [anchor=north west][inner sep=0.75pt]  [color={shamrockgreen}  ,opacity=1 ]  {$\Gamma _{D}$};
\draw (112.5,22.5) node [anchor=north west][inner sep=0.75pt]  [color={denim}  ,opacity=1 ]  {$\Gamma _{N}$};
\draw (112.5,182.32) node [anchor=north west][inner sep=0.75pt]  [color={denim}  ,opacity=1 ]  {$\Gamma _{N}$};
\draw (170.75,216) node [anchor=north west][inner sep=0.75pt]    {$1.0$};
\draw (50.25,216) node [anchor=north west][inner sep=0.75pt]    {$0.0$};
\draw (-6,164) node [anchor=north west][inner sep=0.75pt]    {$0.0$};
\draw (-6,44) node [anchor=north west][inner sep=0.75pt]    {$1.0$};
\draw (113,99.57) node [anchor=north west][inner sep=0.75pt]  [font=\Large]  {$\Omega $};
\draw (28.67,100) node [anchor=north west][inner sep=0.75pt]  [color={shamrockgreen}  ,opacity=1 ]  {$\Gamma _{D}$};

\end{tikzpicture}
\captionof{figure}{\hspace{-0.25mm}Square \hspace{-0.15mm}domain \hspace{-0.15mm}$\Omega$ \hspace{-0.15mm}(\textcolor{gray}{gray}), Dirichlet boundary $\Gamma_D$ (\textcolor{shamrockgreen}{green}), and Neumann boundary $\Gamma_N$ (\textcolor{denim}{blue}).}
\label{fig:unit_square}
\end{figure}
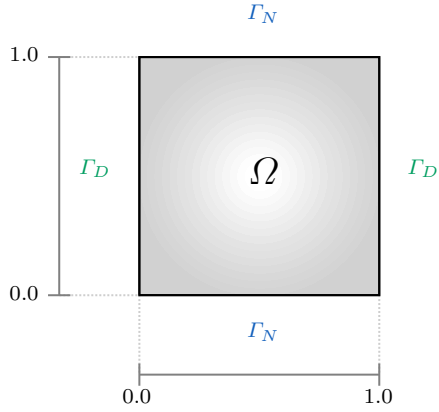
\end{center}

\subsection{Incompressible Stokes problem: \emph{a posteriori} error analysis}

In the second numerical experiment, we review the performance of the local refinement indicators \eqref{def:indicator_stokes} in the local mesh-refinement algorithm (\textit{cf}.\ Algorithm \ref{alg:adaptive}). To this end, let the computational domain be given via the $L$-shaped domain $\Omega \coloneqq (-1,1)^2 \setminus ([0,1]\times [-1,0])$ with Dirichlet boundaries at all boundary parts, \textit{i.e.}, $\Gamma_D\coloneqq\partial\Omega$, (\textit{cf}.\ Figure \ref{fig:stokes_rates}(\textit{right})). Then, for the manufactured (original) velocity vector field $\bu_{\textup{orig}}\in \BH^1(\Omega)$ and kinematic pressure $p_{\textup{orig}}\in L^2(\Omega)$, we consider the example derived in \cite{Ver.1989}, which results in a trivial load, \textit{i.e.}, $\load = \bzero^*$ in $\HminusDirichlet$, in polar coordinates, for every $(r,\theta)\in [0,+\infty) \times [0,2\pi)$  defined by 
\begin{align*}
\bu_{\textup{orig}}(r,\theta) &\coloneqq
r^\alpha\{(1+\alpha)\psi(\theta)(\sin(\theta) ,
-\cos(\theta))^\top+\psi'(\theta)(\cos(\theta),\sin(\theta))^\top\}\,, \\ 
p_{\textup{orig}}(r,\theta) &\coloneqq
\tfrac{r^{\alpha -1}}{\alpha -1} \{ (\alpha+1)^2 \psi'(\theta) + \psi'''(\theta)\}\,,
\end{align*}
where $\alpha = \tfrac{856 399}{1 572 864}$ and the function $\psi\colon [0,2\pi)\to \R$, for every $\theta\in [0,2\pi)$, is defined by
\begin{align*}
\psi(\theta) &= 
\tfrac{1}{\alpha +1}\cos(\tfrac{3\pi}{2}\alpha) \sin((\alpha +1 )\theta) - \cos((\alpha+1)\theta)
\\&\quad+ \tfrac{1}{1-\alpha} \cos(\tfrac{3\pi}{2}\alpha)\sin((\alpha-1)\theta)
+ \cos((\alpha-1)\theta)\,.
\end{align*}
The Dirichlet boundary datum $\uDirichlet\in \BH^1(\Omega)$ is defined so that $\textup{div}\,\uDirichlet=0$ a.e.\ in $\Omega$ and $\uDirichlet=\bu_{\textup{orig}}$ q.e.\ in $\Gamma_D$. Note that the original velocity vector field $\bu_{\textup{orig}}\in\BH^1(\Omega)$ has a singularity at the re-entrant corner  (\textit{i.e.}, at the origin $(0,0)$, \textit{cf}.\ Figure \ref{fig:stokes_rates}(\textit{right})) of the $L$-shaped domain $\Omega$ limiting its maximal regularity to $\bu_{\textup{orig}}\in\BH^{\frac{3}{2}}(\Omega)$.

In Figure \ref{fig:meshes_stokes}, the initial triangulation $\mathcal{T}_{h_0}$ consisting of 114 elements and final triangulation $\mathcal{T}_{h_{12}}$, which is obtained by applying the local mesh-refinement algorithm (\textit{cf}. {Algorithm} \ref{alg:adaptive}), are {depicted}. As expected, the local  mesh-refinement algorithm (\textit{cf}.\ Algorithm \ref{alg:adaptive}) iteratively refines towards the re-entrant corner (\textit{i.e.}, the origin $(0,0)$, \textit{cf}.\ Figure \ref{fig:stokes_rates}(\textit{right})). 
In Figure \ref{fig:stokes_rates}(\textit{left}), for the 
primal errors $\rho_{I}(\bu_{h_k})$, $k=0,\ldots,k_{\max}$, and the dual errors $\rho_{-D}(\BT_{h_k})$, $k=0,\ldots,k_{\max}$, we report the optimal convergence rate with respect to the number of degrees of freedom $\mathcal{O}(\mathtt{num\_dof}_{\smash{k}}^{-\smash{\frac{1}{2}}})=\mathcal{O}(h_k)$,  $k=1,\ldots,k_{\max}\coloneqq 12$, when using adaptive mesh-refinement (\textit{i.e.}, $\theta=\frac{1}{2}$ in Algorithm \ref{alg:adaptive}) and
$\mathcal{O}(\mathtt{num\_dof}_{\smash{k}}^{-\smash{\frac{1}{4}}})=\mathcal{O}(h_k^{-\smash{\frac{1}{2}}})$,  $k=1,\ldots,k_{\max}\coloneqq 4$, when using uniform mesh-refinement (\textit{i.e.}, $\theta=0$ in Algorithm \ref{alg:adaptive}).

\begin{figure}


\tikzset {_ty39njco4/.code = {\pgfsetadditionalshadetransform{ \pgftransformshift{\pgfpoint{0 bp } { 0 bp }  }  \pgftransformscale{1 }  }}}
\pgfdeclareradialshading{_rjfcs6cqi}{\pgfpoint{0bp}{0bp}}{rgb(0bp)=(1,1,1);
rgb(2.142857142857143bp)=(1,1,1);
rgb(14.285714285714285bp)=(0.89,0.89,0.89);
rgb(14.285714285714285bp)=(0.89,0.89,0.89);
rgb(18.482142857142858bp)=(0.86,0.86,0.86);
rgb(25bp)=(0.82,0.82,0.82);
rgb(400bp)=(0.82,0.82,0.82)}


\tikzset{
pattern size/.store in=\mcSize, 
pattern size = 5pt,
pattern thickness/.store in=\mcThickness, 
pattern thickness = 0.3pt,
pattern radius/.store in=\mcRadius, 
pattern radius = 1pt}
\makeatletter
\pgfutil@ifundefined{pgf@pattern@name@_ny8wpwm21}{
\pgfdeclarepatternformonly[\mcThickness,\mcSize]{_ny8wpwm21}
{\pgfqpoint{0pt}{-\mcThickness}}
{\pgfpoint{\mcSize}{\mcSize}}
{\pgfpoint{\mcSize}{\mcSize}}
{
\pgfsetcolor{\tikz@pattern@color}
\pgfsetlinewidth{\mcThickness}
\pgfpathmoveto{\pgfqpoint{0pt}{\mcSize}}
\pgfpathlineto{\pgfpoint{\mcSize+\mcThickness}{-\mcThickness}}
\pgfusepath{stroke}
}}
\makeatother


\tikzset{
pattern size/.store in=\mcSize, 
pattern size = 5pt,
pattern thickness/.store in=\mcThickness, 
pattern thickness = 0.3pt,
pattern radius/.store in=\mcRadius, 
pattern radius = 1pt}
\makeatletter
\pgfutil@ifundefined{pgf@pattern@name@_4xaet1p9u}{
\pgfdeclarepatternformonly[\mcThickness,\mcSize]{_4xaet1p9u}
{\pgfqpoint{0pt}{-\mcThickness}}
{\pgfpoint{\mcSize}{\mcSize}}
{\pgfpoint{\mcSize}{\mcSize}}
{
\pgfsetcolor{\tikz@pattern@color}
\pgfsetlinewidth{\mcThickness}
\pgfpathmoveto{\pgfqpoint{0pt}{\mcSize}}
\pgfpathlineto{\pgfpoint{\mcSize+\mcThickness}{-\mcThickness}}
\pgfusepath{stroke}
}}
\makeatother


\tikzset{
pattern size/.store in=\mcSize, 
pattern size = 5pt,
pattern thickness/.store in=\mcThickness, 
pattern thickness = 0.3pt,
pattern radius/.store in=\mcRadius, 
pattern radius = 1pt}
\makeatletter
\pgfutil@ifundefined{pgf@pattern@name@_z0we6ibg5}{
\pgfdeclarepatternformonly[\mcThickness,\mcSize]{_z0we6ibg5}
{\pgfqpoint{0pt}{-\mcThickness}}
{\pgfpoint{\mcSize}{\mcSize}}
{\pgfpoint{\mcSize}{\mcSize}}
{
\pgfsetcolor{\tikz@pattern@color}
\pgfsetlinewidth{\mcThickness}
\pgfpathmoveto{\pgfqpoint{0pt}{\mcSize}}
\pgfpathlineto{\pgfpoint{\mcSize+\mcThickness}{-\mcThickness}}
\pgfusepath{stroke}
}}
\makeatother


\tikzset{
pattern size/.store in=\mcSize, 
pattern size = 5pt,
pattern thickness/.store in=\mcThickness, 
pattern thickness = 0.3pt,
pattern radius/.store in=\mcRadius, 
pattern radius = 1pt}
\makeatletter
\pgfutil@ifundefined{pgf@pattern@name@_2o6mjvbvk}{
\pgfdeclarepatternformonly[\mcThickness,\mcSize]{_2o6mjvbvk}
{\pgfqpoint{0pt}{-\mcThickness}}
{\pgfpoint{\mcSize}{\mcSize}}
{\pgfpoint{\mcSize}{\mcSize}}
{
\pgfsetcolor{\tikz@pattern@color}
\pgfsetlinewidth{\mcThickness}
\pgfpathmoveto{\pgfqpoint{0pt}{\mcSize}}
\pgfpathlineto{\pgfpoint{\mcSize+\mcThickness}{-\mcThickness}}
\pgfusepath{stroke}
}}
\makeatother


\tikzset{
pattern size/.store in=\mcSize, 
pattern size = 5pt,
pattern thickness/.store in=\mcThickness, 
pattern thickness = 0.3pt,
pattern radius/.store in=\mcRadius, 
pattern radius = 1pt}
\makeatletter
\pgfutil@ifundefined{pgf@pattern@name@_he8jp9jma}{
\pgfdeclarepatternformonly[\mcThickness,\mcSize]{_he8jp9jma}
{\pgfqpoint{0pt}{-\mcThickness}}
{\pgfpoint{\mcSize}{\mcSize}}
{\pgfpoint{\mcSize}{\mcSize}}
{
\pgfsetcolor{\tikz@pattern@color}
\pgfsetlinewidth{\mcThickness}
\pgfpathmoveto{\pgfqpoint{0pt}{\mcSize}}
\pgfpathlineto{\pgfpoint{\mcSize+\mcThickness}{-\mcThickness}}
\pgfusepath{stroke}
}}
\makeatother


\tikzset{
pattern size/.store in=\mcSize, 
pattern size = 5pt,
pattern thickness/.store in=\mcThickness, 
pattern thickness = 0.3pt,
pattern radius/.store in=\mcRadius, 
pattern radius = 1pt}
\makeatletter
\pgfutil@ifundefined{pgf@pattern@name@_tdvyhmk1f}{
\pgfdeclarepatternformonly[\mcThickness,\mcSize]{_tdvyhmk1f}
{\pgfqpoint{0pt}{-\mcThickness}}
{\pgfpoint{\mcSize}{\mcSize}}
{\pgfpoint{\mcSize}{\mcSize}}
{
\pgfsetcolor{\tikz@pattern@color}
\pgfsetlinewidth{\mcThickness}
\pgfpathmoveto{\pgfqpoint{0pt}{\mcSize}}
\pgfpathlineto{\pgfpoint{\mcSize+\mcThickness}{-\mcThickness}}
\pgfusepath{stroke}
}}
\makeatother
\tikzset{every picture/.style={line width=0.75pt}} 

\centering
\subfloat[C][{\centering Logarithmic plots of the primal errors $\rho_I(\bu_{h_k})$, $k=1,\ldots,k_{\max}$, and the dual errors $\rho_{-D}(\BT_{h_k})$, $k=1,\ldots,k_{\max}$,
for $k_{\max}=12$, when using adaptive mesh-refinement (\textit{i.e.}, $\theta=\frac{1}{2}$ in Algorithm \ref{alg:adaptive}), and for $k_{\max}=4$, when using uniform mesh-refinement (\textit{i.e.}, $\theta=0$ in Algorithm \ref{alg:adaptive}).}]{{%
\hspace{-2.5mm}\includegraphics[width=0.575\textwidth]{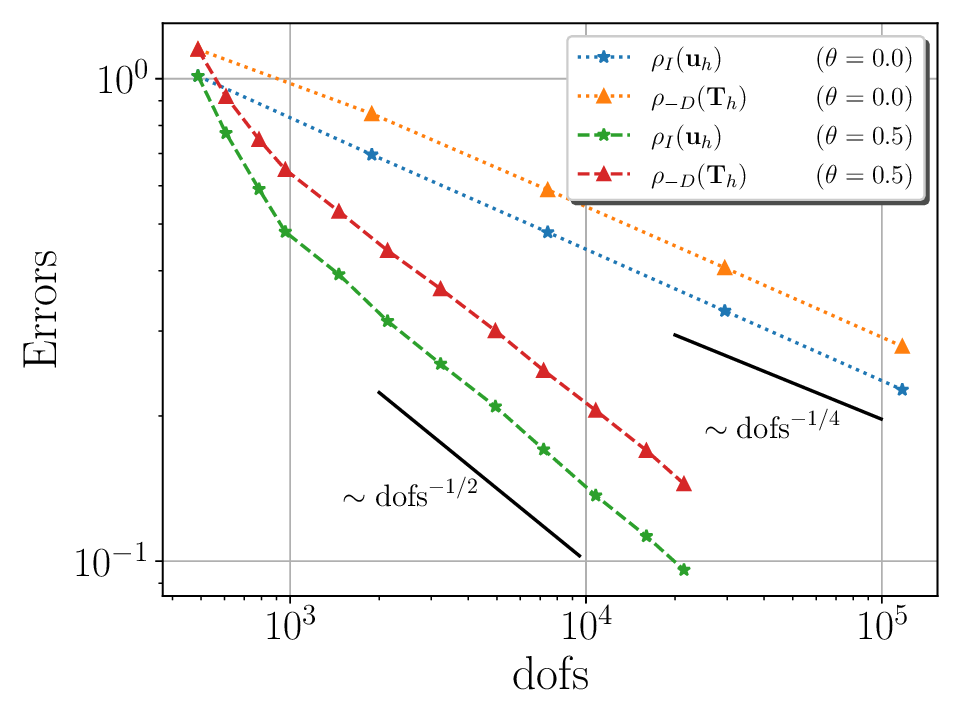}%
}}%
\subfloat[D][{\centering $L$-shaped domain $\Omega$ (\textcolor{gray}{gray}), Dirichlet boundary $\Gamma_D$ (\textcolor{shamrockgreen}{green}), and location of singularity (\textcolor{red!66!black}{red star}).}]{
\begin{tikzpicture}[x=0.86pt,y=0.86pt,yscale=-1,xscale=1]

\path  [shading=_rjfcs6cqi,_ty39njco4] (198.75,42.69) -- (198.75,102.19) -- (138.5,102.19) -- (138.5,162.44) -- (79,162.44) -- (79,42.69) -- (198.75,42.69) -- cycle ; 
\draw   (198.75,42.69) -- (198.75,102.19) -- (138.5,102.19) -- (138.5,162.44) -- (79,162.44) -- (79,42.69) -- (198.75,42.69) -- cycle ; 

\draw  [draw opacity=0][pattern=_ny8wpwm21,pattern size=3.75pt,pattern thickness=0.75pt,pattern radius=0pt, pattern color={shamrockgreen}] (198.75,42.69) -- (209,42.69) -- (209,102.25) -- (198.75,102.25) -- cycle ;
\draw  [draw opacity=0][pattern=_4xaet1p9u,pattern size=3.75pt,pattern thickness=0.75pt,pattern radius=0pt, pattern color={shamrockgreen}] (68.75,42.69) -- (79,42.69) -- (79,162.5) -- (68.75,162.5) -- cycle ;
\draw  [draw opacity=0][pattern=_z0we6ibg5,pattern size=3.75pt,pattern thickness=0.75pt,pattern radius=0pt, pattern color={shamrockgreen}] (79,32.25) -- (198.75,32.25) -- (198.75,42.69) -- (79,42.69) -- cycle ;
\draw  [draw opacity=0][pattern=_2o6mjvbvk,pattern size=3.75pt,pattern thickness=0.75pt,pattern radius=0pt, pattern color={shamrockgreen}] (79,162.44) -- (138.5,162.44) -- (138.5,172.89) -- (79,172.89) -- cycle ;
\draw [color={rgb, 255:red, 206; green, 206; blue, 206 }  ,draw opacity=1 ] [dash pattern={on 0.75pt off 0.75pt on 0.75pt off 0.75pt}]  (39,162.5) -- (79,162.5) ;
\draw [color={rgb, 255:red, 206; green, 206; blue, 206 }  ,draw opacity=1 ] [dash pattern={on 0.75pt off 0.75pt on 0.75pt off 0.75pt}]  (39,42.69) -- (79,42.69) ;
\draw [color={rgb, 255:red, 206; green, 206; blue, 206 }  ,draw opacity=1 ] [dash pattern={on 0.75pt off 0.75pt on 0.75pt off 0.75pt}]  (79,162.44) -- (79,202.44) ;
\draw [color={rgb, 255:red, 206; green, 206; blue, 206 }  ,draw opacity=1 ] [dash pattern={on 0.75pt off 0.75pt on 0.75pt off 0.75pt}]  (198.75,102.19) -- (198.75,202.44) ;
\draw [color={rgb, 255:red, 128; green, 128; blue, 128 }  ,draw opacity=1 ]   (39,42.69) -- (39,162.5) ;
\draw [shift={(39,162.5)}, rotate = 270] [color={rgb, 255:red, 128; green, 128; blue, 128 }  ,draw opacity=1 ][line width=0.75]    (0,5.59) -- (0,-5.59)   ;
\draw [shift={(39,42.69)}, rotate = 270] [color={rgb, 255:red, 128; green, 128; blue, 128 }  ,draw opacity=1 ][line width=0.75]    (0,5.59) -- (0,-5.59)   ;
\draw [color={rgb, 255:red, 128; green, 128; blue, 128 }  ,draw opacity=1 ]   (79,202.44) -- (198.75,202.44) ;
\draw [shift={(198.75,202.44)}, rotate = 180] [color={rgb, 255:red, 128; green, 128; blue, 128 }  ,draw opacity=1 ][line width=0.75]    (0,5.59) -- (0,-5.59)   ;
\draw [shift={(79,202.44)}, rotate = 180] [color={rgb, 255:red, 128; green, 128; blue, 128 }  ,draw opacity=1 ][line width=0.75]    (0,5.59) -- (0,-5.59)   ;
\draw  [draw opacity=0][pattern=_he8jp9jma,pattern size=3.75pt,pattern thickness=0.75pt,pattern radius=0pt, pattern color={shamrockgreen}] (138.5,102.19) -- (148.75,102.19) -- (148.75,161.75) -- (138.5,161.75) -- cycle ;
\draw  [draw opacity=0][pattern=_tdvyhmk1f,pattern size=3.75pt,pattern thickness=0.75pt,pattern radius=0pt, pattern color={shamrockgreen}] (138.75,102.25) -- (198.75,102.25) -- (198.75,112.69) -- (138.75,112.69) -- cycle ;
\draw   (198.75,42.75) -- (198.75,102.25) -- (138.5,102.25) -- (138.5,162.5) -- (79,162.5) -- (79,42.75) -- (198.75,42.75) -- cycle ;
\draw  [fill={rgb, 255:red, 208; green, 2; blue, 27 }  ,fill opacity=1 ] (138.75,97.25) -- (140.15,100.23) -- (143.27,100.7) -- (141.01,103.02) -- (141.54,106.3) -- (138.75,104.75) -- (135.96,106.3) -- (136.49,103.02) -- (134.23,100.7) -- (137.35,100.23) -- cycle ;
\draw [color={rgb, 255:red, 206; green, 206; blue, 206 }  ,draw opacity=1 ] [dash pattern={on 0.75pt off 0.75pt on 0.75pt off 0.75pt}]  (138.5,162.5) -- (138.5,202.5) ;
\draw [color={rgb, 255:red, 206; green, 206; blue, 206 }  ,draw opacity=1 ] [dash pattern={on 0.75pt off 0.75pt on 0.75pt off 0.75pt}]  (39,102.6) -- (79,102.75) ;
\draw [color={rgb, 255:red, 128; green, 128; blue, 128 }  ,draw opacity=1 ]   (39,102.6) -- (39,162.5) ;
\draw [shift={(39,162.5)}, rotate = 270] [color={rgb, 255:red, 128; green, 128; blue, 128 }  ,draw opacity=1 ][line width=0.75]    (0,5.59) -- (0,-5.59)   ;
\draw [shift={(39,102.6)}, rotate = 270] [color={rgb, 255:red, 128; green, 128; blue, 128 }  ,draw opacity=1 ][line width=0.75]    (0,5.59) -- (0,-5.59)   ;
\draw [color={rgb, 255:red, 128; green, 128; blue, 128 }  ,draw opacity=1 ]   (79,202.44) -- (138.5,202.5) ;
\draw [shift={(138.5,202.5)}, rotate = 180.05] [color={rgb, 255:red, 128; green, 128; blue, 128 }  ,draw opacity=1 ][line width=0.75]    (0,5.59) -- (0,-5.59)   ;
\draw [shift={(79,202.44)}, rotate = 180.05] [color={rgb, 255:red, 128; green, 128; blue, 128 }  ,draw opacity=1 ][line width=0.75]    (0,5.59) -- (0,-5.59)   ;

\draw (212.5,65) node [anchor=north west][inner sep=0.75pt]  [color={shamrockgreen}  ,opacity=1 ]  {$\Gamma _{D}$};
\draw (130,17.5) node [anchor=north west][inner sep=0.75pt]  [color={shamrockgreen}  ,opacity=1 ]  {$\Gamma _{D}$};
\draw (102.5,175) node [anchor=north west][inner sep=0.75pt]  [color={shamrockgreen}  ,opacity=1 ]  {$\Gamma _{D}$};
\draw (191,209) node [anchor=north west][inner sep=0.75pt]    {$1.0$};
\draw (70,209) node [anchor=north west][inner sep=0.75pt]    {$0.0$};
\draw (15,157.5) node [anchor=north west][inner sep=0.75pt]    {$0.0$};
\draw (15,37.5) node [anchor=north west][inner sep=0.75pt]    {$1.0$};
\draw (106,69) node [anchor=north west][inner sep=0.75pt]  [font=\Large]  {$\Omega $};
\draw (151.5,130) node [anchor=north west][inner sep=0.75pt]  [color={shamrockgreen}  ,opacity=1 ]  {$\Gamma _{D}$};
\draw (165,115) node [anchor=north west][inner sep=0.75pt]  [color={shamrockgreen}  ,opacity=1 ]  {$\Gamma _{D}$};
\draw (130,209) node [anchor=north west][inner sep=0.75pt]    {$0.5$};
\draw (15,97.5) node [anchor=north west][inner sep=0.75pt]    {$0.5$};
\draw (49,97) node [anchor=north west][inner sep=0.75pt]  [color={shamrockgreen}  ,opacity=1 ]  {$\Gamma _{D}$};
\draw (138,84) node [anchor=north west][inner sep=0.75pt]    {$(0,0)$};

\end{tikzpicture}
}%
\caption{\textit{left}: error decay plot; \textit{right}: computational domain with boundary conditions.}%
\label{fig:stokes_rates}
\end{figure}

\begin{figure}
\centering
\subfloat[C][{\centering Initial triangulation $\mathcal{T}_0$.}]{{%
\includegraphics[width=0.5\textwidth]{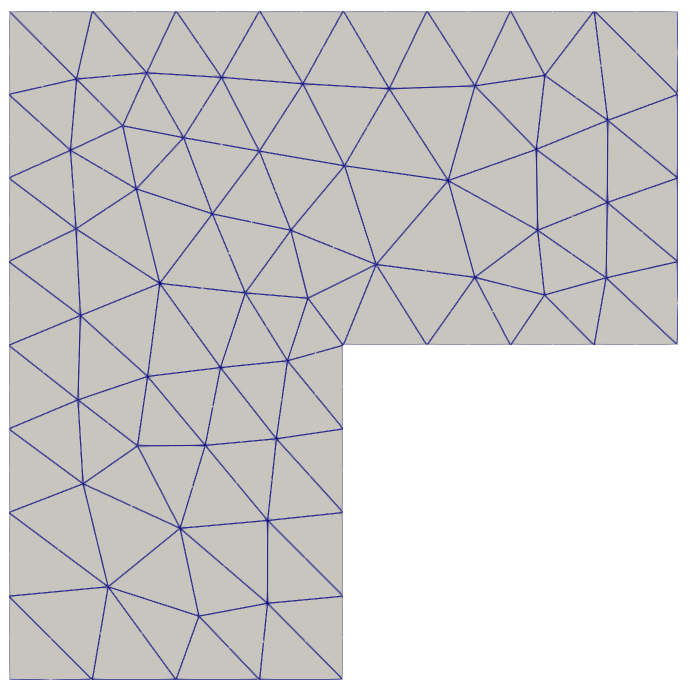}%
}}%
\subfloat[D][{\centering Triangulation $\mathcal{T}_{12} $ after 11 refinements.}]{{%
\includegraphics[width=0.5\textwidth]{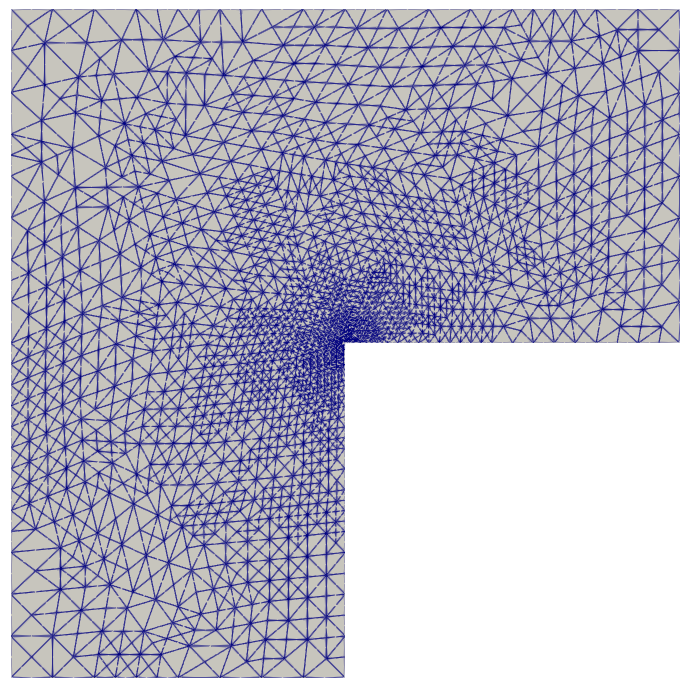}%
}}%
\caption{Initial and locally refined triangulation obtained following Algorithm \ref{alg:adaptive} with refinement parameter $\theta=0.5$ for the Stokes minimisation problem (\textit{cf}.\ Section \ref{sec:stokes}).}%
\label{fig:meshes_stokes}
\end{figure}

\subsection{\emph{A posteriori} error analysis --- Navier--Lam\'e minimisation problem}

\hspace{5mm}In this third numerical experiment, we review the performance of the local refinement indicators \eqref{def:indicator_elasticity} in the local mesh-refinement algorithm (\textit{cf}.\ Algorithm \ref{alg:adaptive}). To this end, we consider Cook's Membrane problem (\textit{cf}.\ \cite{Coo.1974}):
let the computational domain be given via $\Omega\subseteq \R^2$ described in Figure \ref{fig:elasticity_rates_sols}(\textit{right}) with Dirichlet boundary at the left wall, \textit{i.e.}, $\Gamma_D = \{0\}\times (0,0.44)$, 
and Neumann boundaries at the remaining walls, \textit{i.e.}, $\Gamma_N\coloneqq\partial\Omega\setminus \Gamma_D$. On the Dirichlet boundary $\Gamma_D$, we impose homogeneous Dirichlet boundary conditions, \textit{i.e.}, $\widehat{\bu}_{D} = \bzero$ a.e.\ in $\Omega$. On the Neumann boundary, for a given positive parameter $\gamma>0$, we impose that, for every $s=(s_1,s_2)\in \Gamma_N$, there holds
\begin{equation}\label{eq:cook_neumannbc}
(\BT \bn)(s) = 
\begin{cases}
(0,\gamma)^\top & \text{if }s_1=0.48\,, \\
(0,0)^\top & \text{otherwise}\,.
\end{cases} 
\end{equation}
More precisely, in the numerical experiment, we choose $\gamma = 0.01$, the Lam\'e parameters $\mu=1$ and $\lambda=5$, and a trivial load $\load = \bzero^*$ in $\HminusDirichlet$. Note that,
in Cook's Membrane problem (\textit{cf}.\ \cite{Coo.1974}), the original displacement field $\bu_{\textup{orig}}\in \BH^1(\Omega)$ is expected to have a singularity at the top left corner (\textit{i.e.}, at $(0,0.44)$, \textit{cf}.\ Figure \ref{fig:elasticity_rates_sols}(\textit{right})) of the computational domain $\Omega$ limiting its maximal regularity  to $\bu_{\textup{orig}}\in \BH^{\frac{5}{3}}(\Omega)$.

In Figure \ref{fig:meshes_elasticity}(\textit{left}/\textit{middle}), the initial triangulation $\mathcal{T}_{h_0}$ consisting of 119 elements and intermediate triangulation $\mathcal{T}_{h_7}$, which is obtained by applying seven times the local mesh-refinement algorithm (\textit{cf}.\ Algorithm \ref{alg:adaptive})  are depicted. As expected, the local  mesh-refinement algorithm (\textit{cf}.\ Algorithm \ref{alg:adaptive}) iteratively refines towards the top left corner (\textit{i.e.}, at $(0,0.44)$, \textit{cf}.\ Figure \ref{fig:elasticity_rates_sols}(\textit{right})).\linebreak 
Moreover, in Figure \ref{fig:elasticity_rates_sols}(\textit{left}), 
for the extended primal-dual gap estimator values $\overline{\eta}_{\mathrm{gap}}(\bu_{h_k},\bsigma_{h_k}^*)$, $k=0,\ldots,k_{\textup{max}}$,  (\textit{cf}.\ \eqref{eq:elasticity_gap}), we report the optimal convergence rate with respect to the total number of degrees of freedom $\mathcal{O}(\mathtt{num\_dof}_{\smash{k}}^{-\smash{\frac{1}{2}}})=\mathcal{O}(h_k)$,  $k=1,\ldots,k_{\max}\coloneqq 13$, when using adaptive mesh-refinement (\textit{i.e.}, $\theta=\frac{1}{2}$ in Algorithm \ref{alg:adaptive}) and
$\mathcal{O}(\mathtt{num\_dof}_{\smash{k}}^{-\smash{\frac{1}{3}}})=\mathcal{O}(h_k^{-\smash{\frac{2}{3}}})$,  $k=1,\ldots,k_{\max}\coloneqq 5$, when using uniform mesh-refinement (\textit{i.e.}, $\theta=0$ in Algorithm \ref{alg:adaptive}).

\begin{figure}[H]

\tikzset{
pattern size/.store in=\mcSize, 
pattern size = 5pt,
pattern thickness/.store in=\mcThickness, 
pattern thickness = 0.3pt,
pattern radius/.store in=\mcRadius, 
pattern radius = 1pt}
\makeatletter
\pgfutil@ifundefined{pgf@pattern@name@_ptwc1udgt}{
\pgfdeclarepatternformonly[\mcThickness,\mcSize]{_ptwc1udgt}
{\pgfqpoint{0pt}{-\mcThickness}}
{\pgfpoint{\mcSize}{\mcSize}}
{\pgfpoint{\mcSize}{\mcSize}}
{
\pgfsetcolor{\tikz@pattern@color}
\pgfsetlinewidth{\mcThickness}
\pgfpathmoveto{\pgfqpoint{0pt}{\mcSize}}
\pgfpathlineto{\pgfpoint{\mcSize+\mcThickness}{-\mcThickness}}
\pgfusepath{stroke}
}}
\makeatother


\tikzset{
pattern size/.store in=\mcSize, 
pattern size = 5pt,
pattern thickness/.store in=\mcThickness, 
pattern thickness = 0.3pt,
pattern radius/.store in=\mcRadius, 
pattern radius = 1pt}
\makeatletter
\pgfutil@ifundefined{pgf@pattern@name@_qe634pa47}{
\pgfdeclarepatternformonly[\mcThickness,\mcSize]{_qe634pa47}
{\pgfqpoint{0pt}{-\mcThickness}}
{\pgfpoint{\mcSize}{\mcSize}}
{\pgfpoint{\mcSize}{\mcSize}}
{
\pgfsetcolor{\tikz@pattern@color}
\pgfsetlinewidth{\mcThickness}
\pgfpathmoveto{\pgfqpoint{0pt}{\mcSize}}
\pgfpathlineto{\pgfpoint{\mcSize+\mcThickness}{-\mcThickness}}
\pgfusepath{stroke}
}}
\makeatother


\tikzset{
pattern size/.store in=\mcSize, 
pattern size = 5pt,
pattern thickness/.store in=\mcThickness, 
pattern thickness = 0.3pt,
pattern radius/.store in=\mcRadius, 
pattern radius = 1pt}
\makeatletter
\pgfutil@ifundefined{pgf@pattern@name@_sf5cpxeba}{
\pgfdeclarepatternformonly[\mcThickness,\mcSize]{_sf5cpxeba}
{\pgfqpoint{0pt}{-\mcThickness}}
{\pgfpoint{\mcSize}{\mcSize}}
{\pgfpoint{\mcSize}{\mcSize}}
{
\pgfsetcolor{\tikz@pattern@color}
\pgfsetlinewidth{\mcThickness}
\pgfpathmoveto{\pgfqpoint{0pt}{\mcSize}}
\pgfpathlineto{\pgfpoint{\mcSize+\mcThickness}{-\mcThickness}}
\pgfusepath{stroke}
}}
\makeatother


\tikzset{
pattern size/.store in=\mcSize, 
pattern size = 5pt,
pattern thickness/.store in=\mcThickness, 
pattern thickness = 0.3pt,
pattern radius/.store in=\mcRadius, 
pattern radius = 1pt}
\makeatletter
\pgfutil@ifundefined{pgf@pattern@name@_owbrjbr8j}{
\pgfdeclarepatternformonly[\mcThickness,\mcSize]{_owbrjbr8j}
{\pgfqpoint{0pt}{-\mcThickness}}
{\pgfpoint{\mcSize}{\mcSize}}
{\pgfpoint{\mcSize}{\mcSize}}
{
\pgfsetcolor{\tikz@pattern@color}
\pgfsetlinewidth{\mcThickness}
\pgfpathmoveto{\pgfqpoint{0pt}{\mcSize}}
\pgfpathlineto{\pgfpoint{\mcSize+\mcThickness}{-\mcThickness}}
\pgfusepath{stroke}
}}
\makeatother


\tikzset {_jcs3jdu1f/.code = {\pgfsetadditionalshadetransform{ \pgftransformshift{\pgfpoint{0 bp } { 0 bp }  }  \pgftransformscale{1 }  }}}
\pgfdeclareradialshading{_7lcynulhw}{\pgfpoint{0bp}{0bp}}{rgb(0bp)=(0.89,0.89,0.89);
rgb(0bp)=(0.89,0.89,0.89);
rgb(0bp)=(1,1,1);
rgb(19.910714285714285bp)=(0.86,0.86,0.86);
rgb(25bp)=(0.82,0.82,0.82);
rgb(400bp)=(0.82,0.82,0.82)}
\tikzset{every picture/.style={line width=0.75pt}} 

\centering
\hspace*{-3.5mm}\subfloat[C][{\centering\hspace{-0.5mm}Logarithmic \hspace{-0.25mm}plots \hspace{-0.25mm}of \hspace{-0.25mm}the \hspace{-0.25mm}extended \hspace{-0.25mm}primal-dual \hspace{-0.25mm}gap \hspace{-0.25mm}estimator  \linebreak $\overline{\eta}_{\textup{gap}}(\bu_{h_k},\bsigma_{h_k}^*)$, $k=1,\ldots,k_{\max}$,  (\textit{cf}.\ \eqref{eq:elasticity_gap}),
for $k_{\max}=13$, when using adaptive mesh-refinement (\textit{i.e.}, $\theta=\frac{1}{2}$ in Algorithm \ref{alg:adaptive}), and for $k_{\max}=5$, when using uniform mesh-refinement (\textit{i.e.}, $\theta=0$ in Algorithm \ref{alg:adaptive})}]{{%
\includegraphics[width=0.575\textwidth]{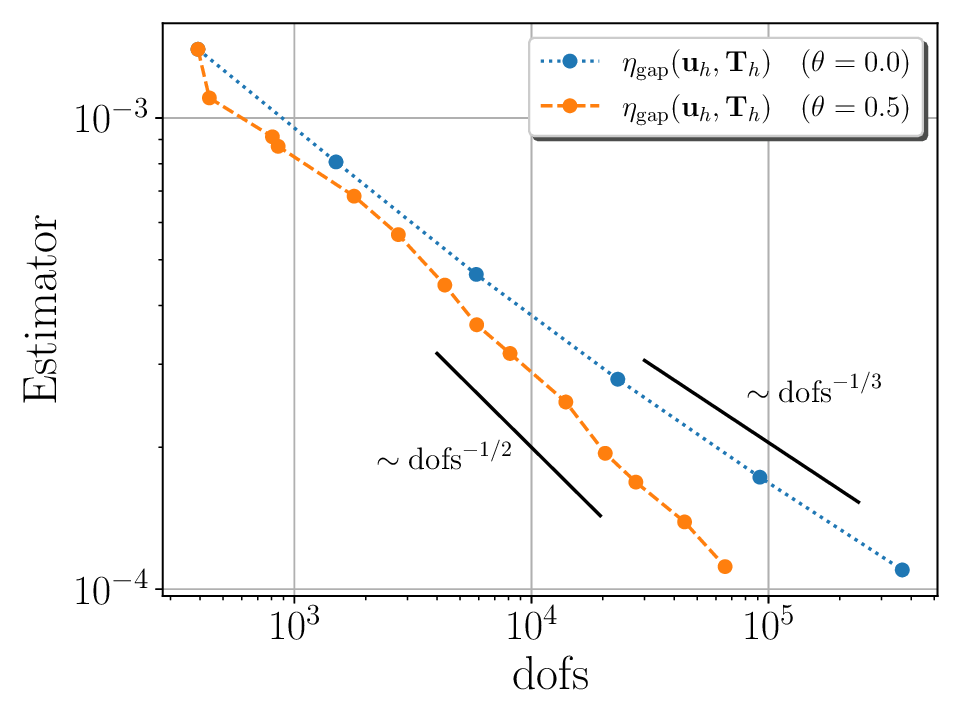}%
}}%
\subfloat[D][{\centering Domain $\Omega$ (\textcolor{gray}{gray}), Dirichlet boundary $\Gamma_D$ (\textcolor{shamrockgreen}{green}), Neumann boundary $\Gamma_N$ (\textcolor{denim}{blue}), and location of singularity (\textcolor{red!66!black}{red star}) in Cook's Membrane problem (\textit{cf}.\ \cite{Coo.1974}).}]{
\begin{tikzpicture}[x=1.1pt,y=1.1pt,yscale=-1,xscale=1]

\draw  [draw opacity=0][pattern=_ptwc1udgt,pattern size=3.75pt,pattern thickness=0.75pt,pattern radius=0pt, pattern color={denim}] (165.86,13.98) -- (165.58,39.05) -- (159.11,38.96) -- (159.39,13.89) -- cycle ;
\draw  [draw opacity=0][pattern=_qe634pa47,pattern size=3.75pt,pattern thickness=0.75pt,pattern radius=0pt, pattern color={denim}] (163.67,43.84) -- (71.93,136.61) -- (67.51,131.48) -- (159.26,38.71) -- cycle ;
\draw  [draw opacity=0][pattern=_sf5cpxeba,pattern size=3.75pt,pattern thickness=0.75pt,pattern radius=0pt, pattern color={shamrockgreen}] (68.03,48.79) -- (67.13,130.69) -- (60.66,130.61) -- (61.56,48.7) -- cycle ;
\draw  [draw opacity=0][pattern=_owbrjbr8j,pattern size=3.75pt,pattern thickness=0.75pt,pattern radius=0pt, pattern color={denim}] (65.53,42.44) -- (157.26,7.87) -- (159.38,14.5) -- (67.65,49.07) -- cycle ;
\draw [color={rgb, 255:red, 206; green, 206; blue, 206 }  ,draw opacity=1 ] [dash pattern={on 0.75pt off 0.75pt on 0.75pt off 0.75pt}]  (30.22,48.89) -- (67.65,49.07) ;
\draw [color={rgb, 255:red, 206; green, 206; blue, 206 }  ,draw opacity=1 ] [dash pattern={on 0.75pt off 0.75pt on 0.75pt off 0.75pt}]  (30,130.33) -- (67.43,130.52) ;
\draw [color={rgb, 255:red, 206; green, 206; blue, 206 }  ,draw opacity=1 ] [dash pattern={on 0.75pt off 0.75pt on 0.75pt off 0.75pt}]  (30.31,14) -- (159.39,13.9) ;
\draw [color={rgb, 255:red, 206; green, 206; blue, 206 }  ,draw opacity=1 ] [dash pattern={on 0.75pt off 0.75pt on 0.75pt off 0.75pt}]  (67.43,130.52) -- (67.79,150.51) ;
\draw [color={rgb, 255:red, 206; green, 206; blue, 206 }  ,draw opacity=1 ] [dash pattern={on 0.75pt off 0.75pt on 0.75pt off 0.75pt}]  (159.2,48.63) -- (159.45,149.67) ;
\draw [color={rgb, 255:red, 128; green, 128; blue, 128 }  ,draw opacity=1 ]   (67.79,150.51) -- (159.45,149.67) ;
\draw [shift={(159.45,149.67)}, rotate = 179.48] [color={rgb, 255:red, 128; green, 128; blue, 128 }  ,draw opacity=1 ][line width=0.75]    (0,5.59) -- (0,-5.59)   ;
\draw [shift={(67.79,150.51)}, rotate = 179.48] [color={rgb, 255:red, 128; green, 128; blue, 128 }  ,draw opacity=1 ][line width=0.75]    (0,5.59) -- (0,-5.59)   ;
\draw [color={rgb, 255:red, 128; green, 128; blue, 128 }  ,draw opacity=1 ]   (30,130.33) -- (30.33,14) ;
\draw [shift={(30.33,14)}, rotate = 90.16] [color={rgb, 255:red, 128; green, 128; blue, 128 }  ,draw opacity=1 ][line width=0.75]    (0,5.59) -- (0,-5.59)   ;
\draw [shift={(30,130.33)}, rotate = 90.16] [color={rgb, 255:red, 128; green, 128; blue, 128 }  ,draw opacity=1 ][line width=0.75]    (0,5.59) -- (0,-5.59)   ;
\draw [color={rgb, 255:red, 128; green, 128; blue, 128 }  ,draw opacity=1 ]   (30,130.33) -- (30.24,48.89) ;
\draw [shift={(30.24,48.89)}, rotate = 90.17] [color={rgb, 255:red, 128; green, 128; blue, 128 }  ,draw opacity=1 ][line width=0.75]    (0,5.59) -- (0,-5.59)   ;
\draw [shift={(30,130.33)}, rotate = 90.17] [color={rgb, 255:red, 128; green, 128; blue, 128 }  ,draw opacity=1 ][line width=0.75]    (0,5.59) -- (0,-5.59)   ;
\path  [shading=_7lcynulhw,_jcs3jdu1f] (67.78,48.55) -- (159.39,14.5) -- (159.26,38.71) -- (67.34,130.97) -- cycle ; 
\draw   (67.78,48.55) -- (159.39,14.5) -- (159.26,38.71) -- (67.34,130.97) -- cycle ; 
\draw  [fill={rgb, 255:red, 208; green, 2; blue, 27 }  ,fill opacity=1 ] (68.03,43.79) -- (69.42,46.76) -- (72.54,47.24) -- (70.29,49.56) -- (70.82,52.83) -- (68.03,51.29) -- (65.23,52.83) -- (65.77,49.56) -- (63.51,47.24) -- (66.63,46.76) -- cycle ;

\draw (116.51,93.48) node [anchor=north west][inner sep=0.75pt]  [color={denim}  ,opacity=1 ]  {$\Gamma _{N}$};
\draw (167.5,21) node [anchor=north west][inner sep=0.75pt]  [color={denim}  ,opacity=1 ]  {$\Gamma _{N}$};
\draw (97.5,14) node [anchor=north west][inner sep=0.75pt]  [color={denim}  ,opacity=1 ]  {$\Gamma _{N}$};
\draw (45,83.73) node [anchor=north west][inner sep=0.75pt]  [color={shamrockgreen}  ,opacity=1 ]  {$\Gamma _{D}$};
\draw (150,156) node [anchor=north west][inner sep=0.75pt]    {$0.48$};
\draw (61,156) node [anchor=north west][inner sep=0.75pt]    {$0.0$};
\draw (11,126.5) node [anchor=north west][inner sep=0.75pt]    {$0.0$};
\draw (6.5,45) node [anchor=north west][inner sep=0.75pt]    {$0.44$};
\draw (11,10) node [anchor=north west][inner sep=0.75pt]    {$0.6$};
\draw (102.5,54.5) node [anchor=north west][inner sep=0.75pt] [font=\Large]   {$\Omega $};
\draw (31,35) node [anchor=north west][inner sep=0.75pt]    {$(0,0.44)$};

\end{tikzpicture}
}%
\caption{\textit{left}: error decay plot; \textit{right}: computational domain with boundary conditions.}%
\label{fig:elasticity_rates_sols}
\end{figure}

\begin{figure}[H]
\centering
\subfloat[C][{\centering Initial triangulation $\mathcal{T}_0$.}]{{%
\includegraphics[width=0.335\textwidth]{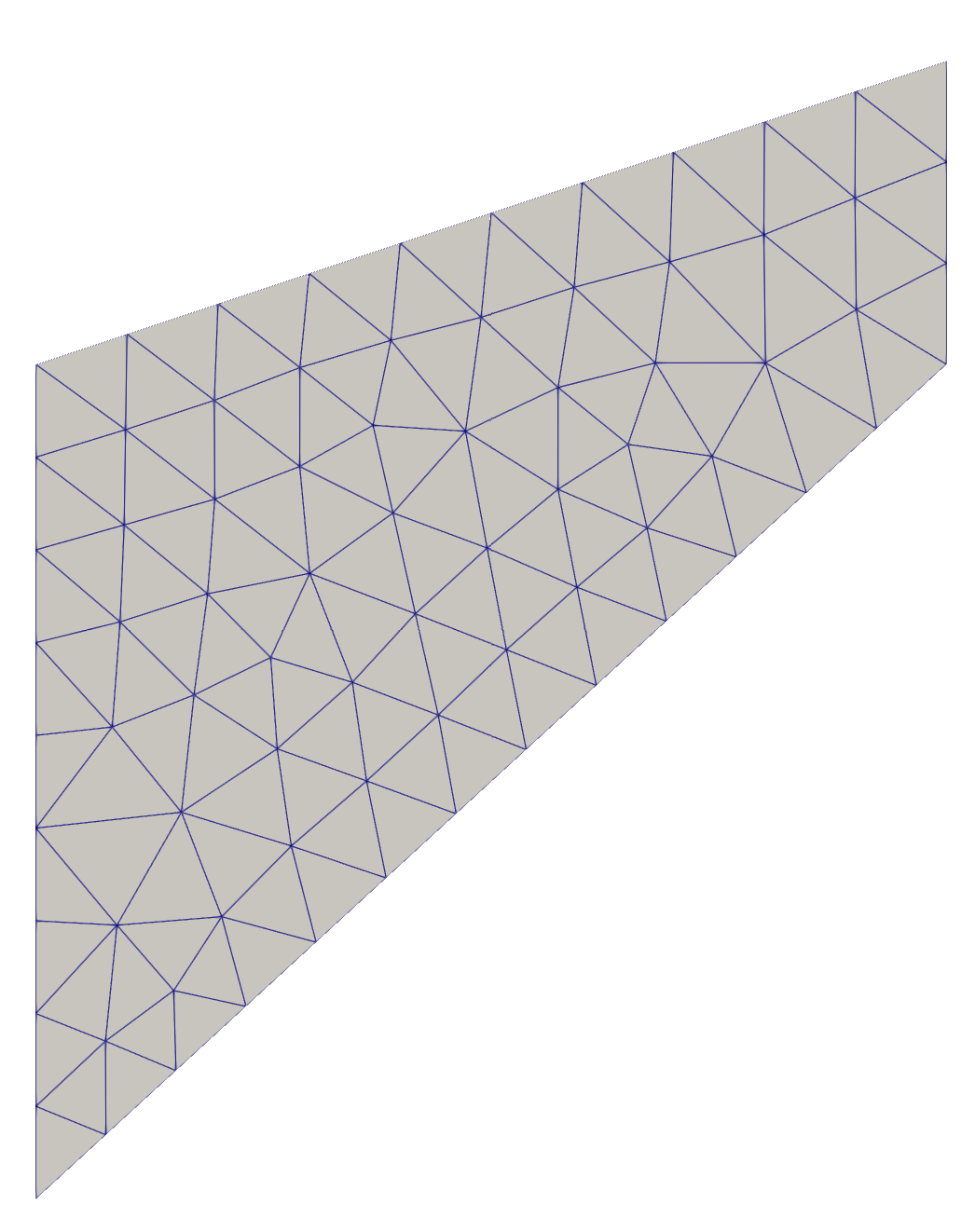}%
}}%
\subfloat[D][{\centering Triangulation $\mathcal{T}_8$ after 7 refinements.}]{{%
\includegraphics[width=0.335\textwidth]{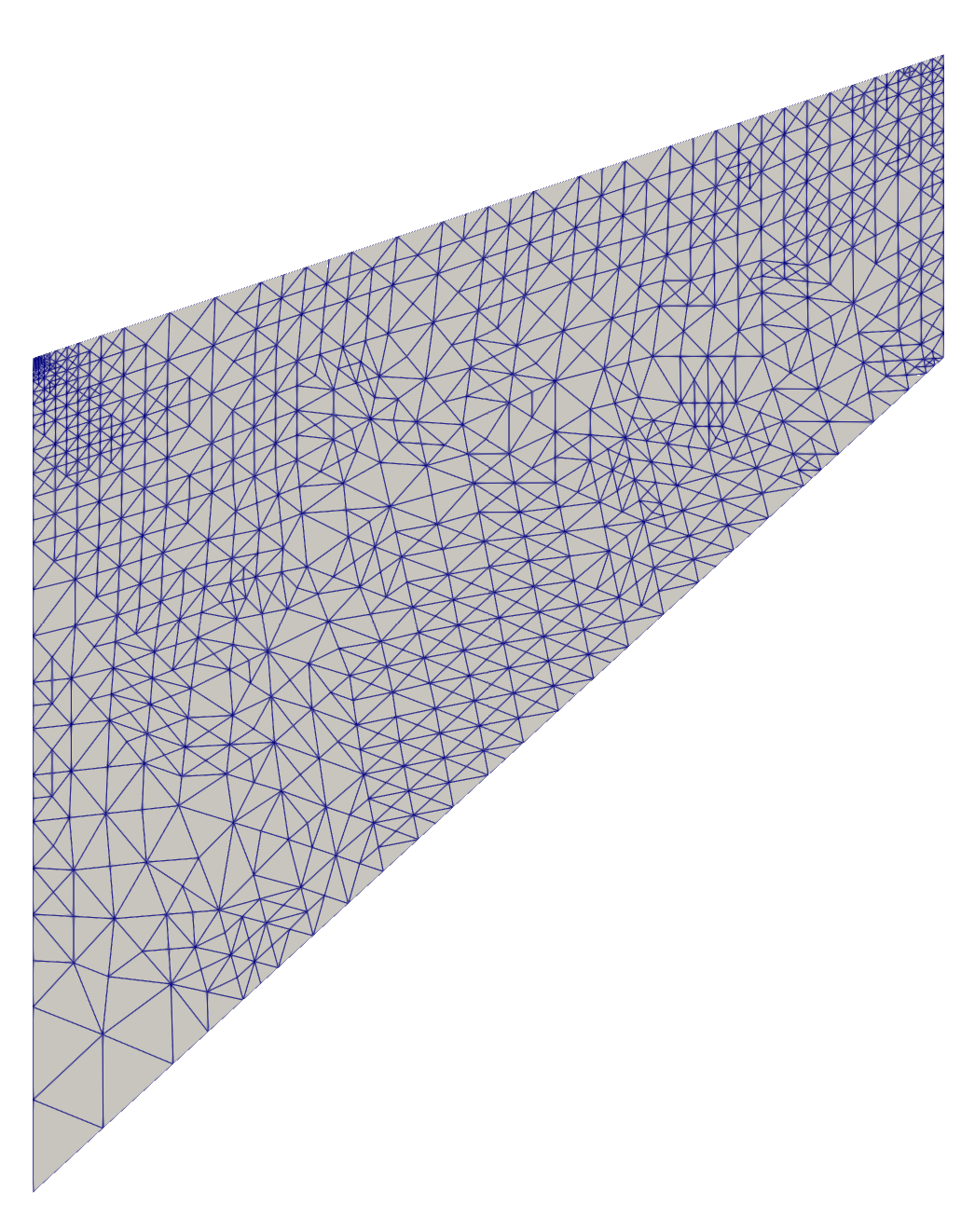}%
}}%
\subfloat[E][{\centering Reference configuration $\Omega$ (\textcolor{denim}{blue}) and deformed configuration $\bu_{h_{13}}(\Omega)$ (\textcolor{shamrockgreen}{green}).}]{{%
\includegraphics[width=0.335\textwidth]{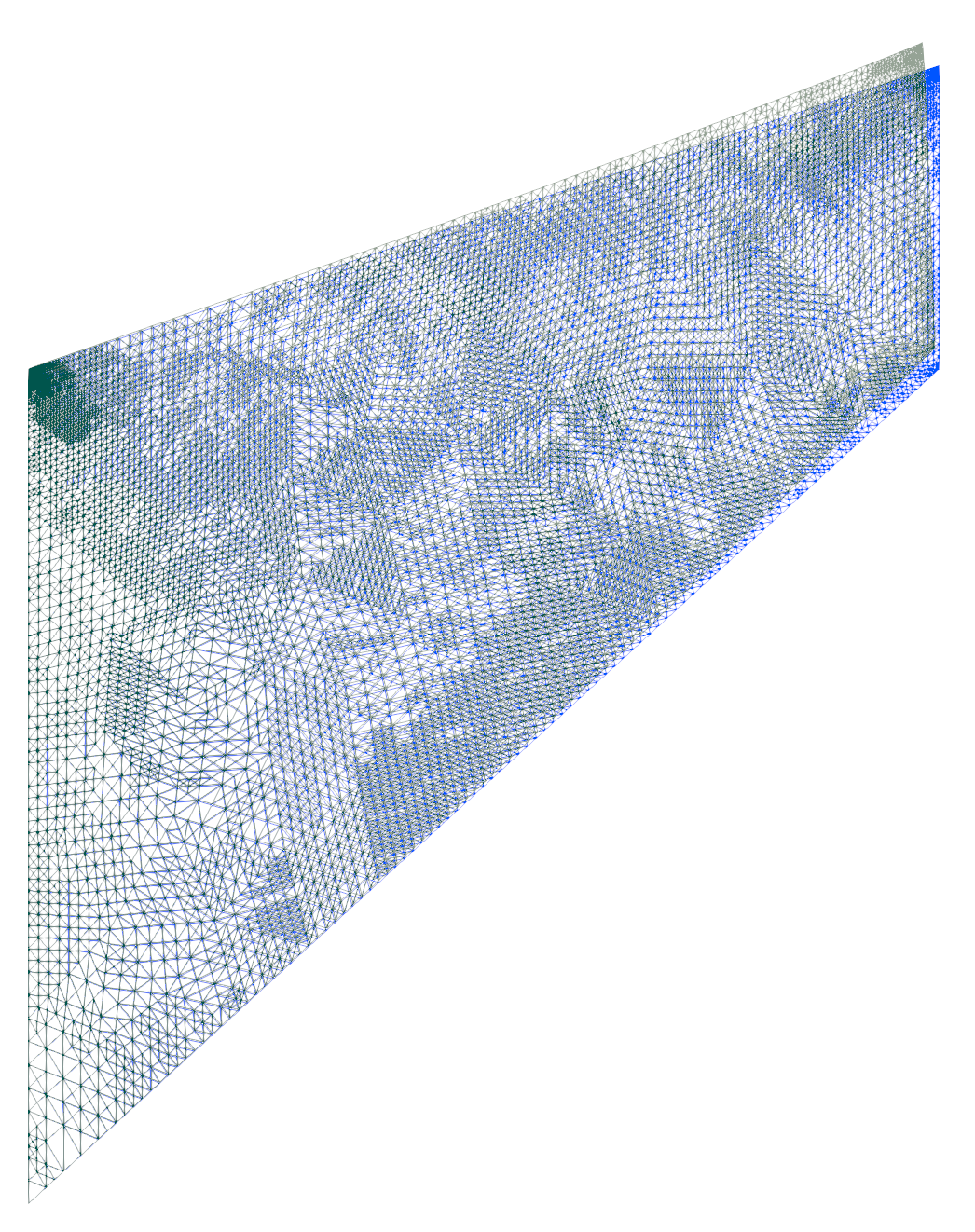}%
}}%
\caption{Initial and locally refined triangulation obtained following Algorithm \ref{alg:adaptive} with refinement parameter $\theta=0.5$ for the Navier--Lam\'e minimisation minimisation problem (\textit{cf}.\ Section \ref{sec:stokes}).}%
\label{fig:meshes_elasticity}
\end{figure}

\begin{acknowledgements}
The authors would like to thank S.\ Bartels for helpful discussions during the development of this work.
\end{acknowledgements}

%
 \section*{Declarations}

 \textbf{Data availability. }
 This study did not generate or analyze any datasets.
 Therefore, data sharing is not applicable to this work.\\

 \noindent
 \textbf{Conflict of interest.}
 The authors declare that they have no conflict of interest.

%
%


%
%

\bibliographystyle{spmpsci}      
\bibliography{literatur}   


\end{document}